\documentclass[11pt]{article}
\usepackage{amsmath, amssymb, theorem, latexsym, epsfig, graphics, eufrak}
\usepackage{subfigure}
\numberwithin{equation}{section}

\theoremstyle{plain}
\theorembodyfont{\itshape}
\newtheorem{theorem}{Theorem}[section]
\newtheorem{proposition}[theorem]{Proposition}
\newtheorem{lemma}[theorem]{Lemma}
\newtheorem{corollary}[theorem]{Corollary}
                                                                                
\theorembodyfont{\rmfamily}
\newtheorem{definition}[theorem]{Definition}

\newtheorem{example}[theorem]{Example}
\newtheorem{remark}[theorem]{Remark}

\newtheorem{problem}[theorem]{Problem}
\newtheorem{conjecture}[theorem]{Conjecture}
\newenvironment{proof}{{\noindent \textbf{Proof}\,\,}}{\hspace*{\fill}$\Box$\medskip}
\def\la{\lambda}
\def\mcd{\mathcal D}
\def\mod{\operatorname{mod}}
\def\cc{\mathbb C}
\def\oc{\overline\cc}
\def\End{\operatorname{End}}
\def\rr{\mathbb R}
\def\ba{K}
\def\bb{R}
\def\zz{\mathbb Z}
\def\La{\Lambda}
\def\cp{\mathbb{CP}}
\def\diag{\operatorname{diag}}
\def\sign{\operatorname{sign}}
\def\mcl{\mathcal L}
\def\rp{\mathbb{RP}}

\def\re{\operatorname{Re}}

\def\mcx{\mathcal X}
\def\mcg{\mathcal G}
\def\mch{\mathcal H}

\def\mcho{\mch^1_{0,\infty}}
\def\mchon{\mch^{1,0}_{0,\infty}}

\def\nn{\mathbb N}
\def\mca{\mathcal A}
\def\mcb{\mathcal B}

\def\mcc{\mathcal C}
\def\mcp{\mathcal P}
\def\wh#1{\widehat#1}

\def\var{\varepsilon}
\def\mcr{\mathcal R}
\def\mca{\mathcal A}
\def\mcb{\mathcal B}

\def\wt#1{\widetilde#1}
\newcommand\ld{\left(\begin{matrix} \frac12 & 0\\ 0 & 0\end{matrix}\right)}

\def\tt{\mathbb T}
\def\gl{\operatorname{GL}}
\def\sl{\operatorname{SL}}
\def\psl{\operatorname{PSL}}
\def\bbf{\mathcal F}
\def\bj{\mathbf J}
\def\bjr{\mathbf J^N(\rr_+)}
\def\bjjr{\bj(\rr_+)}
\def\bbj{\bjr}

\title{On families of constrictions in model of overdamped Josephson junction 
and Painlev\'e 3 equation}
\author{Y.Bibilo\thanks{University of Toronto, Mississauga, Canada},  
A.A.Glutsyuk\thanks{CNRS, France (UMR 5669 (UMPA, ENS de Lyon) and UMI 2615 (ISC J.-V.Poncelet)). 
Email:
aglutsyu@ens-lyon.fr} 
\thanks{HSE University, Moscow, Russia}
\thanks{Kharkevich Institute for Information Transmission Problems, Moscow, Russia}
 \thanks{Supported by  RSF grant 18-41-05003.}}
\begin{document}
\maketitle
\vspace{-0.5cm}
\begin{abstract} The tunneling effect predicted by B.Josephson (Nobel 
Prize, 1973) concerns the   Josephson junction: two superconductors 
separated by a  narrow dielectric. It states existence of a supercurrent through it and equations governing it. 
The {\it overdamped Josephson junction} 
is modeled by  a family of differential equations on 2-torus depending on 3
 parameters:  $B$ (abscissa), $A$ (ordinate), 
$\omega$ (frequency). We study its 
{\it rotation number} $\rho(B,A;\omega)$ 
as a function of  $(B,A)$ with fixed $\omega$. 
The {\it phase-lock areas} are the  level sets $L_r:=\{\rho=r\}$ with non-empty 
interiors; they exist  for $r\in\zz$ (Buchstaber, Karpov, Tertychnyi). 
Each $L_r$ is an infinite chain  of  domains going vertically to infinity  
 and separated by points. Those separating points for which $A\neq0$ are called {\it constrictions}. 
We show that: 1) {\it all the constrictions in $L_r$ lie on the axis $\{ B=\omega r\}$}; 
2) {\it each constriction} is {\it positive}: this means that some its punctured neighborhood on the  
axis $\{ B=\omega r\}$ lies in $\operatorname{Int}(L_r)$. These results confirm experiments by physicists (1970ths) and two mathematical conjectures. 
We first prove deformability of each constriction   to another one, with 
arbitrarily small $\omega$ and the same $\ell:=\frac B\omega$, using equivalent description of model by  linear systems of differential 
equations on $\oc$ (Buchstaber, Karpov, Tertychnyi) and studying their   isomonodromic deformations  
described by   
Painlev\'e 3 equations\footnote{Analytic deformability of each constriction  
 to  constrictions of the same type, $\rho$, $\ell$ and arbitrarily small $\omega$ 
is a joint result of the authors. Non-existence of ghost constrictions with a given $\ell$ for every $\omega$ small enough 
is a result of 
the second author (A.A.Glutsyuk).}. Then  
non-existence of {\it ghost} constrictions (i.e., constrictions either with $\rho\neq\ell=\frac B\omega$, or of non-positive type) with a given $\ell$ for small $\omega$ is proved 
by  slow-fast methods. 
\end{abstract}

\tableofcontents

\section{Introduction}
\subsection{Model of Josephson junction: a brief survey and main results}
The {\it Josephson effect} is a tunnelling effect in superconductivity predicted 
theoretically by B.Josephson in 1962 \cite{josephson} (Nobel Prize in physics, 
1973) and confirmed experimentally by P.W.Anderson and J.M.Rowell in 1963 
\cite{ar}.  It concerns the so-called {\it Josephson junction:} 
a system of two superconductors separated by a very narrow dielectric 
fiber. The Josephson effect is the existence of a supercurrent crossing the junction  
(provided that the dielectric fiber is narrow enough), described by equations 
discovered by Josephson. 

The model of the so-called {\it overdamped Josephson junction},  
see \cite{stewart, mcc,  lev,  schmidt}, \cite[p. 306]{bar}, \cite[pp. 337--340]{lich}, 
\cite[p.193]{lich-rus}, \cite[p. 88]{likh-ulr} is described by the family of nonlinear differential equations
 \begin{equation}\frac{d\phi}{dt}=-\sin \phi + B + A \cos\omega t, \ \omega>0, \ B\geq0.\label{jos}\end{equation}
 Here $\phi$ is the  difference of phases (arguments) of the complex-valued 
 wave functions describing the quantum mechanic 
 states of the two superconductors. Its derivative is 
 equal to the voltage up to known constant factor. 
  
Equations (\ref{jos}) also arise in several models in physics, mechanics and geometry, e.g.,  
in  planimeters, see  \cite{Foote, foott}. 
Here $\omega$ is a fixed constant, and $(B,A)$ are the parameters. The variable 
and parameter changes 
\begin{equation}\tau:=\omega t, \ \theta:=\phi+\frac{\pi}2, \ \ell:=\frac B\omega, \ \mu:=\frac A{2\omega},\label{elmu}\end{equation}
 transform (\ref{jos}) to a 
non-autonomous ordinary differential equation on the two-torus $\mathbb T^2=S^1\times S^1$ with coordinates 
$(\theta,\tau)\in\rr^2\slash2\pi\zz^2$: 
\begin{equation} \frac{d\theta}{d\tau}=\frac{\cos\theta}{\omega} + \ell + 2\mu \cos \tau.\label{jostor}\end{equation}
The graphs of its solutions are the orbits of the vector field 
\begin{equation}\begin{cases} & \dot\theta=\frac{\cos\theta}{\omega} + \ell + 2\mu \cos \tau\\
& \dot \tau=1\end{cases}\label{josvec}\end{equation}
on $\mathbb T^2$. The {\it rotation number} of its flow, see \cite[p. 104]{arn},  is a function $\rho(B,A)$ of parameters\footnote{There is a misprint, 
missing $2\pi$ in the denominator, in analogous formulas in previous papers of the 
second author (A.A.Glutsyuk) with co-authors: \cite[formula (2.2)]{4}, \cite[the formula after (1.16)]{bg2}.}:
$$\rho(B,A;\omega)=\lim_{k\to+\infty}\frac{\theta(2\pi k)}{2\pi k}.$$
Here $\theta(\tau)$ is a general $\rr$-valued solution of the first equation in (\ref{josvec}) 
whose parameter is the initial condition 
for $\tau=0$. Recall that the rotation number is independent on the choice of the initial condition, see \cite[p.104]{arn}. 
 The parameter $B$ is called {\it abscissa,} and $A$ is called the {\it ordinate.} 
 Recall the following well-known definition. 

\begin{definition} \label{defasl} (cf. \cite[definition 1.1]{4}) The {\it $r$-th phase-lock area} is the level set 
$$L_r=L_r(\omega)=\{\rho(B,A)=r\}\subset\rr^2,$$ 
provided that it has a non-empty interior. 
\end{definition}

\begin{remark}{\bf: phase-lock areas and Arnold tongues.} H.Poincar\'e introduced the rotation number of a circle 
diffeomorphism. The rotation number of the flow  of the field (\ref{josvec}) on $\mathbb T^2$ equals 
(modulo $\zz$) the rotation number of the circle diffeomorphism given by its time $2\pi$ flow mapping 
restricted to the cross-section $S^1_{\theta}\times\{0\}$. In Arnold family of circle diffeomorphisms 
$x\mapsto x+b+a\sin x$, $x\in S^1=\rr\slash2\pi\zz$ the behavior of its phase-lock areas for small $a$ demonstrates the tongues 
effect discovered by V.I. Arnold \cite[p. 110]{arn}. That is why the phase-lock areas became 
 ``Arnold tongues'', see  \cite[definition 1.1]{4}. 
\end{remark}

Recall that the rotation number has physics meaning of the mean voltage over a long time interval up to known constant factor. 

Relation of phase-lock effect in model (\ref{jos}) to dynamical systems on torus 
was discovered in \cite{bkt1}. In   physics papers earlier than \cite{bkt1} 
the phase-lock area effect dealt with convergence of differences 
$\phi(t+(n+1)T)-\phi(t+nT)$, 
 where $\phi(t)$ is a solution of (\ref{jos}) and 
$T=\frac{2\pi}{\omega}$. This effect was defined there as a phenomenon of convergence 
of the above differences to $2\pi k$, $k\in\zz$, on an open subset of the 
parameter space. It was observed in \cite{bkt1} that  their convergence is equivalent to 
the statement that the rotation number of the corresponding dynamical system 
(\ref{josvec}) is equal to $k$.

Some figures of the phase-lock areas of family (\ref{josvec}) are presented 
in physics books \cite[p. 339, fig. 11.4]{lich}, \cite[p. 193, fig. 11.4]{lich-rus}, 
\cite[p. 88, fig. 5.2]{likh-ulr}. See also  figures of the phase-lock areas below.

The phase-lock areas of 
 family (\ref{josvec}) were studied by V.M.Buchstaber, O.V.Karpov, S.I.Tertychnyi et al, see \cite{bg}--\cite{bt3}, 
\cite{LSh2009, IRF, krs, RK}, \cite{4}--\cite{gn19}, \cite{tert, tert2} and references therein.  The  following  statements are known results proved mathematically:

1) Phase-lock areas exist only for integer rotation number  values 
({\it quantization effect} observed and proved in \cite{buch2}, later also proved 
in  \cite{LSh2009, IRF}).

2) The boundary of the $r$-th phase-lock area 
 consists of two analytic curves, which are the graphs of two
functions $B=G_{r,\alpha}(A)$, $\alpha=0,\pi$, (see \cite{buch1}; this fact was later explained by A.V.Klimenko via symmetry, see \cite{RK}).

3)  The latter functions have Bessel asymptotics
\begin{equation}\begin{cases} G_{r,0}(s)=r\omega-J_r(-\frac s\omega)+O(\frac{\ln |s|}s) \\ 
G_{r,\pi}(s)=r\omega+J_r(-\frac s\omega)+O(\frac{\ln |s|}s)\end{cases}, \text{ as } s\to\infty\label{bessas}\end{equation}
 (observed and proved on physics level in ~\cite{shap}, see also \cite[p. 338]{lich},
 \cite[section 11.1]{bar}, ~\cite{buch2006}; proved mathematically in ~\cite{RK}).
 
 4) Each phase-lock area is a garland  of infinitely many bounded domains going to infinity in the vertical direction. 
 In this chain each two subsequent domains are separated by one point. This was  
 proved in \cite{RK} using the 
 above statement 3). Those  separation 
 points that lie on the horizontal $B$-axis, namely $A=0$, were calculated explicitly,  and 
 we call them the {\it growth points}, see \cite[corollary 3]{buch1}. The other separation points, which  lie outside the horizontal $B$-axis, are called the  {\it constrictions}. 
 
 5)  For every $r\in\zz$ the $r$-th phase-lock area is symmetric to the $-r$-th one with respect to the vertical 
 $A$-axis.
  
6) Every phase-lock area is symmetric with respect to the horizontal $B$-axis. See Figures 1--3 below.
 
 \begin{figure}[ht]
  \begin{center}
   \epsfig{file=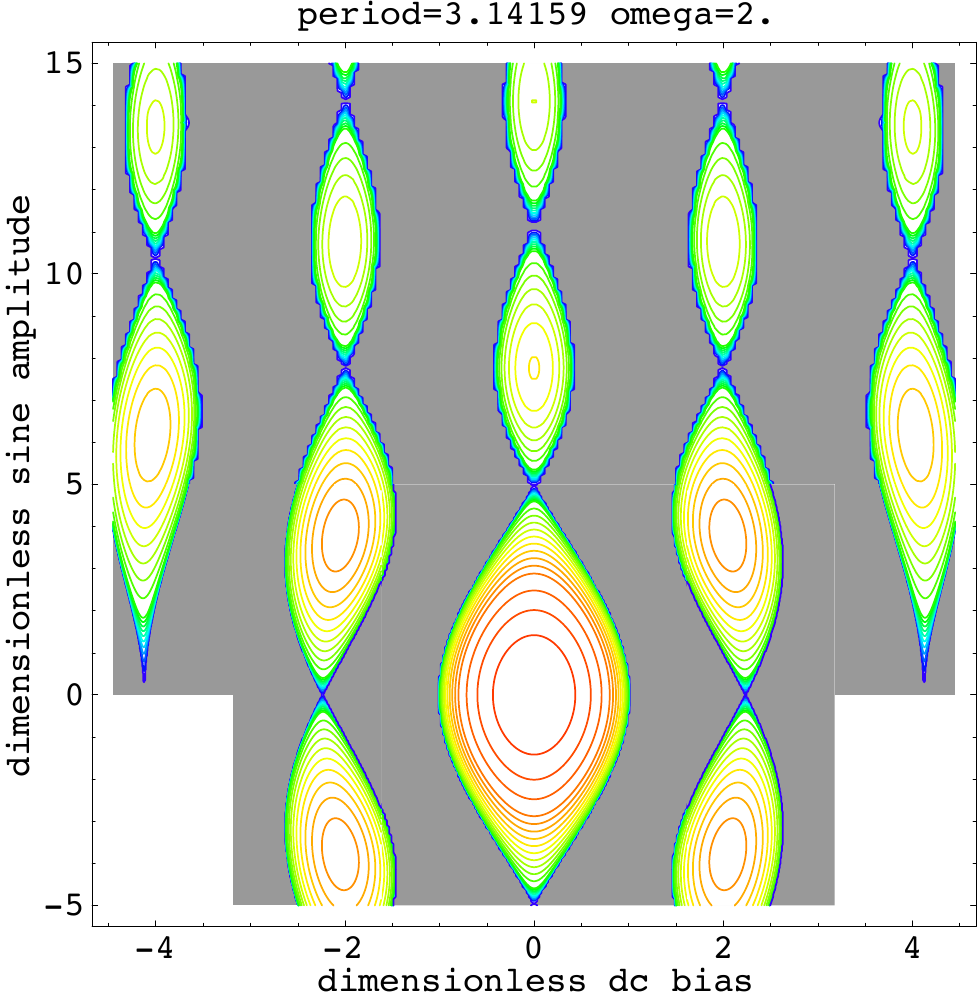}
    \caption{Phase-lock areas and their constrictions for $\omega=2$. The abscissa is $B$, the ordinate is $A$. Figure 
    taken from paper \cite[fig. 1a)]{bg2} with authors' permission.}
  \end{center}
\end{figure} 

 
 \begin{figure}[ht]
  \begin{center}
   \epsfig{file=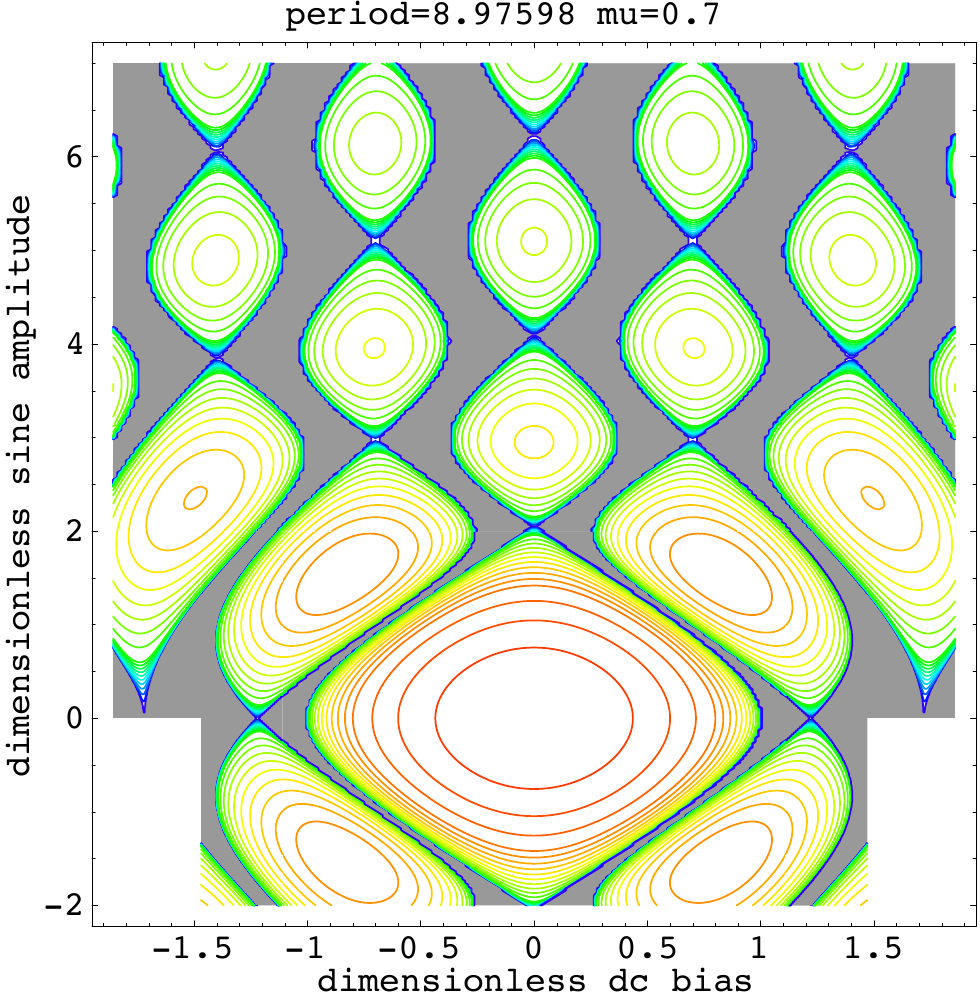}
    \caption{Phase-lock areas and their constrictions  for $\omega=0.7$. Figure taken from papers \cite[fig. 1c)]{bg2}, 
    \cite[p. 331]{bt1} with authors' permission.}
    \end{center}
\end{figure} 


\begin{figure}[ht]
  \begin{center}
   \epsfig{file=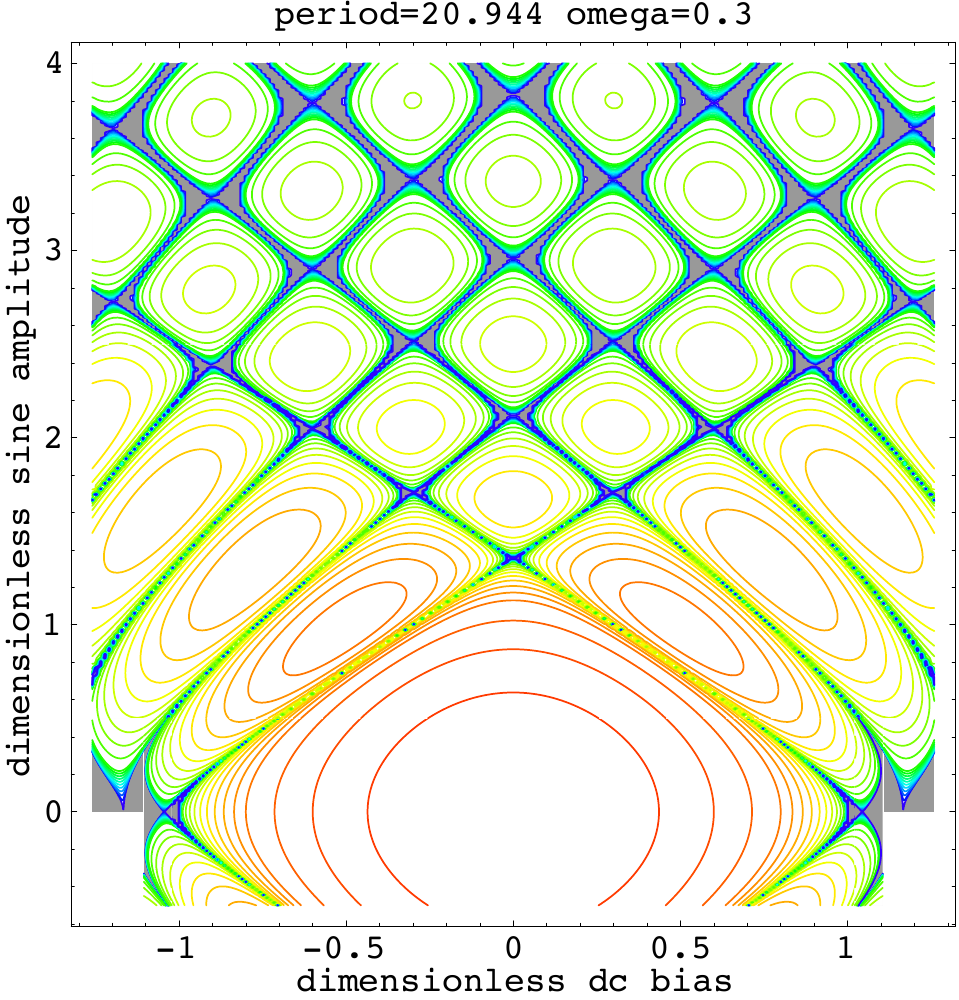}
    \caption{Phase-lock areas and their constrictions  for $\omega=0.3$. Figure 
    taken from paper \cite[fig. 1e)]{bg2} with authors' permission.}
     \end{center}
\end{figure}

\begin{definition} For every $r\in\zz$ and $\omega>0$ we consider the vertical line  
$$\Lambda_r=\{ B=\omega r\}\subset\rr^2_{(B,A)}$$
and we will call it the {\it axis} of the phase-lock area $L_r$. 
\end{definition}

The figures for the phase-lock areas obtained  experimentally are given in the physics books  
on Josephson effect, see \cite[p. 193, fig. 11.4]{lich-rus}, 
\cite[p. 88, fig. 5.2]{likh-ulr},  \cite[p. 339, fig. 11.4]{lich} (which refers  to physics paper 
\cite{ls}). They had shown that in each 
phase-lock area $L_r$ all the constrictions should lie on the same vertical line. 
No mathematical proof was presented there. Numerical illustrations which one can find 
 in the paper \cite{buch2} by V.M.Buchstaber, O.V.Karpov, S.I.Tertychnyi have shown the same effect and that the line containing the constrictions of  the area $L_r$ should coincide with its axis $\La_r$. 
This constriction alignment phenomena was stated as an experimental fact and conjecture 
in \cite[experimental fact A]{4}.

The main results of the paper are the two following theorems. The first theorem 
 confirms the above constriction alignment phenomena. The second one 
 confirms another, positivity property of constrictions that can  be also seen in 
 the figures from physics books mentioned in the above paragraph.  

\begin{theorem} \label{thal} For every $r\in\zz$ and every $\omega>0$ 
all the constrictions of the 
phase-lock area $L_r$ lie in its axis $\La_r$.
\end{theorem}
\begin{remark} \label{remint} 
It was proved in \cite[theorem 1.2]{4} that for every $r\in\zz$ 
 the constrictions in $L_r$ have abscissas $B=\ell\omega$, $\ell\in\zz$, 
 $\ell\equiv r(\mod 2)$, $\ell\in[0,r]$. For further results and discussion of 
  the constriction alignment conjecture see \cite[section 5]{bg2} and \cite{4, g18, gn19}. 
 \end{remark} 

\begin{definition} \cite[p.329]{g18}  A constriction $(B_0,A_0)$ is said to be {\it positive,} if the corresponding 
germ of interior of phase-lock area contains the germ of punctured vertical line interval: 
that is, if  there exists a punctured 
neighborhood $U=U(A_0)\subset\rr$ such that the punctured interval $B_0\times (U\setminus\{A_0\})\subset B_0\times\rr$ lies entirely in the interior of the corresponding  
phase-lock area. A constriction is called {\it negative,} if the above punctured interval can be chosen to lie in the complement to the union of the phase-lock areas. Otherwise it is called {\it neutral}. See Fig. 4.
\end{definition}

\begin{theorem} \label{thpos}\footnote{The main results of the paper (Theorems \ref{thal} and \ref{thpos}) with a sketch of proof were announced in \cite{bibgl}.} All the constrictions are positive.
\end{theorem}
\begin{figure}
\begin{center}
\epsfig{file=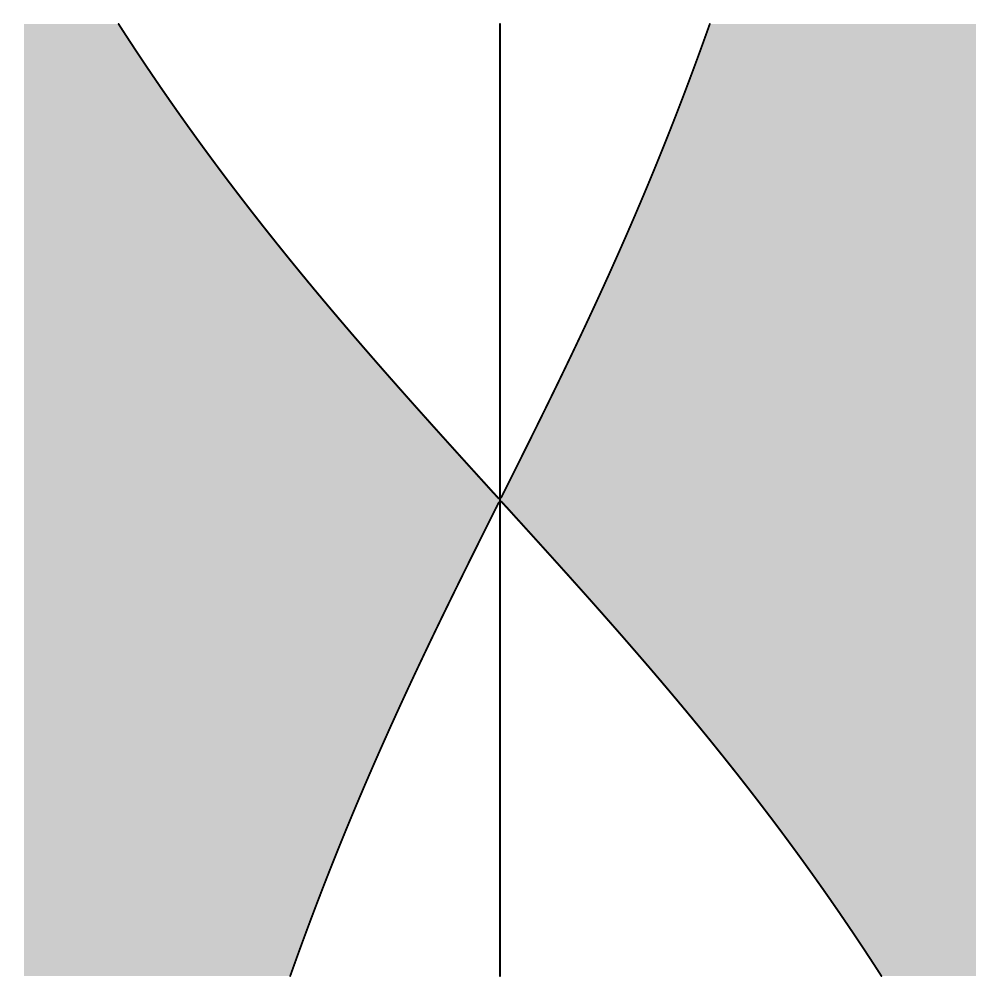, width=0.2\textwidth}
\hspace{2em}
\epsfig{file=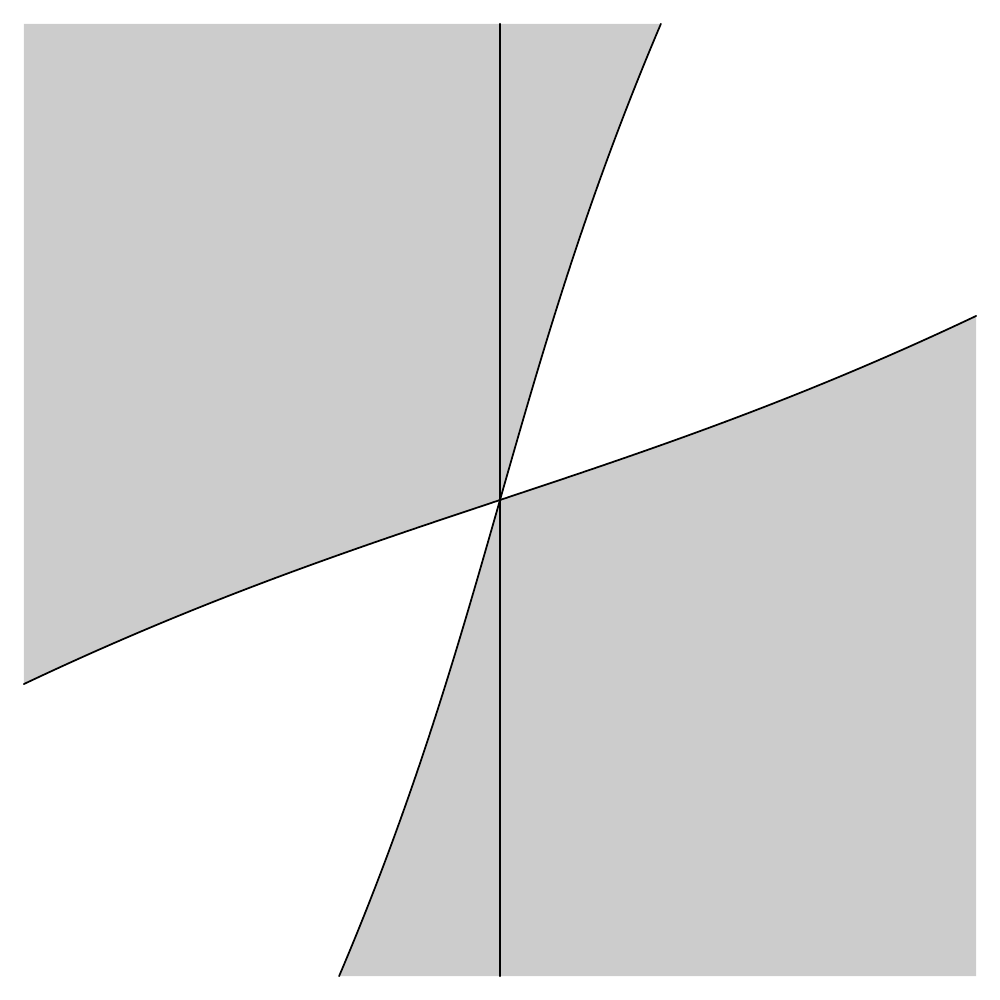, width=0.2\textwidth}
\hspace{2em}
\epsfig{file=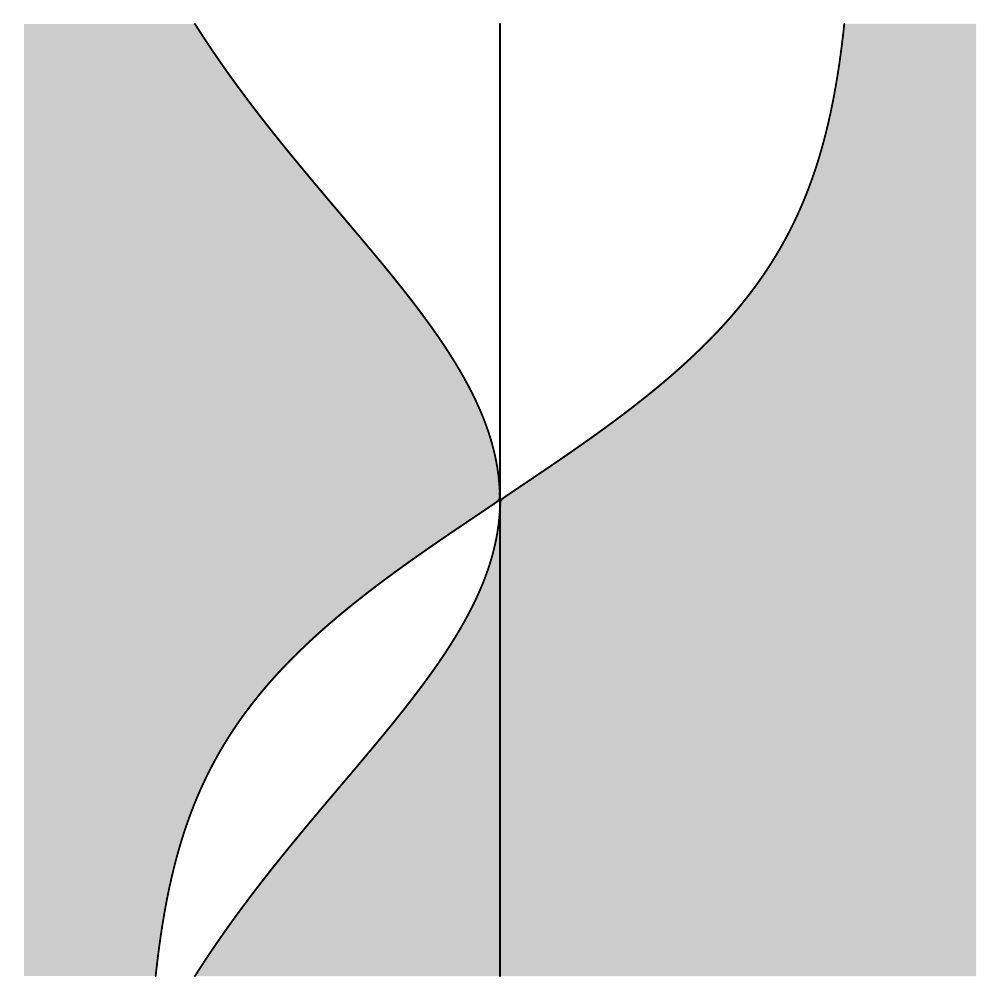, width=0.2\textwidth}
\caption{Positive, negative and neutral constrictions (figure made by S.I.Tertychnyi, included to the paper with his permission). It is proved that negative and neutral constrictions do not exist.}
\end{center}
\end{figure}

\begin{remark} \label{posneg} It was shown in \cite[theorem 1.8]{g18} that each constriction is either 
positive, or negative: there are no neutral constrictions. Positivity of constrictions was stated  there as 
\cite[conjecture  1.13]{g18}. 
 It was also shown in \cite{g18} that Theorem \ref{thpos}  would imply Theorem \ref{thal}.
\end{remark}

\begin{definition} A {\it ghost constriction} is a constriction $(B,A;\omega)$ 
in model of Josephson junction for which either $\ell:=\frac B{\omega}$ is 
different from the corresponding rotation number $\rho(B,A;\omega)$, or the constriction is 
not of positive type. (Note that $\ell\in\zz$ for each constriction, 
see Remark \ref{remint}.) 
\end{definition}

Theorems \ref{thal} and \ref{thpos} taken together are equivalent to the following theorem.

\begin{theorem} \label{noghost} There are no ghost constrictions in the model
of overdamped Josephson junction.
\end{theorem}

Proof of Theorem \ref{noghost} is sketched in the next subsection, 
where the plan of the paper is presented. It is  
 based on the following characterization of constrictions.

\begin{proposition} \label{propoinc} \cite[proposition 2.2]{4} 
Consider the period $2\pi$ flow map $h^{2\pi}$ of system (\ref{josvec}) 
acting on the transversal coordinate $\theta$-circle $\{\tau=0\}$. A point $(B,A;\omega)$ 
is a constriction, if and only if $\omega,A\neq0$ and $h^{2\pi}=Id$. 
\end{proposition}

Some applications of Theorems \ref{thpos} and \ref{thal} and open problems will be discussed in Section 6.

\subsection{Method of proof of 
Theorem \ref{noghost}.  Plan of the paper}

We prove Theorem \ref{noghost} in two steps given by the two following theorems. 
To state them, let us introduce the following notation. We set $\eta:=\omega^{-1}$. 
 For every fixed $\ell\in\zz$ we consider the set 
$$Constr_\ell=\{ (\mu,\eta)\in\rr_+^2 \ | \ (B,A;\omega)=(\ell\eta^{-1},2\mu\eta^{-1};\eta^{-1})  \text{ is a constriction}\}.$$

\begin{theorem} \label{famcons0} For every $\ell\in\zz$ the subset $Constr_{\ell}\subset\rr_+^2$ 
is a regular one-dimensional analytic submanifold   in $\rr_+^2$. The restriction of the 
coordinate $\eta$ to each its connected component is unbounded from above 
(i.e., $\omega$ is unbounded from below). The  rotation number and the 
type of constriction (positive or negative, see Remark \ref{posneg})  are constant on 
each component. 
\end{theorem}

\begin{theorem} \label{noq} For every $\ell\in\zz$ there are no ghost 
constrictions in the axis 
 $\La_\ell:=\{ B=\omega\ell\}$ whenever $\omega>0$ is small enough (dependently on 
 $\ell$).
 \end{theorem}

Theorem \ref{noq} will be proved in Section 5 by methods of the theory of 
slow-fast families of dynamical systems. Theorem \ref{noghost} immediately follows from Theorems \ref{famcons0} and \ref{noq}, see the Subsection 5.6.

The proof of  
Theorem \ref{famcons0} is sketched below. It is based on the following 
equivalent description of model of Josephson junction 
by a family of two-dimensional linear systems of differential equations on the Riemann sphere, 
see \cite{bkt1, buch2, bt1, Foote, LSh2009, IRF}, \cite[subsection 3.2]{bg}. 
The variable change 
$$z=e^{i\tau}=e^{i\omega t}, \ \ \Phi=e^{i\theta}=ie^{i\phi}$$
transforms  equation (\ref{jostor}) on the function $\theta(\tau)$  
to the Riccati equation 
\begin{equation}\frac{d\Phi}{dz}=z^{-2}((\ell z+\mu(z^2+1))\Phi+\frac z{2\omega}(\Phi^2+1)).\label{ric}\end{equation}
Recall that $\ell=\frac B\omega$, $\mu=\frac A{2\omega}$, see (\ref{elmu}). 
Equation (\ref{ric}) is the projectivization of the two-dimensional linear system 
\begin{equation} Y'=\left(\frac{\diag(-\mu,0)}{z^2}+\frac{\mcb}z+\diag(-\mu,0)\right)Y, \ \ \ 
\mcb=\left(\begin{matrix} -\ell &-\frac1{2\omega}\\ \frac1{2\omega} & 0\end{matrix}\right), 
\label{tty}\end{equation}
in the following sense: a function $\Phi(z)$ is a solution of (\ref{ric}), if and only if 
$\Phi(z)=\frac vu(z)$, where the vector function $Y(z)=(u(z),v(z))$ is a solution of 
system (\ref{tty}). For $\mu>0$ system (\ref{tty}) has two irregular 
nonresonant singular points at $0$ and at $\infty$. Its {\it monodromy operator} 
acts on the space $\cc^2$ of germs of its solutions at a given point $z_0\in\cc^*$ 
by analytic extension along a counterclockwise circuit around zero. 
\begin{remark} The variable change $E(z):=e^{\mu z}v(z)$ transforms the 
family of systems (\ref{tty}) to the following family of {\it special double confluent Heun equations,} see \cite{tert2}, \cite{bt0}--\cite{bt4}: 
  \begin{equation} z^2E''+((\ell+1)z+\mu(1-z^2))E'+(\lambda-\mu(\ell+1)z)E=0, \ 
  \la:=\frac1{4\omega^2}-\mu^2.\label{heun}\end{equation}
We will also deal with the so-called {\it conjugate} Heun equation obtained from 
(\ref{heun}) by change of sign at $\ell$: 
  \begin{equation} z^2E''+((-\ell+1)z+\mu(1-z^2))E'+(\lambda+\mu(\ell-1)z)E=0.\label{heun2}\end{equation}
Using this relation to well-known class of Heun equations  a series 
of results on phase-lock area portrait of model of Josephson junction and related problems were obtained 
in  \cite{tert2}, \cite{bt0}--\cite{bt4}, \cite{bg, bg2}. See also a brief survey in the 
next subsection. 
\end{remark}

Recall that an {\it isomonodromic family} of linear systems   is a family in which 
 the collection of residue matrices of formal normal forms at  singular points,  
 Stokes matrices and transition matrices between canonical solution bases at different 
 singular points remain constant (up to appropriate conjugacies).

It is known that $(B,A;\omega)$ is a constriction, if and 
only if $A,\omega\neq0$ and system (\ref{tty}) has trivial monodromy; then $\ell=\frac B{\omega}\in\zz$. 
See \cite[proposition 3.2, lemma 3.3]{4} and Proposition \ref{trivmc} in Section 4.

We denote by $Jos$ the three-dimensional family of systems (\ref{tty}), which will be referred to, as 
{\it systems of Josephson type.} For the proof of Theorem \ref{famcons0} we study their isomonodromic deformations in   the four-dimensional space $\bbj$ of linear systems of the so-called {\it normalized $\rr_+$-Jimbo type} 
  \begin{equation}Y'=\left(-\tau\frac{K}{z^2}+\frac{R}z+\tau N\right)Y,  \ \tau\in\rr_+, \ K, R, N \text{ are 
  real 2x2-matrices,}\label{matkr}\end{equation}
  \begin{equation} N=\left(\begin{matrix} -\frac12 & 0\\ 0 & 0\end{matrix}\right), \ 
  R=\left(\begin{matrix}-\ell & -R_{21}\\ R_{21} & 0\end{matrix}\right),  \ 
  K=-GNG^{-1},  \ R_{21}>0, \ \ell\in\rr,\label{matkrn}\end{equation}
where $G\in\sl_2(\rr)$ is a matrix such that 
\begin{equation}G^{-1}RG=\left(\begin{matrix}-\ell & *\\ * & 0\end{matrix}\right);\label{grg}
\end{equation}
here the matrix elements $*$ may be arbitrary. 

Step 1.  We study the real one-dimensional analytic foliation of the space 
$\bbj$ by isomonodromic 
families of linear systems. These isomonodromic families are given by differential equation (\ref{isonormal}). They are obtained 
(by gauge transformations and rescaling of the variable $z$)  from 
 well-known Jimbo isomonodromic deformations \cite{J}, 
 which are given by real solutions of Painlev\'e 3 equation (P3). Namely, 
the function 
$w(\tau)=-\frac{R_{12}(\tau)}{\tau K_{12}(\tau)}=\frac{R_{21}(\tau)}{\tau K_{12}(\tau)}$ 
should satisfy  the 
P3 equation 
\begin{equation} w''=\frac{(w')^2}w-\frac{w'}{\tau}+w^3-2\ell\frac{w^2}{\tau}-\frac1w+(2\ell-2)\frac1\tau\label{p30}\end{equation}
along the isomonodromic leaves.  We show that the hypersurface $Jos\subset\bbj$ 
corresponds to poles of order 1 with residue 1 of solutions of (\ref{p30}). This implies 
that $Jos$ is  {\it transversal} to the isomonodromic foliation. This is 
the  key lemma in the proof. The only role of the Painlev\'e 3 equation in the proof is the above transversality statement. 
 Relation to Painlev\'e 3 equation led us to 
a series of new open problems presented in Section 6.

Step 2. We consider the subset $\Sigma\subset\bbj$ of systems (\ref{matkr}) with 
trivial monodromy. We show that $\ell\in\zz$ for these systems, and  their germs  at $0$ and 
at $\infty$ are analytically equivalent 
to their diagonal formal normal forms. Using this fact, we  show that $\Sigma$ 
 is a real two-dimensional analytic submanifold in $\bbj$ with the following properties:

2.1) $\Sigma$ is a union of leaves of the isomonodromic foliation; 

2.2) (the key theorem in the proof) there exists a {\it submersive} projection $\mcr:\Sigma\to\rr_x$ given 
by an analytic invariant $\mcr$ of linear systems, the so-called {\it transition cross-ratio}, 
that is constant along the leaves.

Statement 2.1) follows from definition. 
Statement 2.2) is proved by showing that $(x,\tau)$, $x=\mcr$, form local analytic coordinates on 
$\Sigma$. Fix an $\ell\in\zz$, and let $\Sigma_{\ell}\subset\Sigma$ denote the 
subset of systems with the given value of $\ell$. 
For every pair $(x_0,\tau_0)$ corresponding 
to a system from $\Sigma_\ell$ realization of any pair $(x,\tau)$ close to $(x_0,\tau_0)$ 
by a system from $\Sigma_\ell$ can be viewed as a solution of the Riemann--Hilbert type 
problem. It is proved via holomorphic vector bundle argument, 
as in famous works by A.A.Bolibruch  on the Riemann--Hilbert Problem 
and related topics: see \cite{Bol89}--\cite{Bol18} and references therein.  
We glue a holomorphic vector bundle with 
connection on $\oc$ realizing given $(x,\tau)$ from the two trivial bundles: one over the disk $D_2\subset\cc$, and the 
other one on the complement of the closed 
disk $\overline D_{\frac12}\subset\oc$. The connections on the latter trivial bundles are given by the diagonal normal 
forms prescribed by $\ell$ and $\tau$.  The gluing matrix, which is holomorphic on the annulus $D_2\setminus\overline D_{\frac12}$,  
 depends analytically on $(x,\tau)$.  The bundle thus 
 obtained is trivial for $(x_0,\tau_0)$ (by definition). It remains trivial for 
 all $(x,\tau)$ close enough to $(x_0,\tau_0)$. This follows from the classical theorem stating that 
{\it a holomorphic vector bundle close to a trivial one is also trivial \cite[appendix 3, lemma 1, theorem 2]{Bol18}, \cite[theorem 2.3]{rohrl}, \cite{grauert}.} The connection on the trivial bundle 
thus obtained is given by a meromorphic system with order two poles at $0$ and at 
$\infty$ and the same normal forms.  Its gauge equivalence to a normalized $\rr_+$-Jimbo type system 
(\ref{matkr}), (\ref{matkrn}) is proved by a symmetry argument.

The submanifold $Jos$ is transversal to $\Sigma$, by the result of Step 1 and 
Statement 2.1). Therefore, the intersection $Jos\cap\Sigma_\ell$ is a real one-dimensional 
submanifold in $\bbj$. It is transversal to the isomonodromic foliation of 
$\Sigma_\ell$ (Step 1), and hence,  is locally diffeomorphically 
projected to an open subset in $\rr$ by the mapping $\mcr$. The above intersection is 
identified with $Constr_\ell$. 
This implies that $Constr_\ell$ is a one-dimensional submanifold; each its 
 connected component  
is  analytically parametrized by an interval $I=(a,b)$ of values of the parameter 
$x=\mcr$ and hence,  is non-compact. 
 
 Step 3. We show that the coordinate $\eta=\omega^{-1}$ is unbounded from above 
 on each component $\mcc$ in $Constr_\ell$. Assuming the contrary, i.e., that 
 $\eta$ is bounded on $\mcc$, we have that  for every $c\in\{ a,b\}$ at least one of 
 the functions $\mu^{\pm1}$, $\eta^{-1}$ (depending on the choice of $c$) 
 should be unbounded, as $x\to c$.   
 Boundedness of  $\mu$ is proved by using Klimenko--Romaskevich 
 Bessel asymptotics of boundaries of the phase-lock areas  \cite{RK}. For $c\neq0$ 
 we prove boundedness of the functions $\mu^{-1}$, $\eta^{-1}$, as $x\to c$, 
 by studying accumulation points of the set $Constr_{\ell}$ in the union of coordinate 
 axes $\{\eta=0\}\cup\{\mu=0\}$.

 Afterwards, to finish the proof of Theorem \ref{famcons0}, it remains to show 
 that the rotation number and type of constriction are constant on each connected 
 component in $Constr_{\ell}$. We deduce constance of type  from the fact 
 that no constriction can be a limit of the so-called {\it generalized simple intersections:}  
those points of intersections $\La_\ell\cap\partial L_r$, 
$r\equiv\ell(\mod 2)$, that are not constrictions and do not lie in the abscissa axis. This, in its turn, is implied by the two following  facts: 

-  the generalized simple intersections correspond to  Heun equations (\ref{heun2}) 
having a polynomial solution  
 \cite[theorem 1.15]{bg2};  this remains valid for their limits with $A\neq0$;
 
 - no constriction can   correspond to  a Heun equation (\ref{heun2}) with polynomial solution  \cite[theorems 3.3, 3.10]{bg}.

   \subsection{Historical remarks}

Model (\ref{jos}) of overdamped Josephson junction was studied by V.M.Buchstaber, 
O.V.Karpov, S.I.Tertychnyi and other mathematicians and physicists, 
see    \cite{bg}--\cite{bt4}, \cite{lev}, 
\cite{4}--\cite{gn19}, \cite{LSh2009, IRF, krs, RK, tert, tert2} and references therein. Hereby we 
present a brief survey of results that were not mentioned in the introduction. 
Recall that the rotation number quantization effect for a family of dynamical systems on 
$\tt^2$ containing (\ref{josvec}) was discovered in \cite{buch2}. 
I.A.Bizyaev, A.V.Borisov, and I.S.Mamaev noticed that a big family of dynamical systems on torus in which the  rotation number quantization effect realizes 
was introduced   by W.Hess  (1890). It appears that in classical mechanics such systems  were studied in problems on  rigid body movement with fixed point in works by  
 W.Hess, P.A.Nekrassov, A.M.Lyapunov, B.K.Mlodzejewski, N.E.Zhukovsky and others. 
See \cite{bbb,  lyap, mlodnek, nekr, zhuk} and references therein. P.A.Nekrasov observed 
in \cite{nekr} that  the above-mentioned big family of  systems
considered by Hess can be equivalently described by a Riccati equation (or by    
 a linear second order differential equation). 

Transversal regularity of the fibration by level sets $\rho(B,A)=const\notin\zz$ with fixed $\omega$ 
on the complement to the union of the phase-lock areas was proved in \cite[proposition 5.3]{bg2}. Conjectures on alignment and positivity of constrictions 
(now Theorems \ref{thal} and \ref{thpos} respectively) were 
stated respectively  in \cite{4} and \cite{g18} and studied respectively in \cite{4, g18} and \cite{g18}, where some partial results were obtained. Theorem \ref{thal} for 
$\omega\geq1$ was proved in \cite{4}. For further survey on these conjectures 
 see \cite{4, bg2, g18} and references therein. A conjecture saying that 
 the semiaxis $\La_{\ell}^+:=\La_\ell\cap\{ A>0\}$ intersects the corresponding phase-lock area $L_\ell$ by a ray explicitly constructed in \cite{g18} was stated in \cite[conjecture 1.14]{g18}.   It 
was shown in \cite[theorem 1.12]{g18} that the ray in question indeed lies in $L_\ell$.  An equivalent description of model (\ref{jos}) in terms of a family of special  double confluent Heun equations (\ref{heun}) was found by S.I.Tertychnyi in \cite{tert2} and 
further studied in a series of joint papers by V.M.Buchstaber and S.I.Tertychnyi 
\cite{bt0}--\cite{bt4}.  They have shown that the constrictions are exactly those 
parameter values $(B,A;\omega)$ for which the corresponding double confluent Heun 
equation (\ref{heun}) has an entire solution: holomorphic on $\cc$ \cite{bt1}.  
Using this observation  they stated a conjecture 
describing ordinates of the constrictions lying in a given axis $\La_\ell$  as 
zeros of a known analytic function constructed via an infinite matrix product \cite{bt1}. 
This conjecture was studied in \cite{bt1, bt2} and reduced to the conjecture 
stating that if the Heun equation  (\ref{heun})  has an entire solution, then the conjugate 
Heun equation (\ref{heun2}) cannot have polynomial solution. Both conjectures were 
proved in \cite{bg}. New automorphisms of solution space of Heun equations (\ref{heun}) 
were discovered and studied in \cite{bt3, bt4}.

In \cite{bt0} V.M.Buchstaber and S.I.Tertychynyi described those $(B,A;\omega)$, for which  conjugate  Heun equation (\ref{heun2}) has a polynomial solution. Namely, for a given 
$\ell=\frac B{\omega}\in\nn$ their set is  a remarkable algebraic curve, the so-called 
spectral curve (studied in \cite{bt0, gn19}): zero locus of determinant of appropriate three-diagonal matrix with entries being linear 
non-homogeneous functions in the coefficients of equation (\ref{heun2}).  
The fact that those points $(B,A;\omega)$ for which (\ref{heun2}) has a polynomial solution 
are exactly the generalized simple intersections is a result of papers 
 \cite{bt0, bg2},  stated and proved in \cite{bg2}.

 There exists an  antiquantization procedure  that associates Painlev\'e equations to Heun equations; double confluent Heun equations correspond to Painlev\'e 3 equations. 
 See   \cite{sasla, S, slaste} and references therein.

V.M.Buchstaber, O.V.Karpov, S.I.Tertychnyi, D.A.Filimonov, V.A.Kleptsyn, I.V.Schurov 
made numerical experiences that have shown that as $\omega\to0$, 
the "upper" part of the phase-lock area portrait converges to a kind of 
parquet in the renormalized coordinates $(\ell,\mu)$: the renormalized 
phase-lock areas tend  to unions of pieces of parquet, and gaps between 
the phase-lock areas tend to zero. See Fig. 3 and the paper \cite{krs}. 
This is an open problem. In \cite{krs}  V.A.Kleptsyn, O.L.Romaskevich, and 
I.V.Schurov proved some results on smallness of gaps and  their rate of 
convergence to zero, as $\omega\to0$, using methods of slow-fast systems. 

A subfamily of family (\ref{jostor}) of dynamical systems on 2-torus 
was studied by J.Guckenheimer and Yu.S.Ilyashenko in \cite{ilguk} from the 
slow-fast system point of view. They obtained results on  its 
limit cycles, as $\omega\to0$.  
 
An analogue of the rotation number integer quantization effect  in braid groups 
was discovered by A.V.Malyutin \cite{malyutin}.

\section{Preliminaries: irregular singularities, normal forms, 
Stokes matrices and monodromy--Stokes 
data of linear systems} 
\subsection{Normal forms, canonical solutions and Stokes matrices}

All the results presented in this subsection are  particular cases of  classical results 
contained in \cite{2, BUL, bjl, 12, jlp, sib}. 

Recall that two germs of meromorphic linear systems of differential 
equations on a $n$-dimensional vector function $Y=Y(z)$ 
at a singular point (pole), say, 0 are {\it analytically equivalent,} 
if there exists a holomorphic 
 $\gl_n(\cc)$-valued function $H(z)$  on a neighborhood of 0 such that the $Y$-variable change 
 $Y=H(z)\wt Y$ sends  one system to the other one. Two systems are {\it formally} 
 equivalent, if the above is true for a formal power series $\widehat H(z)$ with 
  matrix coefficients that has an invertible free term.
 
 Consider a two-dimensional linear system 
 \begin{equation}Y'=
\left(\frac K{z^2}+\frac Rz+O(1)\right)Y, \ \ \ 
Y=\left(\begin{matrix} u\\ v\end{matrix}\right),\label{eqlin}\end{equation}
on a neighborhood of 0; here the matrix $K$ has distinct eigenvalues $\la_1\neq\la_2$, and $O(1)$ is a holomorphic 
matrix-valued function on a neighborhood of $0$. 
Then we say that the singular point 0 of system (\ref{eqlin}) is {\it irregular non-resonant 
of Poincar\'e rank 1.}  Then $K$ is conjugate to  
$\wt K=\diag(\la_1,\la_2)$, $\wt K=\mathbf H^{-1}K\mathbf H$, $\mathbf H\in GL_2(\cc)$,  and one can achieve that 
$K=\wt K$ by applying the constant linear change (gauge transformation)
$Y=\mathbf H\mathbf{\wh Y}$. System (\ref{eqlin}) is 
formally equivalent to a unique {\it formal normal form} 
\begin{equation}\wt Y'=\left(\frac{\wt K}{z^2}+\frac{\wt R}z\right)\wt Y, \ \wt K=\diag(\la_1,\la_2), 
\ \wt R=\diag(b_1,b_2),\label{nform}
\end{equation}
\begin{equation} \wt R \ \  \text{ is the diagonal part of the matrix } \ \  \mathbf H^{-1}R\mathbf H.
\label{resfo}\end{equation} 
The matrix coefficient $K$ in  system (\ref{eqlin}) and the corresponding  matrix $\wt K$ 
in (\ref{nform}) are called the 
{\it main term matrices}, and $R$, $\wt R$  the {\it residue matrices}. 

Generically, the normalizing series $\widehat H(z)$ bringing (\ref{eqlin}) to (\ref{nform}) 
diverges. At the same time, there exists a covering of a punctured 
neighborhood of zero by two sectors $S_0$ and $S_1$ with vertex at 0 in which 
there exist holomorphic matrix functions $H_j(z)$, $j=0,1$, that are $C^{\infty}$ 
smooth on  $\overline S_j\cap D_r$ for some $r>0$, and such that the variable changes $Y=H_j(z)\wt Y$ 
transform (\ref{eqlin}) to (\ref{nform}). This Sectorial Normalization Theorem 
 holds for the so-called 
{\it good sectors} (or {\it Stokes sectors}.) Namely, consider the rays issued from 0 and forming the 
set $\{\re\frac{\la_1-\la_2}z=0\}$. They are called {\it imaginary dividing rays} (or {\it Stokes rays}).  
A sector $S_j$ is {\it good,} if it contains  one imaginary dividing ray 
 and its closure does not 
contain the other one. 

Let $W(z)=\diag(\wt Y_1(z),\wt Y_2(z))$ denote the canonical diagonal fundamental 
matrix solution of the formal normal form (\ref{nform}); here $\wt Y_\ell(z)$ are solutions of its 
one-dimensional equations. The matrices 
$X^j(z):=H_j(z)W(z)$ are fundamental matrix solutions of the initial equation (\ref{eqlin}) 
defining solution bases in $S_j$ called the {\it canonical sectorial solution bases.} 
In their definition we choose the branches $W(z)=W^j(z)$ of the (a priori multivalued) matrix function $W(z)$ in $S_j$, $j=0,1$, so that $W^1(z)$ is obtained from $W^0(z)$ by counterclockwise analytic extension from $S_0$ to $S_1$. 
And in the same way we define yet 
another  branch 
$W^2(z)$ of $W(z)$ in $S_2:=S_0$ that is obtained from $W^1(z)$ by counterclockwise 
analytic extension from $S_1$ to $S_0$. This yields 
another canonical matrix solution $X^2:=H_0(z)W^2(z)$ in $S_0$, which is obtained 
from $X^0(z)$ by multiplication from the right by the monodromy matrix $\exp(2\pi i\wt R)$ 
of the formal normal form (\ref{nform}).  Let $S_{j,j+1}$ denote the connected component 
of intersection $S_{j+1}\cap S_j$, $j=0,1$, that is crossed when one moves 
from $S_j$ to $S_{j+1}$ counterclockwise, see Fig. 5. The transition matrices $C_0$, $C_1$ 
between thus defined canonical solution bases $X^j$, 
\begin{equation} X^1(z)=X^0(z)C_0 \text{ on }  S_{0,1}; \ \ X^2(z)=X^1(z)C_1 \text{ on } S_{1,2}\label{stokes}\end{equation}
are called the {\it Stokes matrices.}  
\begin{example} Let  $A=\diag(\la_1,\la_2)$, and let  $\la_2-\la_1\in \rr$. Then 
the imaginary dividing rays are the positive and negative imaginary semiaxes. 
The good sectors $S_0$ and $S_1$ covering $\cc^*$ satisfy the following 
conditions:

- the sector $S_0$ contains the positive imaginary semiaxis, and its closure does 
not contain the negative one;

- the sector $S_1$ satisfies the opposite condition. See Fig. 5.
\end{example}

\begin{figure}[ht]
  \begin{center}
   \epsfig{file=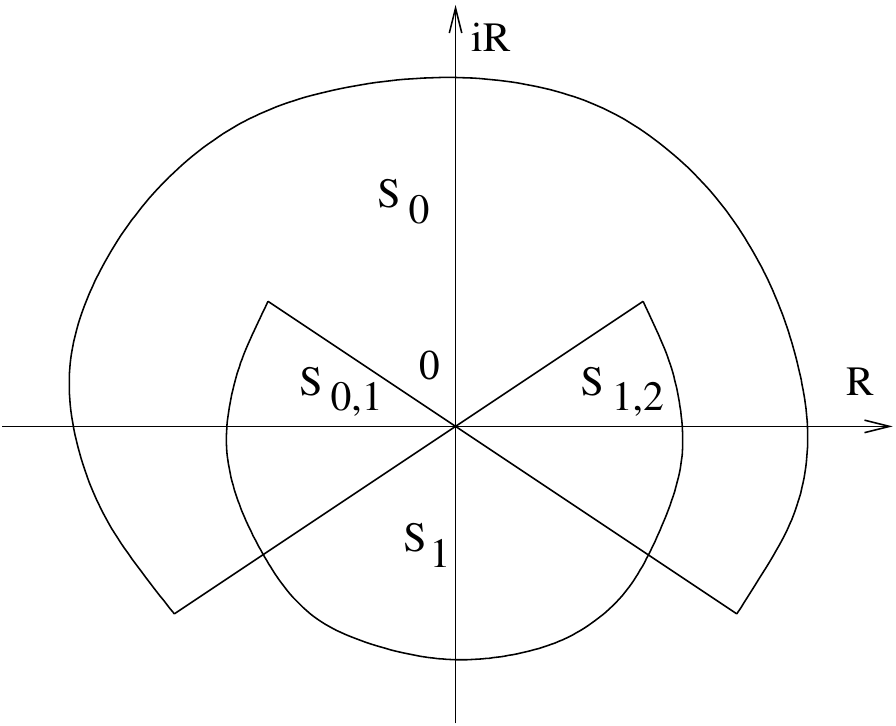}
    \caption{Good sectors in the case, when $\la_1-\la_2\in\rr$.}
  \end{center}
\end{figure}

\begin{example} \label{triangtype} Let us numerate the sectors $S_0$, $S_1$ and the eigenvalues 
$\la_1$, $\la_2$ so that  
\begin{equation}
\overline S_{0,1}\setminus\{0\}\subset\{\re\left(\frac{\la_1-\la_2}z\right)>0\}, \ 
\ \overline S_{1,2}\setminus\{0\}\subset\{\re\left(\frac{\la_1-\la_2}z\right)<0\}.
\label{sigmy01}\end{equation}
This holds, e.g., in the conditions of the above example, if $\la_2-\la_1>0$. 
The canonical solutions of the formal normal form (\ref{nform}) are given by the 
solutions $c_kz^{b_k}e^{-\frac{\la_k}z}$ of one-dimensional equations in (\ref{nform}). They are numerated by indices 
$k=1,2$ of the eigenvalues $\la_k$ of the main term matrix $K$. The corresponding 
solutions of the initial system (\ref{eqlin}) in $S_j$, $j=0,1,2$, i.e., the columns of 
the fundamental matrix $X^j(z)$, are also numerated by the same index $k$ 
and will be denoted by $f_{kj}(z)$. The norm $||f_{1j}(z)||$ is asymptotically dominated by 
$||f_{2j}(z)||$ in $S_{0,1}$, as $z\to0$, 
and the converse asymptotic domination statement holds on $S_{1,2}$. 
This implies (and it is well-known) that $f_{10}\equiv f_{11}$ on $S_{0,1}$ and 
$f_{22}\equiv f_{21}$ on $S_{1,2}$. The Stokes matrices $C_0$ and $C_1$ 
are unipotent:  $C_0$ is upper-triangular and $C_1$ is lower-triangular. 
If the numeration of either eigenvalues, or  sectors (but not both) is opposite, 
or if the singular point under question is $\infty$, not zero (see Remark  \ref{rinf} 
below), then the Stokes matrices are unipotent but of opposite triangular type. 
\end{example}
\begin{remark} \label{riccun} The tautological projection $\cc^2\setminus\{0\}\to\cp^1=\oc$ 
sends canonical sectorial basic solutions $f_{kj}(z)$ of system (\ref{eqlin}) 
to canonical 
sectorial solutions $q_{kj}(z)$ of its projectivization: the corresponding Riccati equation. 
These are the unique $\oc$-valued 
holomorphic solutions of the Riccati equation in the sector $S_j$ 
that extend $C^{\infty}$-smoothly to $\overline S_j\cap D_r$ for a sufficiently small 
 $r>0$. Their values at $0$ are the projections of 
the eigenlines of the main term matrix $K$ with eigenvalues $\la_k$.
\end{remark}

\begin{theorem} \label{anclass} \cite{2, BUL, bjl, 12, jlp, sib} A germ of linear system  at an 
irregular nonresonant singular point is analytically equivalent to its formal normal form, 
if and only if it has trivial Stokes matrices. Two germs of linear systems as above 
are analytically equivalent, if and only if 
their  formal normal forms are the same and their Stokes 
matrix collections are equivalent in the following sense: they are simultaneously 
conjugated by one and the same diagonal matrix (independent on the choice of sector 
$S_{j,j+1}$).
\end{theorem}

Recall that the {\it monodromy operator} of a germ of linear system at $0$ 
acts on the space of germs of its solutions at a point $z_0\neq0$ sending a local 
solution to the result of its counterclockwise analytic extension along a circuit around 
the origin. Let the origin be an irregular nonresonant singular point of Poincar\'e rank 1, 
and let $S_0$, $S_1$ be the corresponding good sectors. 
Let $M$ be the monodromy matrix written in the  canonical sectorial basis of solutions in $S_0$. Let $M_{norm}$ denote the diagonal monodromy matrix of the formal normal form 
in the canonical solution basis with diagonal fundamental matrix. We will call $M_{norm}$ the {\it formal monodromy.} Recall that 
$M_{norm}=\exp(2\pi i\wt R)$.  The matrix 
$M$ is expressed in terms of the formal monodromy $M_{norm}$ and the Stokes 
matrices $C_0$ and $C_1$ via the following 
well-known formula \cite[p.35]{12}: 
\begin{equation}M=M_{norm}C_1^{-1}C_0^{-1}.\label{monst}\end{equation}

\begin{remark} \label{rinf}
We will  also deal with the case, when the singular point under question is $\infty$, 
and the above statements hold in the local coordinate $\wt z:=\frac1z$. 
In the coordinate $z$ the corresponding equation 
and formal normal form take the form 
$$Y'=\left(K+\frac Rz+O\left(\frac1{z^2}\right)\right)Y, \ \ \wt Y'=\left(\wt K+\frac{\wt R}z\right)\wt Y.$$
The matrices $K$, $\wt K$ are called the {\it main term matrices,} and $R$, $\wt R$ the 
{\it residue matrices} of the corresponding systems at $\infty$. Let $\la_1$, $\la_2$ 
be the eigenvalues of the matrix $K$. An  
{\it imaginary dividing ray at infinity} is a ray issued from $0$ 
and  lying in the set $\{\re(\la_1-\la_2)z=0\}$. This yields the definition 
of good sectors "at infinity".  
The Sectorial Normalization and Analytic Classification Theorems and the definition of 
Stokes matrices at infinity are stated in the same way, as above; the sectors $S_0$, $S_{0,1}$, 
$S_1$, $S_{1,2}$ at infinity are also numerated counterclockwise. Formula 
(\ref{monst}) also holds at $\infty$. 
\end{remark}

\subsection{Systems with two irregular singularities. Monodromy--Stokes data} 
\begin{definition} By $\mcho$ we will denote the class of linear systems on the 
Riemann sphere having two singular points, at zero and at infinity, such that both of them 
are irregular nonresonant of Poincar\'e rank 1.  Each system from the class $\mcho$ 
has the type 
\begin{equation}Y'=\left(\frac K{z^2}+\frac Rz+N\right)Y, \ \ \  \ K,R,N\in\End(\cc^2),
\label{yaz}\end{equation}
where each one of the main term matrices $K$ and $N$ at zero and at $\infty$ 
has distinct eigenvalues. 
\end{definition}
\begin{definition} Consider a linear system $\mcl\in\mcho$. 
Fix a point $z_0\in\cc^*$ and two pairs of good sectors 
($S_{0}^0$, $S_1^0$), ($S_{0}^{\infty}$, $S_1^{\infty}$) for the main term matrices at $0$ and $\infty$ respectively, 
see Remark \ref{rinf}. Fix two paths 
$\alpha_p$  in $\cc^*$ numerated by $p=0,\infty$, going from the point 
$z_0$ to a point  in $S_{0}^p$. Let $f_{1p}$, $f_{2p}$ be a canonical sectorial solution basis  
for the system $\mcl$ at $p$ in $S_0^p$. Consider the analytic extensions of the basic functions $f_{kp}$ 
to the point $z_0$ along paths $\alpha_p^{-1}$. Let $\pi:\cc^2\setminus\{0\}\to\cp^1$ 
denote the tautological projection. Set $\Phi:=\frac{Y_2}{Y_1}$, 
\begin{equation} q_{kp}:=\pi(f_{kp}(z_0))\in\cp^1=\oc_{\Phi}.\label{qjp}\end{equation}
Let $M$ denote the monodromy operator of the system $\mcl$ acting on the local solution 
space at $z_0$ (identified with the space $\cc^2$ of initial conditions 
at $z_0$) by analytic extension along counterclockwise circuit around zero. The tuple 
\begin{equation}(q,M):=(q_{10},q_{20},q_{1\infty}, q_{2\infty}; M)
\label{msd}\end{equation}
taken up to the next equivalence is called the {\it monodromy--Stokes data} of the system 
$\mcl$. Namely, two tuples $(q,M),(q',M')\in(\cp^1)^4\times\gl_2(\cc)$ are called 
{\it equivalent\footnote{Here is an equivalent group-action definition. 
The group $\psl_2(\cc)$ acts on $\oc^4\times\gl_2(\cc)$ by action 
$h:q_{kp}\mapsto hq_{kp}$ on points in $\oc$ and conjugation $M\mapsto hMh^{-1}$ 
on matrices. The monodromy--Stokes data is the $\psl_2(\cc)$-orbit of a collection 
$(q,M)$ under this action.},} if there exists a linear operator $H\in\gl_2(\cc)$ whose projectivization 
sends $q_{kp}$ to $q'_{kp}$ and such that $H^{-1}\circ M'\circ H=M$. 
We will also deal with the 
{\it transition matrix} $Q$ comparing the canonical bases at $0$ and at $\infty$ at $z_0$: 
$(f_{1\infty},f_{2\infty})=(f_{10},f_{20})Q$. 
\end{definition}
\begin{remark} The  monodromy--Stokes data of a system $\mcl$ depends only 
on the homotopy class of the pair of paths $(\alpha_0,\alpha_{\infty})$ 
in the space of pairs of paths in $\cc^*$ with a common (variable) starting point $z_0$ and with endpoints 
lying in given sectors $S_0^0$ and $S_0^\infty$ respectively. Indeed, let a homotopy between 
two pairs of paths, $(\alpha_0,\alpha_{\infty})$ with base point $z_0$ and $(\alpha_0',\alpha_\infty')$ 
with base point $z_0'$, move $z_0$ to  $z_0'$ along a path $\beta$ in $\cc^*$. Let $X(z)$ 
be the germ of fundamental matrix of the system $\mcl$ at $z_0$ such that $X(z_0)=Id$. 
Let $H=X(z_0')$ denote the value at $z_0'$ of the analytic extension of the fundamental matrix function 
$X(z)$ along the path $\beta$. Then $H$ transforms the monodromy--Stokes data corresponding 
to $z_0$ and the path pair $(\alpha_0,\alpha_{\infty})$ to that corresponding to $z_0'$ and the path pair  
$(\beta^{-1}\alpha_0,\beta^{-1}\alpha_{\infty})$, as at the end of the above definition.
\end{remark}

\begin{proposition} \label{proconsm} One has $q_{1p}\neq q_{2p}$ for every $p=0,\infty$. 
 The monodromy--Stokes data of a system $\mcl\in\mcho$ determines the collection of 
formal monodromies $M_{norm,p}$, the Stokes matrices 
$C_{jp}$  at $p=0,\infty$, $j=0,1$, and the transition matrix $Q$ 
uniquely  up to the following 
equivalence. Two collections $(M_{norm,p},C_{jp},Q)$ and $(M'_{norm,p},C_{jp}',Q')$ are {\bf equivalent,} 
if $M_{norm,p}=M_{norm,p}'$ and there exists a pair of diagonal matrices $D_0$, $D_{\infty}$ such that  
$C_{jp}'= D_pC_{jp}D_p^{-1}$ for all $j$, $p$, and $Q'=D_0\circ Q\circ D_{\infty}^{-1}$.
\end{proposition}
\begin{proof} The inequality $q_{1p}\neq q_{2p}$ follows from linear independence 
of the basic functions $f_{1p}$, $f_{2p}$, which  implies independence of 
their values at $z_0$. A given pair of distinct points $q_{1p},q_{2p}\in\cp^1$ 
defines a basis $(v_{1p},v_{2p})$ in $\cc^2$ (whose vectors are projected to $q_{jp}$) uniquely up to 
multiplication of vectors by constants. Recall that in the  basis $(f_{1p}, f_{2p})$ of the 
local solution space at $z_0$ with $f_{kp}(z_0)=v_{kp}$ 
the monodromy matrix is given by formula (\ref{monst}): 
\begin{equation}M=M_{norm,p}C_{1p}^{-1}C_{0p}^{-1}.\label{monstp}\end{equation}
 Here the Stokes matrices 
$C_{0p}$, $C_{1p}$ are unipotent of opposite triangular types (determined by 
the main term matrix of the system $\mcl$ at $p$). Let they be, say, 
upper and lower triangular  respectively with the corresponding triangular elements 
$c_0$ and $c_1$. Recall that the formal monodromy matrix 
$M_{norm,p}$ is diagonal, set $M_{norm,p}=\diag(m_{1p},m_{2p})$; $m_{jp}\neq0$. Then 
\begin{equation}m_{1p}=M_{11}, \ c_0=-M_{12}M_{11}^{-1}, \ m_{2p}c_1=-M_{21},
\label{m1p}\end{equation}
\begin{equation}m_{2p}=M_{22}-m_{2p}c_0c_1=M_{22}-M_{12}M_{21}M_{11}^{-1}, \ c_1=-M_{21}m_{2p}^{-1},\label{m2p}\end{equation}
by (\ref{monstp}). This yields expression for the formal monodromy $M_{norm,p}$ 
and the Stokes matrices in terms of $M$. All the latter matrices depend on choice of the basic 
functions $f_{kp}$, which are uniquely defined by $q_{jp}$ up to multiplication by constant factors. 
These rescalings replace $M_{norm,p}$ and $C_{jp}$ by their conjugates by a diagonal matrix $D_p$,  
and $Q$ by $D_0QD_{\infty}^{-1}$.
This does not change the diagonal matrix $M_{norm,p}$. 
The proposition is proved. 
\end{proof}
\begin{remark} Recall that two global linear systems on the Riemann sphere are 
globally analytically (gauge) equivalent, if and only if they are sent one to the other by 
constant linear change  $Y\mapsto HY$, $H\in\gl_2(\cc)$ (i.e., constant gauge equivalent). 
For simplicity everywhere below whenever we work with global systems on $\oc$ 
we omit the word "analytically" ("constant"), and {\it "gauge equivalence"} means {\it "constant gauge equivalence".}
\end{remark}
\begin{theorem} \label{gaeq} 
Two systems $\mcl_1,\mcl_2\in\mcho$ are gauge equivalent, if and only if they have the same formal normal forms at each singular point and 
the same monodromy--Stokes data. In this case each linear automorphism of the fiber 
$\{ z=z_0\}\simeq\cc^2$ sending the monodromy--Stokes data of one system to 
that of the other system extends to a gauge equivalence of systems. Here both monodromy--Stokes data  
 correspond to the same sectors and path collections. 
\end{theorem} 
\begin{proof} The statement of the theorem holds if one replaces the monodromy--Stokes 
data by collection of Stokes matrices and the transition matrix up to equivalence 
from the above proposition,  see \cite[proposition 2.5, p.319]{JMU1}. 
The collection of Stokes and transition matrices (taken up to the latter equivalence) 
is uniquely determined by the monodromy--Stokes data, by the same proposition. 
 Conversely, the monodromy--Stokes data can be restored from 
 the formal monodromy and Stokes and transition matrices. Namely, 
  the monodromy  matrix $M$ in the basis $(f_{10},f_{20})$ is found from (\ref{monstp}). Let us choose coordinates 
  on $\cc^2$ in which $f_{10}(z_0)=(1,0)$, $f_{20}(z_0)=(0,1)$.  Then one has $q_{10}=[1:0]$, $q_{20}=[0:1]$, 
  and  $q_{1\infty}$, $q_{2\infty}$ are the projections of the columns of 
the transition matrix $Q$. Theorem \ref{gaeq} is proved.
\end{proof}

\section{Isomonodromic deformations and Painlev\'e 3 equation}
Here we introduce general Jimbo's isomonodromic deformations of linear systems 
in $\mcho$, which form a one-dimensional holomorphic foliation of the space $\mcho$ 
(Subsection 3.2). Afterwards we study  its restriction to the so-called 
Jimbo type systems, where isomonodromic 
deformations are described by solutions of Painlev\'e 3 equation (\ref{p3}) (Subsection 3.3). 
In Subsection 3.4 we consider the space of real Jimbo type systems (i.e., defined 
by real matrices) with $R_{21}>0>R_{12}$. We introduce the space $\bbj$ of 
their appropriate normalizations with $R_{21}=-R_{12}>0$ by gauge transformations 
and variable rescalings: the so-called  normalized  $\rr_+$-Jimbo type systems. 
Their space $\bbj$ contains the space $Jos$ of systems (\ref{tty}) and is foliated 
by isomonodromic families obtained from Jimbo deformations by normalizations. 
We show that the family $Jos$ is transversal to the isomonodromic 
foliation of $\bbj$, and it corresponds to poles of order 1 with residue 1 of solutions of 
Painlev\'e equations (\ref{p3}) (Subsection 3.5). 
A background material on isomonodromic deformations is recalled in Subsection 3.1.

\subsection{Isomonodromic deformations: definition and Frobenius integrability 
sufficient condition}
Let us give the following definition of isomonodromic family of linear systems 
in $\mcho$, which is equivalent to the classical definition, by Proposition \ref{proconsm}. 
\begin{definition} A family of systems in $\mcho$  is {\it isomonodromic,} if 
the residue matrices of  formal normal forms at their singular 
points and the monodromy--Stokes data remain constant: independent on the 
parameter of the family. 
\end{definition}
\begin{remark}  If a family of systems in question is continuously parametrized 
by a connected parameter space, then constance of the monodromy--Stokes 
data automatically implies constance of the residue matrices $\wt R_p$ of the formal normal 
forms. Indeed, constance of formal monodromies $M_{norm,p}$ follows by Proposition \ref{proconsm}. 
The formula $M_{norm,p}=\exp(2\pi i\wt R_p)$  implies that the residue matrices 
$\wt R_p$ are uniquely determined by $M_{norm,p}$ up to addition of integer diagonal matrices. 
Hence, they are constant, by continuity and connectivity.
\end{remark}
\begin{theorem}\label{isom_thm}{\rm \cite{JMU1},\cite[Theorem 4.1]{FIKN}}
A holomorphic family of linear systems in $\mcho$ depending on a parameter 
$t$ from a simply connected domain $\mcd\subset\cc$, 
\begin{equation}
Y'=\frac{dY}{dz}=\left(\frac{K_2(t)}{z^2}+\frac{K_1(t)}{z}+K_0(t)\right)Y\label{deform_syst}
\end{equation}
 is isomonodromic if
there is a rational in $z$ (with possible poles only at $z=0,\infty$) and analytic in $t$ matrix differential 1-form 
$\Omega=\Omega(z,t)$ on $\oc\times \mathcal{D}$ 
such that
\begin{equation}
\quad\Omega|_{\text{fixed }t}=\left(\frac{K_2(t)}{z^2}+\frac{K_1(t)}{z}+K_0(t)\right)dz,
\label{frob0}\end{equation}
\begin{equation}
\quad d\Omega=\Omega\wedge \Omega.
\label{frob}\end{equation} 
\end{theorem}

Condition (\ref{frob})
 means that $\Omega$ is integrable in the Frobenius sense. See, e.g., \cite[proof of theorem 13.2]{Bol18}.

An isomonodromic deformation with a scalar parameter is often defined by a system of PDEs $\cite{FIKN}$
\begin{equation}
\left\{\begin{array}{ll}
\frac{\partial Y}{\partial z}=U(z,t)Y\\
\frac{\partial Y}{\partial t}=V(z,t)Y,\\
\end{array}\right.
\label{lapaire}\end{equation} 
where $U(z,t)$, $V(z,t)$ are rational in $z\in\oc$ and analytic in $t\in \mathcal{D}$. In that case, one can take $\Omega=U(z,t)dz+V(z,t)dt$. Then condition  (\ref{frob0}) 
of Theorem $\ref{isom_thm}$ is satisfied if 
\begin{equation*}
U(z,t)=\frac{K_2(t)}{z^2}+\frac{K_1(t)}{z}+K_0(t). 
\end{equation*}
Condition (\ref{frob}) is equivalent to the equation 
\begin{equation}\label{frob2}
[U,V]=UV-VU=\frac{\partial V}{\partial z}-\frac{\partial U}{\partial t}.
\end{equation}

\subsection{General Jimbo's isomonodromic deformation}\label{JID}
In this section, we consider an isomonodromic deformation 
introduced by M.~Jimbo in $\cite[p.1156, (3.11)]{J}$ and describe its integrability condition $(\ref{frob})$. 
The deformation space will be  a simply connected domain 
$\mathcal D\subset\cc^*$ containing $\rr_+$. Though the deformation in \cite{J} was 
written in a seemingly special case, it works in the following  general case. 
We are looking for isomonodromic families  of systems $\mcl(t)\in\mcho$ given by 
 system (\ref{lapaire}) of the following type:
\begin{equation} \begin{cases} \frac{\partial Y}{\partial z}=\left(-\frac{\wt K(t)}{z^2}+\frac{R(t)}z
+N(t)\right)Y \ \ \ :=\mcl(t) \\
\frac{\partial Y}{\partial t}=\frac1{zt}\wt K(t)Y,
\end{cases} \ \ \ \ t\in\mcd.\label{lapair}\end{equation}
After the time variable change $t=e^s$ (which cancels "$t$" in the latter denominator), 
the integrability condition  (\ref{frob2}) takes the form of a system of autonomous polynomial ordinary 
differential equations on matrix coefficients in $\wt K$, $R$, $N$ (here and in what 
follows $[U,V]:=UV-VU$): 
\begin{equation} \begin{cases}\wt K'_s=[R,\wt K]+\wt K\\
R'_s=[\wt K,N]\\
N'_s=0.
\end{cases}\label{isofield}\end{equation}
 In the initial time variable $t$ and 
the new matrix variable $K:=\frac1t\wt K$  system (\ref{isofield}) takes the following 
simplified, though non-autonomous, form:
\begin{equation}\begin{cases} tK'=[R,K]\\ R'=[K,N]\\ N'=0.\end{cases}
\label{intpair}\end{equation}

\begin{remark} \label{rkint} Vector field (\ref{isofield}) 
 is a polynomial vector field on the space $\mcho$ identified with a connected  open 
dense subset in the space $\cc^{12}$ with coordinates being matrix coefficients.
Its complex phase curves form a one-dimensional holomorphic foliation 
of the space $\mcho$ by isomonodromic families. The corresponding 
system (\ref{intpair}) considered as a non-autonomous  
differential equation in the linear-system-valued function $\mcl(t)\in\mcho$ has the following  first integrals: 

- the matrix $N$; 

- the conjugacy class of the matrix $K=\frac1t{\wt K}$; 

- the residue matrices of the formal normal forms of $\mcl(t)$ at $0$ and at $\infty$; 

- the conjugacy class of the monodromy operator of the system $\mcl(t)$. 

Invariance of  residues and monodromy follows from Theorem \ref{isom_thm}. 
Invariance of residues can be also deduced directly from (\ref{intpair}) 
and (\ref{resfo}). 
 \end{remark}
 
 \begin{proposition} \label{pequiv} Vector field (\ref{isofield}) is equivariant under 
 gauge transformations acting on $\mcho$. Its real flow preserves the space of 
 systems in $\mcho$ defined by real matrices.
 \end{proposition}
 The proposition follows immediately from expression (\ref{isofield}).

\subsection{Isomonodromic deformations of special Jimbo type systems}

\begin{definition} A {\it special Jimbo type} linear system is a system of type 
\begin{equation} Y'=\left(-t\frac{\ba}{z^2}+\frac{\bb}z+\left(\begin{array}{cc}
-\frac{1}{2 } &0\\0&0
\end{array}\right)\right)Y,  \ \bb=\left(\begin{matrix} -\ell & *\\ * & 0\end{matrix}\right),\label{djimbo}\end{equation}
such that there exists a matrix $G\in\gl_2(\cc)$ for which 
\begin{equation}
K=G\left(\begin{array}{cc}
\frac{1}{2 } &0\\0&0
\end{array}\right)G^{-1},  \ \ G^{-1}\bb G=\left(\begin{matrix} -\ell & *\\ * & 0\end{matrix}\right).
\label{maa}\end{equation}
In (\ref{djimbo}) and (\ref{maa}) the symbol $*$ stands for an arbitrary unknown matrix element. Here all the matrices  are complex. 
\end{definition}

\begin{remark} \label{rkfnf} The formal normal form at $\infty$  of a system (\ref{djimbo}) is 
$$Y'=\left(\diag(-\frac12,0)+\frac1z\diag(-\ell,0)\right)Y.$$
Condition (\ref{maa}) is equivalent to the statement saying 
 that its formal normal form at $0$ is 
 $$Y'=\left(\frac{t}{z^2}\diag(-\frac12,0)+\frac1z\diag(-\ell,0)\right)Y,$$
 by (\ref{resfo}).
\end{remark}

\begin{proposition} The space of Jimbo type systems (\ref{djimbo}), (\ref{maa}) 
corresponding to a given $\ell\in\cc$ is 
invariant under the  flow of field (\ref{isofield}) and hence, is a union of its phase curves. 
The number $\ell$ is a first integral. 
\end{proposition}
The proposition follows from Remarks \ref{rkint} and \ref{rkfnf}.

We study Jimbo's isomonodromic families of   systems (\ref{djimbo}) given by 
(\ref{intpair}), which take the following form: 
 \begin{equation}\begin{cases} t\ba'=-[\ba,\bb]\\
 \bb'=\left[\ld,\ba\right].\end{cases}\label{isomatr}
 \end{equation}
We denote the upper right entries of $K(t)$ and $R(t)$ by $K_{12}(t)$ and $R_{12}(t)$  respectively. 
%
\begin{theorem}\label{tjp} \cite[pp. 1156--1157]{J} Set 
\begin{equation}y(t)=-\frac{R_{12}(t)}{K_{12}(t)}, \ \tau=\sqrt{t}, \ w(\tau)=\frac{ y(\tau^2)}{\tau}.\label{wyt}\end{equation}
For every  Jimbo's isomonodromic family (\ref{isomatr}) of Jimbo type systems 
(\ref{djimbo}) 
the corresponding function $w(\tau)$  satisfies the 
Painlev\'e~3 equation\footnote{There is another frequently mentioned isomonodromic deformation that leads to the Painlev\'e 3 equation $\cite{JMU2,FIKN}$. 
 }
\begin{equation}
w''=\frac{(w')^2}{w}-\frac{w'}{\tau}+\alpha\frac{w^2}{\tau}+
\beta\frac{1}{\tau}+\gamma w^3+\delta\frac{1}{w},\tag{$P_3(\alpha,\beta,\gamma,\delta)$} \label{pp3}
\end{equation}
whose parameters are expressed via the first integral $\ell$  in the following way 
\begin{equation}\label{parP3_j}
\alpha=-2\ell,\quad \beta=2\ell-2,\quad \gamma=1,\quad \delta=-1:
\end{equation}
\begin{equation} w''=\frac{(w')^2}w-\frac{w'}{\tau}-2\ell\frac{w^2}{\tau}+(2\ell-2)\frac1\tau+w^3-\frac1w.\label{p3}\end{equation}
\end{theorem}
\begin{remark} The deformation considered in Jimbo's paper \cite[pp. 1156--1157]{J} 
was of the type
\begin{equation} \begin{cases}\frac{\partial Y}{\partial x}=\left(-\frac{\wt tA(\wt t)}{x^2}+
\frac{B(\wt t)}x+\left(\begin{matrix} 1 & 0\\ 0 & 0\end{matrix}\right)\right)Y, \ 
A(\wt t)=G(\wt t)\left(\begin{matrix} 1 & 0\\ 0 & 0\end{matrix}\right)G^{-1}(\wt t)\\
\frac{\partial Y}{\partial\wt t}=\frac{A(\wt t)}xY\end{cases}\label{jcase}\end{equation}
with the  (constant) residue matrices of formal normal forms at $0$ and at $\infty$ 
being equal to $\frac12\diag(\theta_{0},-\theta_{0})$  and $-\frac12\diag(\theta_{\infty},-\theta_{\infty})$ respectively. 
Jimbo's family (\ref{jcase}) with $\theta_{\infty}=-\theta_0=\ell$ 
 can be transformed to our family (\ref{djimbo}), 
(\ref{isomatr}) by multiplication of the vector function 
$Y(x)$ by the scalar monomial $x^{-\frac\ell2}$, variable rescaling $z=-2x$,  and parameter rescaling $t=-4\wt t$. 
Our function $w(\tau)$ is obtained from analogous function $y(\sqrt t)$ from 
\cite[p.1157]{J} by  rescaling $w(\tau)=-iy(-\frac i2\tau)$, which transforms  the 
Painlev\'e 3 equation satisfied by $y$ (with  parameters from \cite[p.1157]{J}) to (\ref{p3}).
\end{remark}

\subsection{Isomonodromic families of normalized $\rr_+$-Jimbo  systems}
\begin{definition} An {\it  $\rr_+$-Jimbo type system} is a system 
(\ref{djimbo}) given by real matrices $K$, $R$ satisfying (\ref{maa}) 
 with $R_{12}<0<R_{21}$ and $t>0$.  (The 
 matrix $G$, whose inverse diagonalizes $K$, 
can be chosen real and unimodular.) The space of $\rr_+$-Jimbo type systems 
will be denoted by $\bjjr$. A linear system is of {\it normalized $\rr_+$-Jimbo type}, if 
it has the form 
\begin{equation} Y'_\zeta=\left(-\tau\frac{\ba}{\zeta^2}+\frac{\bb}\zeta+\tau\left(\begin{matrix}-\frac12 & 0\\ 0 & 0\end{matrix}\right)\right)Y, 
 \ \bb=\left(\begin{matrix} -\ell & -R_{21}\\ R_{21} & 0\end{matrix}\right), R_{21}>0,\label{djimbo2}\end{equation}
where $K$ and $R$ satisfy (\ref{maa}). The space 
of normalized $\rr_+$-Jimbo type systems will be denoted by $\bjr$.
\end{definition}
\begin{remark} \label{rknn} The space $\bjjr$ 
of $\rr_+$-Jimbo type systems is a union of 
real isomonodromic families $\mcl(t)$, real phase curves of vector field (\ref{isofield}). 
(Proposition \ref{pequiv}). 
Each $\rr_+$-Jimbo type system can be transformed to a normalized one 
by composition of a unique diagonal gauge transformation 
$(Y_1,Y_2)\mapsto(Y_1,\la Y_2)$, $\la>0$ and the variable change 
$z=\tau\zeta$, $\tau=\sqrt t>0$.
\end{remark}
\begin{example} The space $Jos$ of linear {\it systems of Josephson type}, i.e., 
 family (\ref{tty}), is contained in $\bjr$. Its natural inclusion  to $\bjr$ transforms a system (\ref{tty}) 
 with parameters $(\mu,\ell,\omega)$ to a system (\ref{djimbo2}) with $K=\diag(\frac12,0)$ and 
 the parameters $\tau=2\mu$, 
 $\ell$, $R_{21}=\frac1{2\omega}$.
\end{example}

\begin{proposition} \label{prig} {\bf (Rigidity).} No two 
distinct systems in $\bjr$ are gauge equivalent. 
\end{proposition}
\begin{proof} A gauge equivalence must be diagonal: it should keep the main 
term matrix at $\infty$ diagonal. It should also preserve the equality $\bb_{12}=-\bb_{21}$ and 
the inequality $\bb_{21}>0$. Therefore, it is a constant multiple of identity, and leaves 
the system in question invariant. This proves the proposition.
\end{proof} 

\begin{lemma} \label{lemismon} Let $\mch^{1,0}_{0,\infty}\subset\mcho$ be  the open subset consisting of systems 
with $R_{21},R_{12}\neq0$. The set $\bjr$ is a 4-dimensional 
 real-analytic submanifold in $\mchon$.  It carries a real analytic foliation by isomonodromic families (which will be referred to, as 
 {\bf normalized real isomonodromic families}) given by the differential equation 
 \begin{equation}\begin{cases} R'_\tau=2\tau[K,N]+u[N,R]\\ K'_\tau=\frac2{\tau}[R,K]+u[N,K],\end{cases} 
  \ u=\tau\frac{K_{21}-K_{12}}{R_{21}}, \ 
 N=\left(\begin{matrix}
-\frac{1}{2} & 0 \\ 0 & 0
\end{matrix}\right).\label{isonormal}\end{equation}
Along its solutions the function 
$w(\tau)=-\frac{R_{12}(\tau)}{\tau K_{12}(\tau)}=\frac{R_{21}(\tau)}{\tau K_{12}(\tau)}$ 
satisfies  Painlev\'e 3 equation (\ref{p3}). 
The above foliation will be denoted by $\mathcal F$.
\end{lemma}
\begin{proof} Let us show that the closed subset $\bjr\subset\mchon$ 
is a 4-dimensional submanifold.  For every matrix $K\in\operatorname{Mat}_2(\cc)$ with distinct eigenvalues 
and any their fixed order  $(\la_1,\la_2)$, the matrix $G$ such that 
$G^{-1}KG=\diag(\la_1,\la_2)$  is 
uniquely defined  up to multiplication from the right by a non-degenerate 
diagonal matrix. We will cover 
$\mchon$ by two open subsets $W_1,W_2\subset\mchon$: 
$$W_1:=\{ G_{11}\neq0\}; \ \ W_2:=\{ G_{12}\neq0\}.$$
Let us show that the intersection $\bjr\cap W_1$ is a 4-dimensional submanifold in 
$W_1$. Then we prove the similar statement for the intersection $\bjr\cap W_2$.  For every system in $\bjr\cap W_1$  the corresponding matrix 
$G$ can be normalized as above in a unique way so that 
\begin{equation}\det G=1, \ G_{11}=1; \ G_{22}=1+G_{12}G_{21}.\label{gg1}
\end{equation}
Hence, its  matrix $K$ is defined by two parameters $G_{12}$ and $G_{21}$, and 
the correspondence $(G_{12},G_{21})\mapsto K$ is bijective. 
Let us write the second equation in (\ref{maa})  for a normalized system with  
$\bb_{12}=-\bb_{21}$. It says that the matrix 
$$G^{-1}\bb G 
=\left(\begin{matrix} * & *\\ 
G_{21}\ell+\bb_{21} & G_{21}\bb_{21}\end{matrix}\right)\left(\begin{matrix} 
* & G_{12}\\ * & 1+G_{21}G_{12}\end{matrix}\right)$$
has zero right-lower element. This is the equation  
\begin{equation} G_{21}G_{12}\ell+\bb_{21}G_{12}+G_{21}\bb_{21}+G_{21}^2G_{12}
\bb_{21}=0,\label{rightflow}\end{equation}
which is equivalent to the equation
\begin{equation} G_{12}=-\frac{G_{21}\bb_{21}}{G_{21}\ell+\bb_{21}(1+G_{21}^2)}
\label{g1221}\end{equation}
saying that $G_{12}$ is a known rational function of three independent variables 
$G_{21}$, $\bb_{21}$, $\ell$. The latter equivalence holds outside the exceptional set where 
the numerator and the denominator in (\ref{g1221}) vanish simultaneously. 
Vanishing of the numerator is equivalent to vanishing of $G_{21}$ (since $R_{21}\neq0$, 
by assumption), and in this case the denominator equals $R_{21}\neq0$.  Thus, 
the exceptional set is empty. 
This implies that $W_1\cap\bjr$ is a real 4-dimensional analytic submanifold in $W_1$ 
(the fourth parameter is $\tau=\sqrt t$). 

Let us now prove the above statement for $\bjr\cap W_2$. If $G_{12}\neq0$, then 
we can normalize the matrix $G$ in a unique way so that 
\begin{equation}\det G=1, \ G_{12}=1; \ \ G_{21}=G_{11}G_{22}-1.\label{detg21}
\end{equation}
Then the second equation in (\ref{maa}), which says that the matrix 
$$G^{-1}RG=\left(\begin{matrix} G_{22} & -1\\ -G_{21} & G_{11}\end{matrix}\right)
\left(\begin{matrix}-\ell & -\bb_{21}\\ \bb_{21} & 0\end{matrix}\right)
\left(\begin{matrix} G_{11} & 1\\
G_{21} & G_{22}\end{matrix}\right)$$
has zero right-lower element, is  $(G_{11}G_{22}-1)(\ell+G_{22}\bb_{21})+G_{11}
\bb_{21}=0$, which is equivalent to the equation
$$G_{11}=\frac{\ell+G_{22}R_{21}}{R_{21}(1+G_{22}^2)+\ell G_{22}}.$$
Now it suffices to show that the above numerator and denominator cannot vanish 
simultaneously, as in the previous discussion. Indeed, their vanishing means that 
$G_{22}R_{21}=-\ell$ and $R_{21}-\ell G_{22}+\ell G_{22}=R_{21}=0$, which 
is impossible. The first statement of the lemma is proved. 

The space $\bjjr$ of $\rr_+$-Jimbo type systems is a manifold 
projected to the space $\bjr$ via the diagonal gauge normalizations from Remark \ref{rknn}. 
The projection is an analytic bundle with fiber $\rr_+$, by the same remark and since $\bjr$ is 
a submanifold.  
It sends isomonodromic families in $\bjjr$ given by (\ref{isomatr}) 
to normalized isomonodromic families in $\bjr$. 
Let us find the differential equation describing them. 
Fix a $\tau_0>0$ and matrices $K_0$, $R_0$ defining a system in $\bjr$. Set $t_0=\tau_0^2$. Consider the $\rr_+$-Jimbo type system (\ref{djimbo})  defined by the same matrices. 
Let $\wt\mcl(t)$ be its isomonodromic deformation given by equation (\ref{isomatr}), and let 
 $\wt K(t)$, $\wt R(t)$ denote the corresponding matrices. Let $\mcl(t)$ denote its projection to  $\bjr$, which is given by 
 a  gauge transformation family $(Y_1,Y_2)\mapsto(Y_1,\la(t)Y_2)$ and $z$-variable rescalings 
 $z=\tau\zeta$: 
 the matrices defining the systems $\mcl(t)$ are 
 $$K(t)=\La(t)\wt K(t)\La^{-1}(t),  \ \ R(t)=\La(t)\wt R(t)\La^{-1}(t), \ \La(t)=\diag(1,\la(t)),$$
 $\la(t_0)=1$. Isomonodromicity equation (\ref{isomatr}) on  $\wt\mcl(t)$ at $t=t_0$ yields 
 \begin{equation}\begin{cases}t_0K'_t=-[K, R+t_0\nu\diag(0,1)] \\
 R'=[\diag(\frac12,0),K] + \nu[\diag(0,1),R],\end{cases} \ \ \nu=(\ln\la(t_0))'=\la'(t_0). \label{sysno}\end{equation}
 In the second equation in (\ref{sysno})  $R_{21}'=-R_{12}'$, since $R_{12}\equiv-R_{21}$. 
 This yields 
 $$K_{12}-K_{21}+2\nu(R_{21}-R_{12})=0, \ \ \ \nu=\frac{K_{21}-K_{12}}{4R_{21}}.$$
  Substituting the above formula for $\nu$ to (\ref{sysno}), replacing the matrix $\diag(0,1)$ in the commutators by $\diag(-1,0)=\diag(0,1)-Id$, 
  and changing the variable $t$ to $\tau=\sqrt t$ yields (\ref{isonormal}). The function $w(\tau)$  defined in (\ref{wyt}) for the 
  family $\wt\mcl(t)$ coincides with the analogous function defined for the family $\mcl(t)$, since 
  $\frac{R_{12}}{K_{12}}=\frac{\la^{-1}\wt R_{12}}{\la^{-1}\wt K_{12}}=\frac{\wt R_{12}}{\wt K_{12}}$. It satisfies equation (\ref{p3}), by 
  Theorem \ref{tjp}. Lemma \ref{lemismon} is proved.
\end{proof}

\subsection{Transversality property of Josephson type systems}

\begin{lemma} \label{lempol} Consider an arbitrary system  $\mcl\in Jos$. Let $w(\tau)$ be the germ of solution of Painlev\'e equation (\ref{p3}) defining its real isomonodromic 
deformation in the space 
$\bjr$ at the point $\tau_0$ corresponding to the system $\mcl$. Then 
$w(\tau)$ has first order pole at $\tau_0$ with residue 1. 
Conversely, every system in $\bjr$ 
corresponding to a first order pole $\tau_0>0$ of solution of equation (\ref{p3}) 
with residue 1 lies in $Jos$. 
\end{lemma}

\begin{proof} It is well-known that non-zero singular points of solutions of equation 
(\ref{p3}) are poles of order 1 with residues $\pm1$ \cite[p.158]{GLS}. Let us check that 
systems in $Jos$ correspond to poles with residue 1.  Consider a system in $\bjr$ with $\tau=\tau_0$ 
and $K_{12}=0$ (e.g., a system lying in $Jos$) and its isomonodromic deformation given by  (\ref{isonormal}). 
The upper triangular  term in the second matrix equation in (\ref{isonormal}) has the form
\begin{equation} K_{12}'=\frac2{\tau}R_{12}(K_{22}-K_{11})+O(K_{12}), \text{ as } \tau\to\tau_0; \ \ K_{12}(\tau_0)=0.\label{k12't}
\end{equation}
Therefore, $K_{12}(\tau)=\frac2{\tau_0}R_{12}(\tau_0)(K_{22}(\tau_0)-K_{11}(\tau_0))(\tau-\tau_0)+o(\tau-\tau_0)$, 
\begin{equation}w(\tau)=-\frac{R_{12}(\tau)}{\tau K_{12}(\tau)}=-\frac{1+o(1)}{2(K_{22}(\tau_0)-K_{11}(\tau_0))(\tau-\tau_0)+o(\tau-\tau_0)}.\label{wtark}\end{equation}
If the initial system corresponding to $\tau=\tau_0$ lies in $Jos$, then $K_{22}(\tau_0)-K_{11}(\tau_0)=-\frac12$, hence 
$w(\tau)=\frac1{\tau-\tau_0}(1+o(1))$, and $w$ has simple pole with residue 1 at $\tau_0$. 

Conversely, let $w(\tau)$ have 
a simple pole with residue 1 at $\tau_0$. Then $K_{12}(\tau_0)=0$, and $K_{22}(\tau_0)-K_{11}(\tau_0)=-\frac12$, by (\ref{wtark}). 
Note that the trace of the matrix $K$ is constant and equal to $\frac12$. Hence, $K_{22}(\tau_0)=0$, $K_{11}(\tau_0)=\frac12$. 
Now to show that the system in question lies in $Jos$, it suffices to prove that $K_{21}(\tau_0)=0$. Suppose the contrary: $K_{21}(\tau_0)\neq0$. 
Then the matrix $G$, whose inverse conjugates $K(\tau_0)$ to $\diag(\frac12,0)$, is lower triangular with $G_{21}\neq0$. 
We normalize it by constant factor 
to have $G_{11}=1$. Equation (\ref{rightflow}) together with $G_{12}=0$ yield $G_{21}R_{21}(\tau_0)=0$, while $G_{21}, R_{21}(\tau_0)\neq0$. 
The contradiction thus obtained proves that $K(\tau_0)=\diag(\frac12,0)$ and the system in question lies in $Jos$. Lemma \ref{lempol} is proved.
\end{proof}

\begin{lemma} \label{lcross} {\bf (Key Lemma).} The  submanifold $Jos\subset\bbj$ 
is transversal to the isomonodromic foliation $\bbf$ from Lemma \ref{lemismon}. 
\end{lemma}
\begin{proof} Way 1 of proof. In an isomonodromic family given by (\ref{isonormal}) the derivative  $\ba_{12}'$  is non-zero at $\tau$ corresponding to a system lying in $Jos$, see (\ref{k12't}). 
Hence, this family is transversal to the hypersurface $Jos$. 

Way 2 of proof. Points of the hypersurface $Jos$ correspond to simple poles of solutions of equation 
(\ref{p3})  satisfied along leaves. This together with the 
 fact that a simple pole of an analytic family of 
 functions depends analytically on parameter implies the statement of 
 Lemma \ref{lcross}. 
 \end{proof}

\section{Analytic families of constrictions. Proof of Theorem \ref{famcons0}}

For every linear system $\mcl$ let $M(\mcl)$ denote its monodromy operator. 

In the proof of Theorem \ref{famcons0} we use the following proposition. 
\begin{proposition} \label{trivmc}
A point $(B,A;\omega)$ is a constriction, if and 
only if $A,\omega\neq0$ and the corresponding system (\ref{tty}) has trivial monodromy.  \end{proposition}
 \begin{proof} Proposition 3.2 from \cite{4} states that a point is a constriction, if and only if (\ref{tty}) has 
 projectively trivial monodromy: the monodromy matrix is a scalar multiple of identity.
Another  criterion given by \cite[lemma 3.3]{4} states that a point is a constriction, if and only if $\ell\in\zz$ and 
the germ of linear system (\ref{tty}) at the origin is analytically equivalent to its formal normal form. 
In this case system (\ref{tty}) and its formal normal form have the same monodromy matrices in appropriate bases. 
The  monodromy of the normal form is given by the diagonal matrix $\diag(e^{-2\pi i \ell},1)$, which is 
identity if $\ell\in\zz$. Proposition \ref{trivmc} is proved. 
\end{proof}

\begin{corollary} The  systems (\ref{tty}) corresponding to constrictions lie in the set 
$$\Sigma:=\{\mcl\in \bbj \ | \ M(\mcl)=Id\}.$$ 
\end{corollary}

For every system $\mcl\in\bbj$ let us choose good sectors $S_0$ and $S_1$ 
that contain the upper (respectively, lower) half-plane punctured at $0$, see  
Fig. 5. Consider its monodromy--Stokes data 
$(q_{10},q_{20},q_{1\infty},q_{2\infty}; M)$ defined by the base point $z_0=1\in 
S_{1,2}\subset
S_0\cap S_1$ and trivial paths $\alpha_0,\alpha_\infty\equiv 1$. Set 
\begin{equation}\mcr(\mcl):=
\frac{(q_{10}-q_{1\infty})(q_{20}-q_{2\infty})}{(q_{10}-q_{2\infty})(q_{20}-q_{1\infty})}\in\oc.
\label{transcr}\end{equation}
We will call $\mcr(\mcl)$ the {\it transition cross-ratio} of the system $\mcl$. 
It  depends only on the monodromy--Stokes 
data and not on choice of its representative. 

For the proof of  Theorem \ref{famcons0} we first prove the  following theorem 
and  lemma in Subsections 4.1 and 4.2 respectively.

\begin{theorem} \label{crossub} {\bf (Key Theorem).} 
The subset $\Sigma\subset\bbj$ 
is a two-dimensional analytic submanifold, a union of leaves of 
the real isomonodromic foliation $\bbf$ from Lemma \ref{lemismon}. 
 One has $\ell\in\zz$ for every system in $\Sigma$. 
  The function $\mcr$ is constant on leaves of $\bbf$ in $\bbj$. The restriction $\mcr|_\Sigma$ is real-valued; 
  it is an analytic submersion $\Sigma\to\rp^1=\rr\cup\{\infty\}$. The map 
  $(\mcr,\tau):\Sigma\to\rp^1\times\rr_+$ is a local diffeomorphism. 
\end{theorem}
For every $\ell\in\zz$ by $\Sigma_\ell\subset\Sigma$ we denote the subset of systems with given $\ell$. 

\begin{lemma} \label{constran} 
 For every $\ell\in\zz$ the subset $Constr_{\ell}\subset(\rr_+^2)_{(\mu,\eta)}$, 
$\eta=\omega^{-1}$, is a real-analytic one-dimensional submanifold identified 
with the intersection $Jos\cap\Sigma_\ell$. The restriction of the function $\mcr$ to the latter intersection yields a mapping $Constr_\ell\to\rr\setminus\{0,1\}$ that is 
a local analytic diffeomorphism.
\end{lemma}
Afterwards in Subsection 4.3 we prove the following more precise version of 
the first two statements of Theorem \ref{famcons0}. 
\begin{theorem} \label{famcons1} 1) For every connected component $\mcc$ of the submanifold 
$Constr_{\ell}$ the mapping $\mcr:\mcc\to\rr$ is a diffeomorphism 
onto an interval $I=(a,b)$. 

2) Let $C:=\mcr^{-1}:I\to\mcc$ denote the inverse function. For every $c\in\{ a,b\}\setminus\{0\}$ 
there exists a sequence $x_n\in I$, $x_n\to c$, as $n\to\infty$, such that 
$\eta_n=\eta(C(x_n))\to\infty$, i.e., $\omega(C(x_n))\to0$. 
\end{theorem}
In Subsection 4.4 we prove constance of the rotation number 
and type of constriction on each connected component in $Constr_\ell$ 
 and finish the proof of Theorem \ref{famcons0}.
\subsection{Systems with trivial monodromy. Proof of Theorem \ref{crossub}}

In the proof of Theorem \ref{crossub} we use a series of propositions. 

\begin{proposition} \label{protriv} Every system $\mcl\in\mcho$ with trivial monodromy 
(e.g., every system in $\Sigma$) 
has trivial Stokes matrices and trivial formal monodromies at both singular points 
$0$, $\infty$. In particular, the residue matrices of its formal normal forms 
have integer elements. If $\mcl\in\Sigma$, then one has $\ell\in\zz$.
\end{proposition}
\begin{proof} The proof  repeats  arguments from  \cite[proof of lemma 3.3]{4}. Triviality of the Stokes matrices follows from formulas 
(\ref{m1p}) and (\ref{m2p}). Then $M=M_{norm}=Id$, by (\ref{monst}). Hence, $\ell\in\zz$, 
if $\mcl\in\Sigma$.
\end{proof}
\begin{proposition} \label{pstr} Let in a system $\mcl\in\mcho$, see (\ref{yaz}), 
the matrices $K$, $R$, $N$ be real, and let 
each one of the matrices $K$, $N$ have distinct real eigenvalues. Let 
the Stokes matrices of the system $\mcl$ at $0$ and at $\infty$ be trivial. 
Then the transition cross-ratio $\mcr(\mcl)$ is either real, or infinite. 
\end{proposition}
\begin{proof} Let $f_{1j,p}, f_{2j,p}$ denote the canonical sectorial solution basis 
of the system $\mcl$ at point $p=0,\infty$ in the sector $S_j$, $j=0,1$, see Fig. 5. 
The complex conjugation $\wh\sigma:(Y_1,Y_2;z)\mapsto(\overline Y_1,\overline Y_2; 
\bar z)$ leaves $\mcl$ invariant and sends graphs of its solutions to graphs 
of solutions. Its projectivization $\sigma:(\Phi,z)\mapsto(\overline\Phi,\bar z)$,  $\Phi:=\frac{Y_2}{Y_1}$, permutes  the sectors $S_0$, $S_1$ 
and   graphs of the projectivized solutions 
$$g_{k0,p}:=\pi\circ f_{k0,p}, \ g_{k1,p}:=\pi\circ f_{k1,p}.$$ 
Here $\pi:\cc^2\setminus\{0\}\to\cp^1=\oc_{\Phi}$ is the tautological projection. 
This follows from uniqueness of projectivized sectorial basic solutions 
(Remark \ref{riccun}). Triviality of Stokes matrices implies that $g_{k0,p}=g_{k1,p}$ 
is a global holomorphic $\oc$-valued function on $\cc^*$. 
In particular, for every $z\in\rr$ one has $g_{k0,p}(z)=g_{k1,p}(z)$; hence, 
$(g_{k0,p}(z),z)$ is a fixed point of the involution $\sigma$ and 
$g_{k0,p}(z)\in\rr\cup\{\infty\}$. Finally, 
 $q_{kp}=g_{k0,p}(1)\in\rr\cup\{\infty\}$ for every $k=1,2$ and $p=0,\infty$, and thus, 
$\mcr(\mcl)\in\rr\cup\{\infty\}$. Proposition \ref{pstr} is proved.
\end{proof}

\begin{proposition} \label{prodis} For every system $\mcl\in\Sigma$ the corresponding 
collection of points $q_{kp}$, $k=1,2$, $p=0,\infty$ consists of at least three distinct 
points. One has $q_{1p}\neq q_{2p}$ for every $p=0,\infty$. 
\end{proposition}
\begin{proof} One has $q_{kp}=g_{k0,p}(1)=
\pi\circ f_{k0,p}(1)$, where $f_{10,p}$, $f_{20,p}$ 
form  the canonical basis of solutions of the system in $S_0$. Their linear independence 
implies linear independence of their values at $z=1$, and hence, the inequality 
$q_{1p}\neq q_{2p}$. Let us now prove that among the points $q_{kp}$ there are 
at least three distinct ones. To do this, we use the fact that 
$g_{k,p}(z):=g_{k0,p}(z)=g_{k1,p}(z)$ are two meromorphic functions on  $\cc^*\cup\{ p\}$, 
$p=0,\infty$. 
Meromorphicity on $\cc^*$ follows from Proposition \ref{protriv} and 
the proof of Proposition \ref{pstr}. Meromorphicity at $p$  
follows from  Remark \ref{riccun}.  Suppose the contrary: there are only two distinct points among $q_{kp}$. 
Then $g_{1,0}\equiv g_{k_1,\infty}$, $g_{2,0}\equiv g_{k_2,\infty}$, where 
$(k_1,k_2)$ is some permutation of $(1,2)$. Therefore,  
$g_{1,p}$ and $g_{2,p}$ are meromorphic on $\oc$, by the above 
discussion. Their graphs are disjoint, since so are graphs of their restrictions to $\cc^*$  (being phase curves of the 
Riccati foliation on $\cp^1\times\oc$ defined by $\mcl$), and their values at 
each point $p\in\{0,\infty\}$ are distinct and equal to 
the projections of the eigenlines of the main term matrix at $p$ (Remark \ref{riccun}). 
The main term matrix at infinity being diagonal, one has $g_{1,\infty}(\infty)=[1:0]$, $g_{2,\infty}(\infty)=[0:1]$. 
But graphs of two meromorphic functions on $\oc$ with values in $\cp^1=\oc$ may be disjoint only if the 
functions are constant. Indeed, $\operatorname{H}_2(\oc\times\oc,\zz)=\zz\oplus\zz$ (K\"unneth Formula), 
and the intersection form on the latter homology group is given by the formula 
$<(m_1,n_1), (m_2,n_2)>=m_1n_2+m_2n_1$. See the corresponding background material in 
\cite[chapter 0, section 4]{grh}. The homology class of graph of a rational function $F$ of degree $n$ is 
$(1,n)$; $n>0$, if $F\not\equiv const$. Therefore, if $F\not\equiv const$, then the intersection 
index of its graph with the graph of any rational function is positive. Hence, 
$g_{1,\infty}\equiv[1:0]$, $g_{2,\infty}\equiv[0:1]$, 
and the constant functions  $\Phi(z)\equiv0$, $\Phi(z)\equiv\infty$ are solutions of the Riccati equation corresponding 
to $\mcl$. This implies that the matrices of the system $\mcl$ are diagonal,  which is obviously impossible for a system 
from $\bbj$. The contradiction thus obtained proves the proposition.
\end{proof}

\begin{proposition} \label{procross} For every collection 
$q_0=(q_{10},q_{20},q_{1\infty},q_{2\infty})\in\oc^4$ that has at least three distinct points 
there exists a neighborhood 
$\mathcal V=\mathcal V(q_0)\subset\oc^4$ 
such that two collections in $\mathcal V$ lie in the same $\psl_2(\cc)$-orbit, 
if and only if they have the same cross-ratio.
\end{proposition}
\begin{proof} Fix a neighborhood $\mathcal V$ such that 
 3 distinct points in 
 $q_0$ 
 remain distinct in each collection from $\mathcal V$. 
Let us normalize them by the $\psl_2(\cc)$ action in such a way that these points be 
$0$, $1$, $\infty$: such normalization is unique. 
Then the fourth point is uniquely determined by the cross-ratio. 
\end{proof}

\begin{proof} {\bf of Theorem \ref{crossub}.} 
A system $\mcl\in\bbj$  is uniquely defined 
by the formal invariants $\ell$, $\tau$ and the  monodromy--Stokes data (Theorem 
\ref{gaeq} and Proposition \ref{prig}). Let now $M(\mcl)=Id$. Then the latter data are 
reduced to the $\psl_2(\cc)$-orbit of 
 the collection $(q_{10},q_{20},q_{1\infty},q_{2\infty})$. The latter collection 
  consists of at least three distinct points (Proposition 
\ref{prodis}). Therefore,  {\it each system $\mcl\in\Sigma_\ell$ 
has a neighborhood 
$\mathcal W=\mathcal W(\mcl)\subset\Sigma_\ell$ 
such that two systems in $\mathcal W$ have the same monodromy--Stokes data, 
if and only if the corresponding cross-ratios $\mcr$ are equal.} This follows from 
  Proposition 
\ref{procross} and the above discussion. One has $(\mcr,\tau)(\mcl)\in\rp^1\times\rr$, by 
Propositions \ref{protriv} and \ref{pstr}. This together with the above 
statement on unique local determination by $\mcr$
 imply that the mapping 
$\Pi:(\mcr,\tau):\Sigma\to\rp^1\times\rr$ is  locally injective.

\begin{proposition} \label{pinve} 
For every $\ell\in\zz$ and  $\mcl_0\in\Sigma_\ell$, 
set $T_0=(\mcr_0,\tau_0):=(\mcr,\tau)(\mcl_0)$, there exist neighborghoods 
$V_1=V_1(\mcl_0)\subset\bbj$, $V_2=V_2(T_0)\subset\rp^1\times\rr$ and 
an analytic inverse $g=(\mcr,\tau)^{-1}:V_2\to V_1$ with $g(V_2)=V_1\cap\Sigma_\ell$. 
\end{proposition}
\begin{proof} We have to realize each $T=(\mcr,\tau)$ close to $T_0$ 
by a linear system from $\Sigma$. To  this end, we first realize $T$ by 
an abstract two-dimensional holomorphic  vector bundle over $\oc$ 
with connection. Namely, we take two linear systems defined by 
the given formal normal forms at $0$ and $\infty$ respectively: 
\begin{equation}\mch_0: \ \ Y'=\left(\frac1{z^2}\diag(-\frac{\tau}2,0)+\frac1z\diag(-\ell,0)\right)Y;\label{normy0}
\end{equation}
\begin{equation}\mch_\infty: \ \ Y'=\left(\diag(-\frac{\tau}2,0)+\frac1z\diag(-\ell,0)\right)Y.
\label{normyi}\end{equation}
We consider the following trivial bundles with connections over discs covering
 $\oc$: the bundle $F_0:=\cc^2_{Y^0}\times D_2$ equipped with the system $\mch_0$;  
 the bundle $F_{\infty}:=\cc^2_{Y^{\infty}}\times(\oc\setminus \overline D_{\frac12})$ 
equipped with the system $\mch_{\infty}$. The  bundle realizing $T$ is 
obtained by the following gluing $F_0$ and $F_{\infty}$ over the annulus $\mca:=D_2\setminus\overline D_{\frac12}$. 
 Let $v_1=(1,0)$, $v_2=(0,1)$ denote the standard basis in $\cc^2$. For every $\mcr$ close enough to $\mcr_0$ 
fix a linear isomorphism $\mathbf L_1:\cc^2\to\cc^2$ such  that the tautological projection to $\cp^1=\oc$ 
 of the collection of vectors $\mathbf L_1v_1$, $\mathbf L_1v_2$, $v_1$, $v_2$ has  the given cross-ratio $\mcr$ 
 and $\mathbf L_1$ depends analytically on $\mcr$. 
  Let  $W^0(z)=\diag(e^{\frac{\tau}{2}(\frac1z-1)}z^{-\ell}, 1)$, $W^{\infty}(z)=\diag(e^{-\frac\tau2(z-1)}z^{-\ell}, 1)$ 
  be the standard fundamental matrix solutions of systems $\mch_0$, $\mch_{\infty}$ normalized to be 
  equal to the identity at $z=1$. Set 
  \begin{equation}\mathbf L_z=\mathbf L_{z,\mcr,\tau}=W^{\infty}(z)\mathbf L_1(W^0(z))^{-1}.\label{lzz}\end{equation}
  Let $E=E(\mcr,\tau)$ denote the  disjoint union $F_0\sqcup F_{\infty}$ pasted by the following identification: for every $z\in\mca$ the point 
  $(Y^0,z)\in F_0$ is equivalent to $(Y^{\infty}, z)\in F_{\infty}$, if $Y^{\infty}=\mathbf L_zY^0$. The space $E$ 
  inherits a structure of holomorphic vector bundle over $\oc$ with a well-defined meromorphic connection induced by 
  the formal normal forms $\mch_0$, $\mch_{\infty}$ in the charts $F_0$ and $F_{\infty}$ (which paste 
  together by $\mathbf L_z$ to the same connection  over $\mca$). 
  This connection has two Poincar\'e rank 1 irregular nonresonant singular points at $0$ and $\infty$ 
  where it is analytically equivalent to  $\mch_0$ and $\mch_{\infty}$. 
  Note that the monodromy--Stokes data and the transition cross-ratio are well-defined for bundles with 
  connections as well, provided that the singularities at $0$ and at $\infty$ are irregular nonresonant of Poincar\'e rank 1. 
  The transition cross-ratio of the bundle $E(\mcr,\tau)$ coincides with $\mcr$, by construction. 
  
  Let now $\wh V_2$ be a small ball centered at $T_0=(\mcr_0,\tau_0)$ in the complex product 
  $\oc_\mcr\times\cc_\tau$ 
  (in its local chart centered at $T_0$). 
  Set $\wh E:=\sqcup_{(\mcr,\tau)\in \wh V_2}E(\mcr,\tau)$.   
  This is a holomorphic  vector bundle  over the product $\oc\times\wh V_2$. 
    
  {\bf Claim 1.} {\it The bundle $\wh E$ is trivial, if the ball $\wh V_2$ is small enough.}
  
  \begin{proof} The bundle $E(T_0)$ is trivial, since it has the same monodromy--Stokes data and formal normal 
  forms, as the system $\mcl_0$ (which is a connection on trivial bundle), 
  and by Theorem \ref{gaeq} (which remains valid for bundles 
  with connections). It is glued from two trivial bundles over the domains 
  $D_2$ and $\oc\setminus\overline D_{\frac12}$ by 
  the transition matrix function $\mathbf L_{z,T_0}$. Triviality implies that there exist (and unique) 
  $\gl_2(\cc)$-valued matrix functions $U_0(z)$ and $U_{\infty}(z)$ holomorphic on  $D_2$ and 
  $\oc\setminus\overline D_{\frac12}$ respectively 
  such that 
  $U_0(z)=U_\infty(z)\mathbf L_{z,T_0}$ 
on $\mca$ and $U_\infty(\infty)=Id$. They are holomorphic on  bigger domains $D_3\Supset\overline D_2$, 
  $\oc\setminus\overline D_{\frac13}\Supset\oc\setminus D_{\frac12}$, by the above statement applied 
  to the latter bigger domains and holomorphicity of the transition 
  matrix function $\mathbf L_{z,T_0}$  on $\cc^*$. 
Consider the following new trivializations of the trivial bundles $\cc^2_{Y^0}\times (D_2\times \wh V_2)$ and 
$\cc^2_{Y^{\infty}}\times((\oc\setminus\overline D_{\frac12})\times \wh V_2)$:
  $$\wt Y^0:=U_0(z) Y^0, \ \wt Y^{\infty}:=U_\infty(z)Y^{\infty}.$$
 In the new coordinates $\wt Y^0$ and $\wt Y^{\infty}$ the  fiber identifications gluing $\wh E$ of the above 
 trivial bundles over points 
 $(z,T)\in\mca\times\wh V_2$ become the following: 
 a point $(\wt Y^0,z,T)$ is identified with $(\wt Y^{\infty}, z,T)$, if $M(z,T)\wt Y^0=\wt Y^{\infty}$, where 
  $$M(z,T)=U_\infty(z)\mathbf L_{z,T} U_0^{-1}(z).$$
  Therefore, $\wh E$  can be viewed as the bundle 
  glued from two trivial bundles on $D_2\times\wh V_2$ and $(\oc\setminus\overline D_{\frac12})\times\wh V_2$ 
  by the transition matrix function $M(z,T)$ holomorphic on $\overline\mca\times\wh V_2$. 
  One has $M(z,T_0)=Id$, by construction. Choosing $\wh V_2$ small enough, one can make 
  $M(z,T)$ continuous on $\overline{\mca\times\wh V_2}$ and make 
  the $C^0$-norm $||M(z,T)-Id||$ on $\overline{\mca\times\wh V_2}$  arbitrarily small. Therefore, the bundle 
  $\wh E$ glued by $M(z,T)$ is  "close to  trivial", and hence, is trivial, whenever
   $\wh V_2$ is small enough, by \cite[appendix 3, lemma 1]{Bol18}. (Formally speaking, this lemma should 
   be applied after rescaling the coordinates in the  chart containing $\wh V_2$ in the parameter space 
   to make  $\wh V_2$ the  unit ball.) The claim is proved.
   \end{proof}

Let $V_2\subset\wh V_2$ be the subset of real points of the complex ball $\wh V_2$, 
which is a real planar disk. The claim implies that the family $E(T)|_{T\in V_2}$ yields a 
family of connections on the trivial bundle $\cc^2\times\oc$ depending analytically on the parameter $T\in V_2$. 
They should be linear systems in $\mcho$, 
since the singularities at $0$ and $\infty$ are irregular non-resonant of Poincar\'e rank 1. 
This  yields an analytic map $g:V_2\to V_1$ from
 a  neighborhood $V_2=V_2(T_0)\subset\rp^1\times\rr$ to a domain  
 $V_1\subset\mcho$ such that for every 
 $(\mcr,\tau)\in V_2$ the system $g(\mcr,\tau)$ has trivial monodromy, transition cross-ratio equal to $\mcr$, and is analytically equivalent to  formal normal forms (\ref{normy0}), 
 (\ref{normyi}) near $0$ and $\infty$ respectively. 
 Without loss of generality we consider that $g(T_0)=\mcl_0$, 
 applying a gauge transformation independent on $(\mcr,\tau)$. 
 For every system in $g(V_2)$ the 
 corresponding points $q_{kp}\in\oc_{\Phi}$ from the monodromy--Stokes data given by the base point $z_0=1$ 
 and trivial paths $\alpha_0\equiv\alpha_\infty\equiv1$ lie on the same circle, since their cross-ratio 
 $\mcr$ lies in $\rr\cup\{\infty\}$.
 The latter circle is unique, since there are at least three distinct points $q_{kp}$: this is true for $T=T_0$ 
 (Proposition \ref{prodis}) 
 and remains valid for all $T\in V_2$, provided that $\wh V_2$ is chosen small enough . 
  We normalize the systems in $g(V_2)$ so that the latter circle is the real line, 
  applying an analytic family of gauge transformations depending on $(\mcr,\tau)$. 
  
  {\bf Claim 2.} {\it The systems in $g(V_2)$ are defined by real matrices.}
  
  \begin{proof} The transformation 
  $\wh\sigma:(Y_1,Y_2; z)\mapsto(\overline Y_1,\overline Y_2; 
  \bar z)$ applied to systems in $g(V_2)$ preserves formal normal forms and monodromy--Stokes data, by construction and the above normalization. 
  Therefore, it sends each system in $g(V_2)$ to a system 
  gauge equivalent to it, and the collections of points $q_{kp}$ in the  
  fiber $\oc\times\{1\}$ are the same for both systems. 
  Their  gauge equivalence restricted to the fiber $\cc^2\times\{1\}$ 
  should fix the lines corresponding to $q_{kp}$. Hence, it is identity up to scalar factor, since the number of distinct 
  points $q_{kp}$ is at least three. 
  Therefore,  the systems in question coincide. 
  Thus,  $\wh\sigma$ fixes each system in $g(V_2)$, which means that its matrices are 
  real.
  \end{proof} 
  
  The  main term matrix $N$  at $\infty$ 
  of each system in $g(V_2)$ is real, and its eigenvalues are
   $-\frac\tau2$, $0$. It is close to $\diag(-\frac\tau2,0)$, if 
  $V_2$ is small enough. 
  Therefore, it is conjugated to the diagonal matrix $\diag(-\frac\tau2,0)$ 
  by a real matrix $H$ close to the identity. 
  The matrix $H$ is unique up to left multiplication by a real diagonal matrix. It can be 
  chosen in a unique way so that the  gauge transformation $Y=H^{-1}\wt Y$ makes 
  $R_{21}=-R_{12}>0$. This yields a family of gauge transformations 
  sending systems in $g(V_2)$ to systems lying in $\bbj$, and hence,  in 
  $\Sigma_\ell$ (triviality of monodromy). From now on, the mapping $V_2\to\bbj$ thus 
  constructed will be denoted by $g$. By construction, its image lies in  $\Sigma_\ell$,  
  and  for every $(\mcr',\tau')\in V_2$ the transition 
  cross-ratio $\mcr$ and the formal invariant  $\tau$ of the system 
  $g(\mcr',\tau')$ are respectively $\mcr'$ and $\tau'$. Conversely, 
  every system $\mcl\in\Sigma_\ell$ close enough to $\mcl_0$ has invariants $(\mcr,\tau)$ 
   lying in $V_2$, and hence $\mcl=g(\mcr,\tau)$, by construction, Theorem \ref{gaeq} and Proposition \ref{prig}.  This proves 
   Proposition \ref{pinve}.
\end{proof}

 The mapping $g$  is an immersion, since the projection  
 $\mcl\mapsto(\mcr,\tau)(\mcl)$ is real-analytic and $(\mcr,\tau)\circ g=Id$. This together with Proposition \ref{pinve} 
implies that $\Sigma_\ell$ is a 2-dimensional submanifold, and $(\mcr,\tau):\Sigma_\ell\to\rp^1\times\rr$ 
is a local diffeomorphism. Hence, the projection $\mcr:\Sigma_\ell\to\rp^1$ (which is constant along isomonodromic 
leaves) 
is a submersion.  Theorem \ref{crossub} is proved. 
\end{proof}
  
\subsection{The manifold of constrictions. Proof of Lemma \ref{constran}}
The space of systems (\ref{tty}) with given $\ell$ 
is identified with $(\rr_+)^2_{\mu,\eta}$, $\eta=\omega^{-1}$. They are represented as systems in $Jos\subset\bjr$ with 
parameters $\tau=2\mu$, $\ell$, $R_{21}=\frac{\eta}2$.  
The  constriction subset $Constr_\ell\subset(\rr_+)^2_{\mu,\eta}$ is thus identified with the  
intersection  $Jos\cap\Sigma_\ell$, by  Proposition \ref{trivmc}. The latter 
intersection is transversal, since $\Sigma_\ell$ is a union of leaves of the isomonodromic foliation $\bbf$ and $Jos$ is transversal  to $\bbf$ (Lemma \ref{lcross}). Therefore, $Constr_\ell$ is a one-dimensional 
submanifold  transversal to the isomonodromic foliation on $\Sigma_\ell$.  Hence,  
$\mcr: Constr_\ell\to\rp^1$ is a local 
analytic
diffeomorphism (submersivity of the projection 
$\mcr:\Sigma_\ell\to\rp^1$, see Theorem \ref{crossub}). It remains to show that $\mcr\neq0,1,\infty$ on $Constr_\ell$. 
 \begin{proposition} \label{cons4} For every constriction $(B,A;\omega)$ the collection 
 of points $q_{kp}$ from the monodromy--Stokes data of 
 the corresponding linear system (\ref{tty}) consists of four distinct points. Or equivalently, 
 $\mcr\neq0,1,\infty$.
 \end{proposition}
 \begin{proof} One has
 $q_{1p}\neq q_{2p}$. Hence, the only a priori possible coincidences are the following. 
 
 Case 1): $q_{k0}=q_{k\infty}$ for some $k$. Then the same equality holds for the 
 other $k$, by symmetry $(\Phi,z)\mapsto(\Phi^{-1},z^{-1})$ of the corresponding 
 Riccati equation (\ref{ric}). Thus, the collection of points $q_{kp}$ consists of 
 two distinct points. This contradicts to Proposition \ref{prodis}. 
 
 Case 2): $q_{k0}=q_{(3-k)\infty}$ for some $k$. This means that the transition matrix 
 between the canonical solution base of system (\ref{tty}) at $0$ and the canonical 
 base at $\infty$ taken in inverse order is a triangular matrix. But this contradicts to 
 \cite[theorem 2.10, statement (2.19)]{g18}. 
 
 Finally none of cases 1), 2) is possible. Proposition \ref{cons4} is proved.
 \end{proof}
 Lemma \ref{constran} follows from Proposition \ref{cons4} and the discussion before it.
  
\subsection{Asymptotics and unboundedness. Proof of Theorem \ref{famcons1}}

The subset $Constr_{\ell}\subset(\rr_+^2)_{\mu,\eta}$ 
is a submanifold that admits a locally diffeomorphic projection $\mcr$ 
to $\rr\setminus\{0,1\}$  (Lemma \ref{constran}). This implies that it has no compact components, since no compact 
component can admit a locally diffeomorphic mapping to $\rr$. 
Therefore, each its component $\mcc$ 
is diffeomorphic to some interval $I=(a,b)$ with coordinate $x:=\mcr$.  This implies 
the first statement of Theorem \ref{famcons1}. 
To prove its second statement, the existence of a sequence $x_n\to c$ with $\eta(C(x_n))\to\infty$ for $c\in\{ a,b\}\setminus\{0\}$, 
$C=\mcr^{-1}$, 
 we will 

-  use the following Klimenko--Romaskevich Bessel asymptotic result \cite{RK} 
to show that boundedness of $\eta$ implies boundedness of $\mu$;

- prove that $(\mu,\eta)(x_n)$ cannot converge to $(0,0)$, by using solution of variational equation to (\ref{jostor}) and studying  local parametrization of the 
 analytic 
 subset 
 in $\rr^2$ containing $Constr_\ell$;
 
-  show that  if $\eta(x_n)\to0$, then $c=\lim x_n=0$.

Let us recall that the boundary of the phase-lock area $L_r$ consists of two curves $\partial L_{r,0}$, $\partial L_{r,\pi}$, corresponding to those parameter values, for which the Poincar\'e map of the corresponding dynamical system (\ref{josvec}) acting on the circle $\{\tau=0\}$ has fixed points $0$ and $\pi$ respectively. These  are graphs 
$$\partial L_{r,\alpha}=\{ B= G_{r,\alpha}(A)\}, \ G_{r,\alpha} \text{ are analytic functions on } \rr; \  \alpha=0,\pi.$$
 \begin{theorem} \cite[theorem 2]{RK}. There exist positive constants 
 $\xi_1$, $\xi_2$, $K_1$, $K_2$, $K_3$ such that the following statement 
 holds. Let $r\in\zz$, $A$, $\omega>0$ be such that 
 \begin{equation}|r\omega|+1\leq \xi_1\sqrt{A\omega}, \ \ A\geq \xi_2\omega.
 \label{preeq}\end{equation}
  Let $J_r$ denote the $r$-th Bessel function. Then 
  \begin{equation}\left|\frac1\omega G_{r,0}(A)-r+\frac1{\omega} 
  J_r\left(-\frac A{\omega}\right)\right|\leq\frac1A\left(K_1+\frac{K_2}{\omega^3}+
  K_3\ln\left(\frac A{\omega}\right)\right),\label{bbo}\end{equation}
  \begin{equation}\left|\frac1\omega G_{r,\pi}(A)-r-\frac1{\omega} 
  J_r\left(-\frac A{\omega}\right)\right|\leq\frac1A\left(K_1+\frac{K_2}{\omega^3}+
  K_3\ln\left(\frac A{\omega}\right)\right).\label{bbo2}\end{equation}
 \end{theorem}

 \begin{proposition} \label{boundcom} Fix an $\ell\in\zz$. 
 For every $\eta_0>0$  the intersection 
 $$Constr_{\ell,\eta_0}:=Constr_\ell\cap\{0<\eta<\eta_0\}\subset\rr_+\times(0,\eta_0)$$ 
 is a one-dimensional analytic submanifold with infinitely many connected components, 
 and each component is bounded.  
 \end{proposition}
 \begin{proof}
 Let $u_1<u_2<\dots$ denote the sequence of points of local maxima of the modulus $|J_\ell(-u)|$, which tends to 
 plus
 infinity.  
 
 {\bf Claim.} {\it Fix an $\eta_0>0$ and an $\ell\in\zz$. 
 For every $k\in\nn$ large enough (dependently 
 on $\eta_0$ and $\ell$) the interval $\wh I_k:=\{ \mu=\frac{u_k}2\}\times(0,\eta_0)$ does not intersect the constriction set $Constr_\ell$.}
 
 \begin{proof} In the coordinates $(\mu,\eta)$ inequalities (\ref{preeq}), 
 (\ref{bbo}) and (\ref{bbo2}) 
 can be rewritten for $r=\ell$ respectively as 
 \begin{equation} |\frac{\ell}\eta|+1\leq\frac{\xi_1}\eta\sqrt{2\mu}, \ \mu\geq\frac{\xi_2}2,
 \label{preeqn}\end{equation}
 \begin{equation} \left|\eta G_{\ell,0}(\frac{2\mu}\eta)-\ell+ \eta 
  J_\ell(-2\mu)\right|\leq\frac{\eta}{2\mu}\left(K_1+K_2\eta^3+
  K_3\ln(2\mu)\right),\label{bbon*}\end{equation}
  \begin{equation} \left|\eta G_{\ell,\pi}(\frac{2\mu}\eta)-\ell-\eta 
  J_\ell(-2\mu)\right|\leq\frac{\eta}{2\mu}\left(K_1+K_2\eta^3+
  K_3\ln(2\mu)\right).\label{bbon2*}\end{equation}
  For every $k$ large enough the value $\mu=\frac{u_k}2$ satisfies inequality (\ref{preeqn}) 
  for all $\eta\in(0,\eta_0)$. Substituting $\mu=\frac{u_k}2$ to the right-hand side in (\ref{bbon*}) 
 transforms it to a sequence of functions of $\eta\in(0,\eta_0)$ 
 with uniform asymptotics $\eta(O(\frac1{u_k})+O(\frac{\ln{u_k}}{u_k}))$, as $k\to\infty$. 
 The values $|J_\ell(-u_k)|$ are known to behave asymptotically as $\frac1{\sqrt u_k}$ 
 (up to a known constant factor). 
 Therefore, they dominate the  
 right-hand sides in (\ref{bbon*}) and (\ref{bbon2*}). This together with (\ref{bbon*}),  (\ref{bbon2*}) implies that 
 for every $k$ large enough, 
 set $A_k=\frac{u_k}{\eta}$, 
  $$G_{\ell,0}(A_k)=\ell\omega-J_\ell(-u_k)(1+o(1)), \ G_{\ell,\pi}(A_k)=\ell\omega+J_\ell(-u_k)(1+o(1)),$$
  as $k\to\infty$, uniformly in $\omega>\omega_0=\eta_0^{-1}$. 
  This implies that $\ell\omega$ lies between $G_{\ell,0}(A_k)$ and $G_{\ell,\pi}(A_k)$ for large $k$. 
    Therefore, for every $k$ large enough and every $\omega>\omega_0$ 
    the point $(\ell\omega, A_k;\omega)$ lies in the interior of the phase-lock area $L_\ell$, and hence, is not a constriction. 
    This proves the claim.
 \end{proof}
 
 For every point  $q\in Constr_\ell$ and every $k$ large enough dependently on $q$ 
 the connected component of the point $q$ in $Constr_\ell$ is separated from 
 infinity by the segment $\wh I_k$ from the above claim. This proves boundedness 
 of connected components.  
Infiniteness of number of connected components follows from their boundedness and 
 the fact that for every given $\ell\in\zz$ and $\omega>0$ the 
 vertical line $\La_{\ell}=\{ B=\omega\ell\}$ contains an infinite sequence of constrictions 
 with $A$-ordinates converging to $+\infty$; the latter fact follows from  
 \cite[the discussion after definition 2]{RK} and \cite[theorem 1.2]{4}. 
 This finishes the proof of Proposition \ref{boundcom}. 
 \end{proof}

\begin{lemma} \label{noco0}
 For every $\ell\in\zz$ the subset $Constr_\ell\subset\rr_+^2$ does not accumulate to zero. 
That is, there exists no sequence of constrictions $(B_k,A_k;\omega_k)$ with 
$B_k=\ell\omega_k$ where $\omega_k\to+\infty$ and $\mu_k:=\frac{A_k}{2\omega_k}\to0$, 
as $k\to\infty$. 
\end{lemma} 
For the proof of Lemma \ref{noco0} (given below) let us recall that the first 
equation in system (\ref{josvec}) describing model of Josephson junction 
takes the following form in the  new parameters $\mu$ and $\eta$: 
\begin{equation}\dot\theta:=\frac{d\theta}{d\tau}=\eta\cos\theta+\ell+2\mu\cos\tau, \ 
\ \eta=\omega^{-1}, \ \mu=\frac A{2\omega}.\label{jn}\end{equation}
The constrictions correspond to those values of $(\mu,\eta)\in\rr_+^2$  for which 
the time $2\pi$ flow map 
$$h=h^{2\pi}=h_{\mu,\eta}$$
 of equation (\ref{jn}) acting on the $\theta$-circle $\{\tau=0\}$ 
is identity (Proposition \ref{propoinc}).  For the proof of the lemma it suffices to show that $(0,0)$ is an isolated point 
 in the analytic subset $\{ h_{\mu,\eta}=Id\}\subset\rr^2_{\mu,\eta}$. 
 This is done by using the following formulas for a solution $\theta(\tau)$ 
 of (\ref{jn}) and its derivatives in parameters for $\eta=0$. 
 \begin{proposition} Let $\theta(\tau,\theta_0;\mu,\eta)$ 
denote the  solution of equation (\ref{jn})
 with initial condition $\theta(0)=\theta_0$. One has the following formulas for the 
 solution and its partial derivatives in the parameters $(\mu,\eta)$: 
 \begin{equation}\theta(\tau,\theta_0;\mu,0)=\theta_0+\ell\tau+2\mu\sin\tau, \ \ h_{\mu,0}=Id,
 \label{labsol}\end{equation}
 \begin{equation}\theta(\tau,\theta_0):=\theta(\tau,\theta_0,0,0)=\theta_0+\ell\tau,\label{solao}
 \end{equation}
 \begin{equation} \theta'_\eta=\frac{\partial\theta}{\partial\eta}=\frac1\ell(\sin(\theta_0+\ell\tau)-\sin\theta_0) \text{ at the locus } 
 \{ \mu=\eta=0\},\label{phiu1}\end{equation}
 \begin{equation} 
 \theta'_\mu=2\sin\tau, \ \ \theta^{(k)}_{\mu\dots\mu}=\frac{\partial ^k\theta}{\partial\mu^k}=0 \text{ for } k\geq2 \text{ at  the locus } \{ \eta=0\}.
 \label{phiu}\end{equation}
 The following two formulas 
 hold
  at the locus $\{ \mu=\eta=0\}$: 
 \begin{equation} \frac{\partial^2\theta}{\partial\eta^2}=-
 \frac{\tau}\ell+\frac1{2\ell^2}
 (\sin2(\theta_0+\ell\tau)-\sin2\theta_0)-\frac2{\ell^2}\sin\theta_0(\cos(\theta_0+\ell\tau)-
 \cos\theta_0);\label{d2phi}\end{equation}
 \begin{equation} \frac{\partial^{k+1}\dot\theta}{\partial\eta\partial\mu^k}=2^ks_k(\theta_0+\ell\tau)
 \sin^k\tau, \text{ where } s_k(y)=\begin{cases} (-1)^{\frac k2}\cos y \text{ for even } k\\
 (-1)^{\frac{k+1}2}\sin y \text{ for odd } k.\end{cases}\label{highphi}
 \end{equation}
 Here "dot" is the derivative in $\tau$. 
 \end{proposition}
 \begin{proof} Formulas (\ref{labsol}) and (\ref{solao}) are obvious. 
  The equation in variations for the derivative $\theta'_\eta$ is 
  \begin{equation}\dot\theta'_\eta=\cos(\theta_0+\ell\tau+2\mu\sin\tau)+O(\eta), \ 
  \text{ as } \eta\to0.\label{ouu}
  \end{equation}
  The derivative $\theta'_\eta$ is a solution of (\ref{ouu}) vanishing at $\tau=0$. Therefore, 
  for $\mu=\eta=0$ it  is given by (\ref{phiu1}).
  Formulas (\ref{phiu}) follow immediately by differentiating (\ref{labsol}) in $\mu$. 
  Formula (\ref{highphi}) follows by differentiating (\ref{ouu}) in $\mu$ and taking 
  the value thus obtained at $\mu=\eta=0$. It remains to prove (\ref{d2phi}). 
  Differentiating equation (\ref{jn}) in $\eta$ twice at $\eta=\mu=0$ and substituting 
  (\ref{phiu1}) yields the following differential equation for the derivative $\theta''_{\eta\eta}=\frac{\partial^2\theta}{\partial\eta^2}$: 
  $$\dot\theta''_{\eta\eta}=-2\sin\theta\theta'_\eta=
  -\frac2\ell\sin(\theta_0+\ell\tau)(\sin(\theta_0+\ell\tau)-\sin\theta_0).$$
 Taking the primitive in $\tau$ of the right-hand side that vanishes at $\tau=0$ yields 
 (\ref{d2phi}). The proposition is proved.
 \end{proof}
 \begin{proposition} \label{propoin} Let $\ell\in\nn$. The Taylor expansion in $(\mu,\eta)$ 
 of the time $2\pi$ flow map $h_{\mu,\eta}(\theta_0)$ takes the form 
 \begin{equation} h_{\mu,\eta}(\theta_0)=\theta_0-\frac{\pi}{\ell}\eta^2+g(\theta_0)\eta\mu^\ell+
 o(\eta^2)+o(\eta\mu^\ell), \text{ as } \mu,\eta\to0,\label{tayu}\end{equation}
 where $g(\theta_0)$ is a non-constant function of  $\theta_0$ that is equal to either 
 $\sin\theta_0$, or $\cos\theta_0$ up to non-zero constant factor. 
 \end{proposition}
  \begin{proof} The Taylor coefficient of the difference $h_{\mu,\eta}(\theta_0)-\theta_0$ 
  at $\mu^k\eta^m$  at  the locus $\mu=\eta=0$ equals  
  $\frac1{k!m!}\frac{\partial^{k+m}\theta}{\partial \mu^k\partial\eta^m}(\tau,\theta_0; 0,0)$ 
  where $\tau=2\pi$. The latter 
  derivative at $\tau=2\pi$ vanishes for $(k,m)=(0,1), (n,0)$, by (\ref{phiu1}), (\ref{phiu});  
  it equals $-\frac{2\pi}\ell$ for $(k,m)=(0,2)$, by (\ref{d2phi}).
    
  {\bf Claim.} {\it The above $(k,1)$-th derivative is $2\pi$-periodic in $\tau$, if 
  $1\leq k\leq\ell-1$.  If $k=\ell$, it is equal to $g(\theta_0)\tau$ plus a $2\pi$-periodic 
  function; here $g(\theta_0)$ has the same type, as in Proposition \ref{propoin}.} 
  
  \begin{proof} The $(k,1)$-th derivative 
   equals the primitive of the right-hand side in (\ref{highphi}). 
   The latter right-hand side is 
  a linear combination of values of 
  $\sin$ (or $\cos$) of $\theta_0+r\tau$, $r\in\zz$, $\ell-k\leq r\leq
  \ell+k$. Moreover, the coefficient at the "lower term", the 
  $\sin$ ($\cos$) of $\theta_0+(\ell-k)\tau$, is non-zero, by elementary trigonometry. 
  Therefore, the primitive of the latter right-hand side in (\ref{highphi}) is a linear 
  combination of $\cos$ ($\sin$) of the above arguments, except for a possible term 
  with $r=0$, which is $\tau\cos\theta_0$ ($\tau\sin\theta_0$) up to constant factor. 
  For  $k\leq\ell-1$ the latter term does not arise. For $k=\ell$ it arises with  a non-zero 
  constant factor, by the above  discussion. The claim is proved.
  \end{proof}
  
  One has $\frac{\partial^{k+m}\theta}{\partial \mu^k\partial\eta^m}(0,\theta_0; 0,0)=0$, 
  by definition. This together with the above claim and  discussion implies the 
  statements of Proposition \ref{propoin}.
  \end{proof}
  
  \begin{proof} {\bf of Lemma \ref{noco0}.} Suppose the contrary: the set $Constr_\ell$ 
  accumulates to zero. Recall that it lies in the ambient analytic 
  subset
   in $\rr\times\rr$ 
  defined by the equation $h_{\mu,\eta}=Id$. (The Poincar\'e map $h_{\mu,\eta}$ 
  is M\"obius, being  the restriction to $S^1=\{|\Phi|=1\}$ of the 
  monodromy map of Riccati equation (\ref{ric});  the equation 
  $h_{\mu,\eta}=Id$ is written in the Lie group $\operatorname{Aut}(D_1)\simeq\psl_2(\rr)$.) 
   Therefore, the latter 
   analytic
    subset 
    contains  an irreducible germ of analytic curve $\Gamma$ at 0 with 
  $\eta|_\Gamma, \mu|_\Gamma\not\equiv0$, since $\eta,\mu\neq0$ on $Constr_\ell$. 
  Hence, $\Gamma$ can be considered as a graph of (may be singular) analytic function 
  $\mu=c\eta^\alpha(1+o(1))$, $\alpha>0$, $c\neq0$. Substituting the latter expression 
  for $\mu$ to the Taylor formula (\ref{tayu}) yields
  \begin{equation} h_{\mu,\eta}(\theta_0)=\theta_0-\frac{\pi}{\ell}\eta^2+c^\ell g(\theta_0)
  \eta^{1+\ell\alpha}
  +o(\eta^2)+o(\eta^{1+\ell\alpha}).\label{hum}\end{equation}
  The right-hand side in (\ref{hum}) should be identically equal to $\theta_0$, 
  since $h_{\mu,\eta}=Id$ for $(\mu,\eta)\in\Gamma$. This together with (\ref{hum}) implies 
  that its  second and third terms should cancel out: $1+\ell\alpha=2$ and  
  $g(\theta_0)\equiv c^{-\ell}\frac{\pi}\ell$. But we know that $g(\theta_0)\not\equiv const$. The contradiction thus obtained proves Lemma \ref{noco0}.
  \end{proof} 
  
  \begin{proof} {\bf of the second statement of Theorem \ref{famcons1}.} 
 Suppose the contrary: as $x\in I$ tends to a non-zero endpoint $c\in\{ a,b\}$ 
 of the interval $I$, the function $\eta=\eta(C(x))$ is bounded  from above. But then 
 $\mu(C(x))$ is also bounded from above, by Proposition \ref{boundcom}. 
 The component $\mcc$ being a non-compact submanifold in $\rr_+^2$, it should go to "infinity" 
 (to the boundary), as $x\to c$. Therefore, there exists a sequence $x_k\to c$ 
 such that $C(x_k)\to C^*\in\{\mu\eta=0 \ | \ \mu,\eta\geq0\}$ (boundedness of  $\mu$ and 
   $\eta$). One has $C^*\neq(0,0)$, by Lemma \ref{noco0}.  Let show that 
 two other possible cases treated below are impossible.
 
 Case 1):  $C^*=(0,\eta)$, $\eta>0$.  Then the  equations (\ref{jn}) corresponding to 
 $C(x_k)=(\mu_k,\eta_k)$ have identity Poincar\'e map and limit to the  equation 
  \begin{equation}\frac{d\theta}{d\tau}=\eta\cos\theta+\ell,\label{josl0}\end{equation}
 which   should also have identity Poincar\'e map. In the case, when $\ell=0$, 
 this is obviously impossible, since the dynamical system on $\tt^2$ given by (\ref{josl0}) is 
 hyperbolic with an attracting periodic orbit $\theta\equiv\frac{\pi}2$. In the case, when 
 $\ell\in\nn$, the rotation number of the above system  is an integer non-negative 
 number $\rho<\ell$. This follows from the fact that the $\ell$-th phase-lock area $L_\ell$ 
  intersects 
 the $B$-axis $\{ A=0\}=\{\mu=0\}$ at the so-called growth point with known abscissa 
 $B(\ell,\omega)=\sqrt{\ell^2\omega^2+1}$, $\omega=\eta^{-1}$, see 
 \cite[corollary 3]{buch1}, 
  while $C(x_k)$ correspond to constrictions with 
  the abscissas 
  $\ell\omega_k<B(\ell,\omega_k)$. 
  Therefore, the points $C(x_k)\in Constr_\ell$ 
  also  correspond to the same rotation number $\rho<\ell$, whenever $k$ is large 
  enough (continuity of the 
  rotation number function and its integer-valuedness on the 
  points $(B_k,A_k;\omega_k)$ corresponding to $C(x_k)$). Thus, the points $C(x_k)$ 
  correspond to constrictions lying on the axis $\La_\ell=\{ B=\ell\omega\}$ with 
  non-negative rotation number $\rho<\ell$. But all the constrictions lying in 
  $\La_\ell$ should correspond to rotation numbers no less than $\ell$, by 
    \cite[theorem 1.2]{4}. The contradiction thus obtained shows that the case under consideration is impossible. 
    
    Case 2): $C^*=(\mu,0)$, $\mu>0$. Then the linear system (\ref{tty}) corresponding 
    to $C^*$ is diagonal, and hence, has zero cross-ratio $\mcr=x$. Hence, 
    the cross-ratios $x_k$ corresponding to $C(x_k)$ tend to zero. But their limit 
    $c$ is non-zero, by assumption. The contradiction thus obtained shows that 
    Case 2) is also impossible and finishes the proof of Theorem \ref{famcons1}. 
    \end{proof}

\subsection{Constance of rotation number and type. Proof  of Theorem \ref{famcons0}} 
Without loss of generality we can and will restrict 
ourselves 
to the case, when $\ell\in\zz_{\geq0}$, 
due to symmetry. 

All the statements 
    of Theorem \ref{famcons0} except for the last one  follow immediately from 
    Theorem \ref{famcons1}. Let us prove  its last statement: constance of rotation 
    number and type. Fix a connected component $\mcc$ of the 
    manifold $Constr_\ell$. Constance of the rotation number function 
    on $\mcc$ follows from its continuity and integer-valuedness. Constance 
    of the constriction type  is obvious for $\ell=0$: the $A$-axis lies in $L_0$, 
    hence, all its constrictions are positive. Thus, everywhere below we consider that 
    $\ell\in\nn$ (symmetry). To prove constance of type, we 
    use the following proposition. To state it, 
    let us recall that for every $\omega>0$ a {\it generalized simple intersection} 
    is a point $(B,A;\omega)$ with $\ell=\frac B{\omega}\in\zz$, $A\neq0$ and 
    $\rho=\rho(B,A;\omega)\equiv \ell(\mod2\zz)$  that lies in the boundary 
    of the phase-lock area $L_\rho=L_\rho(\omega)$ and that is not a constriction 
    \cite[definition 1.16]{gn19}; they exist only for $\ell\neq0$. 
    
     \begin{proposition} \label{nopint} 
 A constriction $C=(B,A;\omega)$ cannot be a limit of generalized simple intersections 
 with some $\omega_k\to\omega$.
\end{proposition}
\begin{proof} One has $\ell=\frac{B}{\omega}\in\zz$. 
Without loss of generality we can 
and will consider that $\ell\geq1$ (symmetry). Generalized 
simple intersections  correspond to 
special double confluent Heun equations (\ref{heun2}) having polynomial solution 
\cite[theorem 1.15]{bg2}. If, to the contrary,  the constriction $C$ 
were a limit of generalized simple intersections, then it would also corresponds to 
equation (\ref{heun2}) having polynomial solution. But this is impossible, 
by \cite[theorems 3.3 and 3.10]{bg}. The contradiction thus obtained proves 
the proposition.
\end{proof}

Let a constriction $C(x_0)\in Constr_\ell$ be negative. Let us show that for every 
$x$ close to $x_0$ 
the constriction $C(x)=(B(x),A(x); \omega(x))$ is also negative: the case of positive constriction is treated analogously. (Note that each 
constriction is either positive, or negative, by \cite[theorem 1.8]{g18}.) 
Let $\rho\in\zz$ denote the rotation number 
of the constriction $C(x_0)$. Set $\omega_0:=\omega(x_0)$, 
$\La_\ell(\omega):=\{ B=\ell\omega\}\subset\rr^2_{B,A}$. For every $r>0$ let  
 $U_{r}\subset\rr^2$ denote the disk of radius $r$ 
  centered at $(B(x_0),A(x_0))$. Fix an $r>0$ such that  $(B(x_0),A(x_0))$ is the only point of intersection $\partial L_\rho(\omega_0)\cap\La_\ell(\omega_0)$   lying in  $U_{2r}$. Such an $r$ exists, since the latter intersection is discrete, by analyticity of 
  the graphs $\partial L_{\rho,0}$, $\partial L_{\rho,\pi}$ forming $\partial L_\rho$, and since none of these graphs is a vertical line.

Case 1). Let  for every $x$ close enough to $x_0$ the point $(B(x),A(x))$  be the only 
point of intersection $\partial L_\rho(\omega(x))\cap\La_\ell(\omega(x))$ lying in $U_r$. Then all the above 
constrictions $C(x)$ have the same, negative type, by definition.

Case 2). Let now the unique point $C(x_0)$ of intersection $\partial L_{\rho}(\omega(x_0))\cap \La_\ell(\omega(x_0))$ 
split into several intersection points, as we perturb $x=x_0$ slightly.  
Then all these points are constrictions, by Proposition \ref{nopint} and since $\rho\equiv\ell(\mod 2)$, see \cite[theorem 3.17]{4}. Their number is 
finite,  and they split the intersection   
 $\La_\ell(\omega(x))\cap U_r$ into a finite number of intervals. 
Any two adjacent division intervals either both lie outside the phase-lock 
area $L_\rho(\omega(x))$, or both lie inside $L_\rho(\omega(x))$, since the constriction separating them is either 
negative, or positive (see \cite[theorem 1.8]{g18} and Remark \ref{posneg}). 
The division intervals adjacent to $\partial U_r$ 
should lie outside, since this is true for $x=x_0$ and by continuity. 
Therefore, all the above intervals   lie outside. 
Hence, all the  constrictions bounding them are negative.  
 Theorem \ref{famcons0} is proved.

\section{Slow-fast methods. Absence of ghost constrictions for small $\omega$}

We prove Theorem \ref{noq} in Subsections 5.1--5.5. 
Theorem \ref{noghost} 
will be proved in Subsection 5.6. 

It suffices to prove absence of ghost constrictions with $B=\omega\ell$, $\ell\in\nn$, and $A>0$, 
by symmetry and since the constrictions with $\ell=0$ are positive and lie in $L_0$. Thus, everywhere below 
without loss of generality we consider that $\ell\in\nn$. 
It is already known that 
\begin{equation} \text{there are no constrictions  } (\ell\omega,A) \text{ with } A\in(0,1-\ell\omega],
\label{nocons}\end{equation}
since all the points $(B,A)$ with $|B|+|A|\leq1$ lie in 
 the phase-lock area $L_0$ \cite[proposition 5.22]{bg2}, and all the constrictions in $L_0$ lie in the $A$-axis. 
  
First in Subsection 5.1 for small $\omega$ we prove absence of ghost constrictions in the semiaxis 
 $\La_{\ell}$ with ordinate greater than 
 \begin{equation}A_{\frac12}=A_{\frac12}(\omega):=1+(\ell-\frac12)\omega.\label{a12om}\end{equation}
Their absence follows from results of \cite{g18, gn19}, which imply that the whole  ray 
$\{\ell\omega\}\times[A_{\frac12},+\infty)$ lies in the phase-lock area
  $L_\ell$ 
  for small $\omega$. 
   In Subsection 5.5 we show that there are  
 no constrictions $(\ell\omega,A)$  with $A\in(1-\ell\omega,A_{\frac12}(\omega))$, whenever $\omega$ is small enough. This is done by 
 studying family of systems (\ref{josvec}) modeling Josephson junction as a 
 slow-fast family of dynamical systems, with small $\omega$ and 
 $A=A_{\alpha}(\omega)=1+(\ell-\alpha)\omega+o(\omega)$. 
 The corresponding background material 
 on slow-fast systems is given in Subsection 5.2.  The key lemma used in the proof 
 of absence of the above-mentioned constrictions is the Monotonicity Lemma 
 stated and proved in Subsection 5.4. 
 It concerns a pair of slow-fast families 
 (\ref{josvec}) corresponding to two families of ordinates $A_{\alpha_1}$ and 
 $A_{\alpha_2}$ as above with 
 $0<\alpha_1<\alpha_2$. It deals with their Poincar\'e maps of the 
 cross-section $\{\tau=0\}$ lifted to the universal cover as maps of the  line 
 $\{\tau=0\}$ to $\{\tau=2\pi\}$.   The Monotonicity Lemma  states that 
 the Poincar\'e map of  the system 
 (\ref{josvec}) corresponding to $A_{\alpha_2}$ is less than the analogous 
 Poincar\'e map for  $A_{\alpha_1}$, whenever $\omega$ 
 is small enough. Its  proof  is based on the Comparison Lemma 
 on arrangement and disjointness of slow flowboxes of the systems in question 
 (stated and proved in Subsection 5.3). 
 
 \subsection{Absence of ghost constrictions with big ordinates}
 \begin{lemma} \label{quehigh} For every 
 $\ell\in\nn$ 
 and every $\omega>0$ small enough 
 dependently on $\ell$ the ray 
 $$\La_{\ell,\frac12}:=\La_\ell\cap\{A\geq A_{\frac12}\}\subset\La_\ell,$$ 
 see (\ref{a12om}), lies in the phase-lock area with the rotation number $\ell$. It contains no ghost 
 constrictions.
 \end{lemma}
 \begin{proof} The intersection of the phase-lock area $L_\ell$ with 
 the semiaxis $\La_{\ell}^+:=\La_\ell\cap\{ A>0\}$ contains a ray $S\ell$ bounded 
 by a point $\mcp_\ell$, the so-called higher generalized simple intersection \cite[theorem 1.12]{g18}. 
 Therefore, for the proof of the inclusion $\La_{\ell,\frac12}\subset L_\ell$ it suffices 
 to show that $A(\mcp_\ell)<A_{\frac12}$ whenever $\omega$ is small enough. Let 
 us show that 
 \begin{equation} A(\mcp_\ell)=1+(\ell-1)\omega+o(\omega), \text{ as } \omega\to0.
 \label{mcpl}\end{equation}
 To do this, 
 let us recall the definition of the point $\mcp_\ell$.  Set 
 $$\mu:=\frac{A}{2\omega}, \ \  \la:=\frac1{4\omega^2}-\mu^2=\frac{1-A^2}{4\omega^2}.$$
 Consider the corresponding Heun equation (\ref{heun2}).  Fix an $\omega>0$. 
   The value $\mu(\mcp_\ell)=\frac{A(\mcp_{\ell})}{2\omega}$ is the maximal number 
 $\mu>0$ for which equation (\ref{heun2}) has a polynomial solution, see \cite[definition 1.9]{g18}, \cite[theorem 1.15]{bg2}. It was 
 shown in \cite{bt0} that existence of polynomial solution is equivalent to 
 the condition that the point $(\la,\mu)$ lies in a remarkable algebraic curve 
 $\Gamma_\ell\subset\rr^2$, the so-called {\it spectral curve.} Thus, for every 
  $\omega>0$ the point $(\la(\mcp_\ell),\mu(\mcp_\ell))$ lies in $\Gamma_\ell$, and it is the point 
 in $\Gamma_\ell$ with the biggest  coordinate $\mu$. 
 As $\omega\to0$, 
one has $\frac1{\omega}={\sqrt{4(\la+\mu^2)}}\to\infty$, thus, $(\la,\mu)\to\infty$. 
It is known that the complexified curve $\Gamma_\ell$ intersects the complex 
infinity line in $\cp^2$  at $\ell$ distinct regular real points. Their 
 asymptotic directions correspond to the ratios $\frac{\la}{\mu}$ equal to 
$\ell-1,\ell-3,\dots,-(\ell-1)$,  and the corresponding local branches are real. This was 
proved by I.V.Netay \cite[proposition 1.10]{gn19}. 
Therefore, as a point of the curve $\Gamma_\ell$ 
tends to its  infinite point, one has $\mu\to\infty$, 
$$\la=O(\mu)=o(\mu^2), \ \frac1{4\omega^2}=\la+\mu^2\simeq\mu^2, \ 2\omega\mu=A\simeq 1,$$
$$\frac{\la}{\mu}=\frac{1-A^2}{4\omega^2\mu}\simeq k, \ \ \ k\in\{ \ell-1,\ell-3,\dots,-(\ell-1)\}.$$
But $\frac{1-A^2}{4\omega^2\mu}=\frac{(1-A)(1+A)}{2\omega A}\simeq\frac{1-A}{\omega}$. 
The latter ratio should tend to a number $k$ as above. Therefore, as a point 
in $\Gamma_{\ell}$ tends to infinity, one of the following asymptotics takes place: 
$$A=1+m\omega+o(\omega), \  \ m=-k\in\{ \ell-1,\ell-3,\dots,-(\ell-1)\}.$$
The asymptotics corresponding to points with the maximal possible $A$ is given 
by $m=\ell-1$. This proves (\ref{mcpl}). Hence, $A(\mcp_\ell)<A_{\frac12}=1+(\ell-\frac12)\omega$, whenever 
$\omega$ is small enough, by (\ref{mcpl}). The inclusion $\La_{\ell,\frac12}\subset L_\ell$ 
 is proved. It implies that all the constrictions in $\La_{\ell,\frac12}$ are positive,  
 lie in  $L_\ell$, and hence, are not ghost. The lemma is proved.
 \end{proof}
 \subsection{Model of Josephson junction with small $\omega$ as slow-fast system}
 We study one-parameter subfamilies of vector fields (\ref{josvec}) on $\tt^2$ parametrized by small $\omega$ as 
 slow-fast families of dynamical systems, where $\ell=\frac B\omega\equiv const$  and $A$ depends on $\omega$. 
 To do this,   we recall the following results on topology of 
 the zero level curve of the $\theta$-component in (\ref{josvec}): the so-called slow curve
 $$\gamma=\gamma_{B,A}:=\{ f(\theta,\tau)=0\}, \  \  \ f(\theta,\tau):=\cos\theta+B+A\cos\tau.$$ 
 \begin{proposition} \label{pconv} (see \cite[proposition 2]{krs}). 
For every $(A,B)\in\rr^2_+$ with $|1-B|<A<1+B$ the curve $\gamma$
 is a regular strictly convex contractible curve lying in the interior of the 
fundamental square $[0,2\pi]^2$ of the torus $\tt^2$.  See Fig. 6a).
\end{proposition}
\begin{remark} The curve $\gamma$ is always symmetric with respect to the 
horizontal and vertical lines through the center of the latter square.
\end{remark}
 
For completeness of presentation we give the proof of Proposition \ref{pconv}.

\begin{proof} {\bf of Proposition \ref{pconv}.}  
Let $1-B<A<1+B$. Let us now show 
that the curve $\gamma$ does not 
intersect the boundary of the above fundamental square. 
Indeed, on the boundary either $\cos\theta=1$, or $\cos\tau=1$. 
If $\cos\theta=1$, then $f(\theta,\tau)=\cos\theta+B+A\cos\tau\geq1+B-A>0$. 
If $\cos\tau=1$, then $f(\theta,\tau)\geq-1+B+A>0$. 
Therefore, $f(\theta,\tau)\neq0$ on the boundary of the fundamental square, and $\gamma$ lies in its interior. 
For the proof of strict convexity it suffices to show 
that the
 value\footnote{The value of the Hessian form of a function $f$ on its skew gradient, 
i.e., the expression in the left-hand side in (\ref{hess+}) was introduced by S.Tabachnikov in \cite{tab08}.} 
  of the Hessian form 
  of the function $f$ on its skew gradient tangent to 
its level curves is positive on $\gamma$. 
That is, 
\begin{equation}\frac{\partial^2 f}{\partial\theta^2}\left(\frac{\partial f}{\partial\tau}\right)^2+
 \frac{\partial^2 f}{\partial\tau^2}\left(\frac{\partial f}{\partial\theta}\right)^2-2\frac{\partial^2 f}{\partial\theta\partial\tau}\left(\frac{\partial f}{\partial\tau}\right)\left(\frac{\partial f}{\partial\theta}\right)>0 \ \ 
 \text{ on } \  \gamma.\label{hess+}\end{equation}
 
Substituting  $u:=\cos\theta$, $v:=\cos\tau$ to the latter left-hand side and dividing 
it by $A$ yields the following equivalent inequality: 
 \begin{equation}-Au(1-v^2)-v(1-u^2)>0, \ \text{ whenever } u+B+Av=0  \text{ and } 
 |u|,|v|\leq1.\label{insf}\end{equation}
 Substituting $u=-B-Av$ to the left-hand side in (\ref{insf}) transforms it to the polynomial 
 $$P(v)=ABv^2+v(A^2+B^2-1)+AB.$$
  One has $P(v)>0$ for every $v\in\rr$, since its discriminant is negative, i.e., 
  $-2AB<A^2+B^2-1<2AB$. Indeed, the latter inequality can be rewritten as 
  $|A-B|<1<A+B$, which is equivalent to the system of inequalities of the proposition 
  for positive $A$ and $B$. The proposition is proved.
 \end{proof}
 \begin{proposition} \label{pnoc} In the case, when $A,B>0$ and 
  $A=1-B$, the curve $\gamma$ 
   is regular, except for one singular point $q=(\pi,0)$ of type 
  "transversal double  self-intersection". Its intersection with the interior of the 
  fundamental square $[0,2\pi]^2$ is a convex curve.
   In the case, when 
  $B>0$ and $0<A<1-B$, the curve $\gamma$ is regular and 
  consists of two non-contractible closed 
  connected components of homological type $(0,1)$ in the standard 
  basis in $H_1(\tt^2_{\theta,\tau})$. See Fig. 6b),c).
  \end{proposition}
  \begin{proof} Consider the first the case: $A=1-B$. Convexity is preserved under passing to limit, as 
  $A>1-B$ tends to $1-B$. The singular point 
  statement and uniqueness of singular point follow by straightforward 
  calculation, the Implicit Function Theorem and Morse Lemma. In more detail,  $\gamma$ being a level 
  curve of an analytic function $f(\theta,\tau)$, its singular points (if any) are the critical points of the function $f$ 
  contained in $\gamma$. The critical points  are those with $\cos\theta,\cos\tau=\pm1$. 
 The only critical point in $\gamma$ is the one with $\cos\theta=-1$, $\cos\tau=1$, i.e., $q=(\pi,0)$.  
 This is a Morse critical point  with index  $-1$, i.e., the Hessian form of the function 
 $f$ at $q$ has eigenvalues of opposite signs: $\frac{\partial^2f}{\partial\theta\partial\tau}=0$, 
 $\frac{\partial^2f}{\partial\theta^2}=1$, $\frac{\partial^2f}{\partial\tau^2}=-A<0$. Hence, 
 it is a  transversal  self-intersection singular point of the curve $\gamma$ (Morse Lemma). See Fig. 6b).
  
  As $A$ and $B$ vary, the topological type of the curve $\gamma$ may change only near those parameter values, 
  for which $\gamma$ is a critical level curve of the function $f(\theta,\tau)$. It follows from the above critical point 
  description that $\gamma$ is a critical level curve, if and only if $\pm1+B\pm A=0$ for some of the four possible 
  sign choices. Therefore, the topological type is constant in the domain $\{ B>0, \ 0<A<1-B\}$ in the parameter space. 
To find this topological type, fix a point 
$(B_0,A_0)\in\rr_+^2$ 
with $A_0=1-B_0$. We show that as $A>A_0$ decreases and crosses the value $A_0$,  
   a connected contractible curve $\gamma_{B_0,A}$  given by Proposition \ref{pconv} is transformed to 
   two disjoint curves isotopic to the $\tau$-circle. For $A$ close to $A_0$ the complement of each $\gamma_{B_0,A}$
    to a small disk $U$ centered at  the singular point $q$ is a regular curve  
   depending analytically on the parameter $A$. It consists of two connected components $\gamma_{B_0,A;\pm}(U)$ 
   disjoint from the circle $\{\theta=\pi\}$ and 
projected diffeomorphically to an interval $(\var,2\pi-\var)$ of the $\tau$-circle; here $\var=\var(U)$ is small. 
(The projection interval is the same for both components, since the symmetries $\theta\mapsto-\theta$, $\tau\mapsto-\tau$ 
preserve each curve $\gamma_{B,A}$.) The curves $\gamma_{B_0,A}\cap U$ 
    form a 
    singular
     foliation in $U$ by level curves of  the function $g(\theta,\tau):=
  -\frac1{\cos\tau}(\cos\theta+B_0)$ with critical value $A_0$ corresponding to a Morse critical point $q$ of index $1$. 
  The union of local branches of the singular curve $\gamma_{B_0,A_0}$ at $q$ 
   is invariant under the above symmetries,  and the local branches intersect transversally. Therefore, they are 
   transversal to the circles $\{\theta=\pi\}$, $\{\tau=0\}$.  
 For $A>A_0$ close to $A_0$ the curve $\gamma_{B_0,A}$ is strictly convex, and its intersection with 
  $U$ is a union of two connected components separated by the circle $\{\tau=0\}$, 
  by Proposition \ref{pconv}. This implies that for $A<A_0$ close to $A_0$ the local level curve 
  $\gamma_{B_0,A}\cap U$  consists of two components intersecting the circle $\{\tau=0\}$,  diffeomorphically 
  projected to an interval in the $\tau$-circle and disjoint from the circle $\{\theta=\pi\}$. 
  Adding the latter components to $\gamma_{B_0,A;\pm}(U)$ results 
  in two closed curves in $\tt^2$ disjoint from the circle $\{\theta=\pi\}$ and 
  projected diffeomorphically onto the $\tau$-circle. See Fig. 6c). Thus, they are isotopic to the $\tau$-circle. 
  This proves the last statement of Proposition 
\ref{pnoc}.
\end{proof}
\begin{figure}[ht]
  \begin{center}
   \epsfig{file=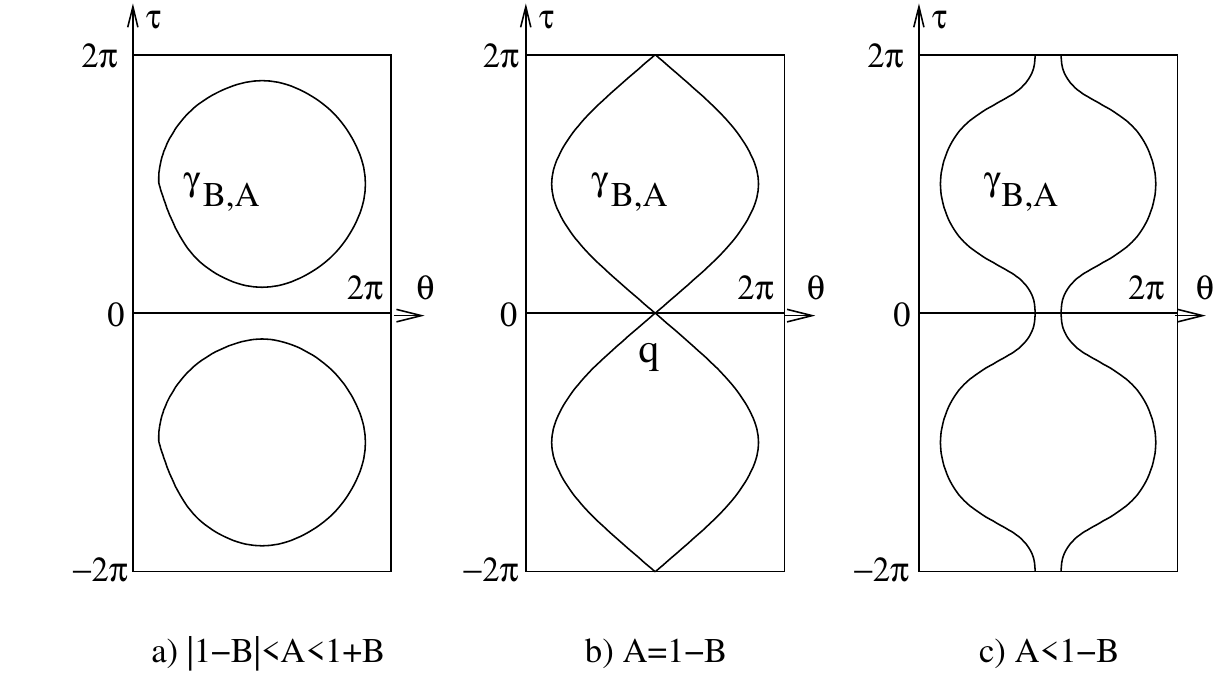}
    \caption{Different topological types of the  curve $\gamma_{B,A}=\{f(\theta,\tau)=0\}$ for $B,A>0$, $A<1+B$. 
    We present its liftings to 
     the universal covering $\rr^2_{\theta,\tau}$.}
  \end{center}
\end{figure} 

 Consider family (\ref{josvec}) 
 with  a fixed $\ell\in\nn$ and $\mu=\frac{A(\omega)}{2\omega}$ where
  \begin{equation}A(\omega)=A_{\alpha}(\omega)
  =1+(\ell-\alpha)\omega+o(\omega), \text{ as } \omega\to0; \ \alpha>0 
   \text{ is a  constant.}\label{asa}\end{equation}
  Multiplying family (\ref{josvec}) by $\omega$  
  yields a slow-fast family of dynamical systems 
 \begin{equation}\begin{cases} \dot\theta_t= f_{\alpha}(\theta,\tau;\omega)\\
 \dot \tau_t=\omega,\end{cases}  \ \ t=\omega^{-1}\tau, \  \ f_{\alpha}(\theta,\tau;\omega)=
\cos\theta + \ell\omega + A_{\alpha}(\omega) \cos \tau\label{slf}\end{equation}
on $\tt^2$ with $\omega\to0$.  
According to the commonly used terminology in the theory of slow-fast systems, see e.g. \cite{ilguk}, we will call the curve  
$$\gamma_\alpha(\omega):=\{ f_{\alpha}(\theta,\tau;\omega)=0\}\subset\tt^2=\rr^2_{(\theta,\tau)}\slash
2\pi\zz$$
 the {\it slow curve} of family (\ref{slf}). Propositions \ref{pconv} and \ref{pnoc} imply the following

 \begin{corollary} \label{cconv} 
For every fixed $\ell,\alpha\in\rr_+$ with $\alpha\neq2\ell$,  for every 
$\omega$ small enough dependently on $\ell$ and $\alpha$ 

(i) if $0<\alpha<2\ell$, then the slow curve of system (\ref{slf}) is convex, regular, contractible 
and lies in the interior of the fundamental square $[0,2\pi]^2$; 

(ii) if $\alpha>2\ell$, then  the slow curve  is regular 
and consists of two non-contractible closed connected components of 
homological type $(0,1)$.
\end{corollary}

 \begin{remark} \label{rtend} Fix an arbitrary $\alpha>0$. 
 As $\omega\to0$,  the slow curve  tends to the square with vertices 
 $(0,\pi)$, $(\pi,2\pi)$, $(2\pi,\pi)$, $(\pi,0)$, whose sides 
 lie in the lines $\{\theta+\tau=2\pi\pm\pi\}$, $\{\tau-\theta=\pm\pi\}$. 
 The corresponding vector fields converge to a 
 vector field with zero $\tau$-component and whose $\theta$-component 
 has simple zeros on the edges of the above square (with vertices deleted). 
 \end{remark}

  We deal with the liftings to the universal cover $\rr^2$ over $\tt^2$ 
   of vector fields (\ref{slf}) and their phase portraits. The lifted fields 
  will be denoted by the same symbol (\ref{slf}). 
 The slow curve $\gamma_\alpha=\gamma_\alpha(\omega)\subset\tt^2$ will be identified with its lifting $\gamma_\alpha^0$ 
 to the square $[0,2\pi]^2\subset\rr^2$. Its other lifting, obtained from the latter one 
 by translation by the vector $(2\pi,0)$ will be denoted by $\gamma^1_\alpha$. 
 \begin{definition} The {\it interior component} of the complement $\tt^2\setminus\gamma_{\alpha}$ is its connected 
 component containing the point $(\pi,\pi)$. Its liftings to  the squares $[0,2\pi]^2$ and $[2\pi,4\pi]\times[0,2\pi]$ will be 
 called the interior components of the complements  of the latter squares to the curves $\gamma^0_\alpha$ and $\gamma^1_\alpha$ respectively. 
 \end{definition}

 Fix  constants $h_0$, $h_1$, $h_2$ such that 
 $$\frac{3\pi}2<h_0<h_1<h_2<2\pi.$$
 For example, one can take, $h_0=\frac{6.5\pi}4$, $h_1=\frac{7\pi}4$, $h_2=\frac{15\pi}8$. 
  \begin{proposition} \label{pgeom} For every $\omega>0$ small enough the restriction of the function 
  $f_{\alpha}(\theta,\tau):=f_{\alpha}(\theta,\tau;\omega)$ 
   to the rectangle $[0,4\pi]\times[0,2\pi]$ is negative exactly in the 
 interior components of  complements of the curves $\gamma_\alpha^j$, 
 $j=0,1$, and positive outside the closure of the latter  components. 
 The strip 
 $$\Pi:=\{h_1\leq\tau\leq h_2\}$$
 intersects the curve $\gamma^1_\alpha$ by two disjoint 
 graphs (called  left and right) 
  $$L_{1,\alpha}:=\{\theta=\psi_1(\tau)\}, \ L_{2,\alpha}:=\{\theta=\psi_2(\tau)\}, \ \tau\in[h_1,h_2], \ \psi_1<\psi_2.$$ 
 The latter graphs converge 
uniformly in the $C^1$-norm to  segments parallel to the lines $\{\tau=\theta\}$ and $\{\tau=-\theta\}$ 
respectively, as $\omega\to0$. 
 \end{proposition}
 Proposition \ref{pgeom} follows from  Remark \ref{rtend}. 
 \begin{proposition} \label{pslf} Let $\alpha>0$. Let 
 $I_+\subset\rr^2$ denote the horizontal segment 
 connecting the points $(2\pi,h_0)$ and $(3\pi,h_0)$.  
  The intersection of the strip 
 $\Pi$ with the 
 orbit of the segment 
 $I_+$ by flow of vector  field (\ref{slf}) is  a flowbox denoted by 
 $$F_{\alpha,+}=F_{\alpha,+}(\omega).$$ 
It will be called a {\bf slow flowbox}. Its flow lines are uniformly 
$\omega$-close to $L_{1,\alpha}$ in the $C^1$-norm.  
The  intersections $F_{\alpha,+}\cap\{\tau=h\}$ with $h\in[h_1,h_2]$ are segments 
whose lengths are uniformly bounded (in $h$, $\omega$) by an exponentially 
small  quantity $\exp(-\frac c{\omega})$; $c>0$ is  independent on $h$ and 
$\omega$. See Fig. 7. 
\end{proposition} 
  \begin{proof} The proposition follows from Proposition \ref{pgeom} 
   and the classical theory of slow-fast systems. See, e.g., 
   \cite[theorem 3 and proposition 4]{ilguk}. 
 \end{proof}
 \begin{figure}[ht]
  \begin{center}
   \epsfig{file=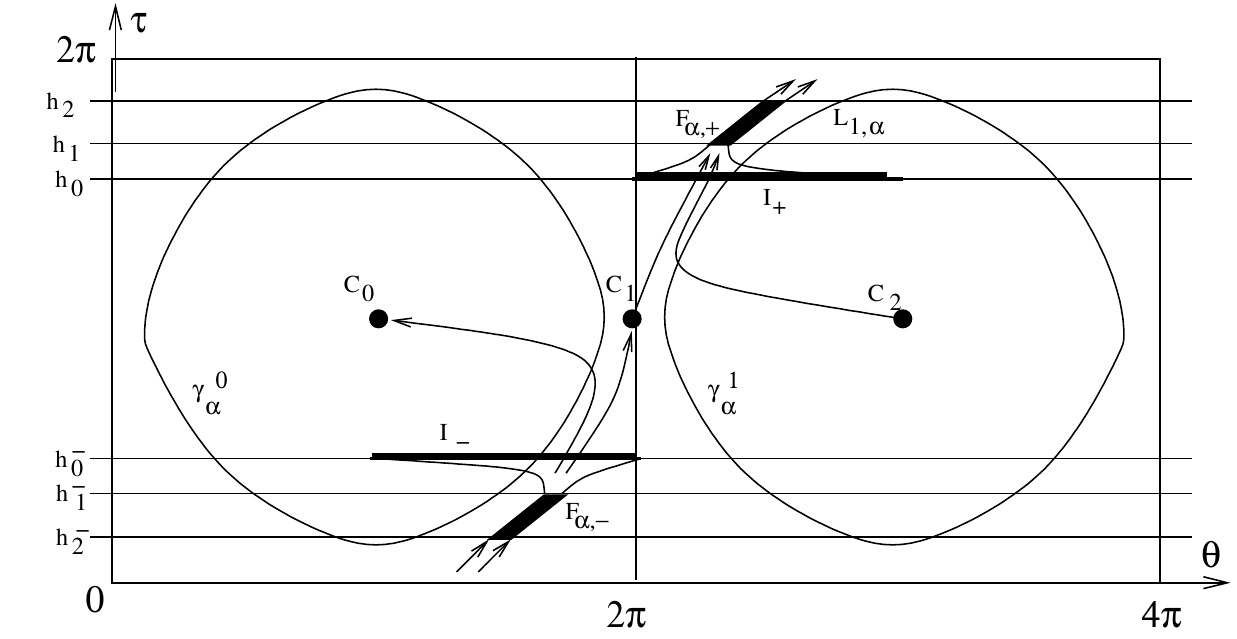}
    \caption{The slow flowboxes $F_{\alpha,\pm}$ (black) and orbits of points $C_j$.}
  \end{center}
\end{figure} 

\begin{remark} \label{remsym} The phase-portrait of vector field (\ref{slf}) is symmetric with 
respect to the points 
$$C_0:=(\pi,\pi), \ C_1:=(2\pi,\pi), \ C_2:=(3\pi,\pi);$$ 
the symmetry changes the sign (i.e., orientation) of the field. 
Let $I_-$ denote the horizontal segment symmetric to $I_+$ with respect to the point $C_1$, see Fig. 7. 
The above construction applied to the inverse vector field, the segment $I_-$ 
and the heights $h_j^-:=2\pi-h_j$ 
yields the slow flowbox 
$$F_{\alpha,-} \text{ symmetric to } F_{\alpha,+} \text{ with respect to the point } C_1.$$
\end{remark}
\subsection{The Comparison Lemma}
\begin{lemma} {\bf (Comparison Lemma).} Let $0<\alpha_1<\alpha_2$. 
 Consider two families (\ref{slf})$_j$, $j=1,2$, 
 of dynamical systems  (\ref{slf}) with $A=A_{\alpha_j}(\omega)$ satisfying (\ref{asa}). 
 For every $\omega>0$ small enough the corresponding 
 flowboxes $F_{\alpha_1+}$ and $F_{\alpha_2,+}$ are disjoint and 
 $F_{\alpha_2,+}$ lies on the left from the flowbox $F_{\alpha_1,+}$. Similarly, 
 the  flowboxes $F_{\alpha_1-}$ and $F_{\alpha_2,-}$ are disjoint and
 $F_{\alpha_2,-}$ lies on the right from the flowbox $F_{\alpha_1,-}$.
 \end{lemma}
 It suffices to prove the statement 
 of the lemma for the flowboxes $F_{\alpha_j,+}$, by symmetry (Remark \ref{remsym}). 
Here and below we use the next proposition. 
 \begin{proposition} \label{pangle} For every $\omega$ small enough 
 the following statements hold. The vectors of the 
 fields (\ref{slf})$_1$ and (\ref{slf})$_2$ form a positively oriented basis 
 at each point of the union of two strips 
 $$W:=\{0\leq\tau<\frac{\pi}2\}\cup\{\frac{3\pi}2<\tau\leq2\pi\}.$$
 At each point in the $\frac\omega8$-neighborhood of the flowbox $F_{\alpha_1,+}$ the  
 angles between the  vectors of the fields are 
  greater than $\sigma:=\arctan(2+b)-\arctan2$, 
  $b=\frac{\cos h_0}2(\alpha_2-\alpha_1)$. 
  The image of the flowbox  $F_{\alpha_1,+}$ under the unit time 
   flow map of the field (\ref{slf})$_2$ is disjoint from $F_{\alpha_1,+}$, and its intersection with 
   the strip $\Pi$ lies on  the left from $F_{\alpha_1,+}$. 
 \end{proposition}
 \begin{proof} The vectors of the fields (\ref{slf})$_1$ and (\ref{slf})$_2$ 
 have the same $\tau$-component equal to $\omega$. The difference of their 
 $\theta$-components is  $f_{\alpha_1}(\theta,\tau;\omega)-f_{\alpha_2}(\theta,\tau;\omega)=(\alpha_2-\alpha_1)\omega(1+o(1))\cos\tau>0$ on $W$, whenever $\omega$ is 
 small enough, since $\cos\tau>0$ on $W$. Therefore, the vectors 
 of the field (\ref{slf})$_2$ are directed to the left from the vectors of the 
 field (\ref{slf})$_1$ on $W$, that is,  
  the orientation statement of the proposition holds. For every $\omega$ small enough one has 
  \begin{equation} f_{\alpha_1}(\theta,\tau;\omega)-f_{\alpha_2}(\theta,\tau;\omega)>b\omega, \ 
   b:=\frac{\cos h_0}2(\alpha_2-\alpha_1), \text{ if } \tau\in[h_0,h_2],\label{star1}\end{equation}
   by the above asymptotics, and also 
   \begin{equation}\frac\omega4< f_{\alpha_1}(\theta,\tau;\omega)<2\omega\ \ \text{ on the } \frac\omega8-\text{neighborhood 
   of  }\ \ F_{\alpha_1,+}.\label{star2}\end{equation}
   Indeed, the flow lines of the field (\ref{slf})$_1$ in $F_{\alpha_1,+}$ $C^1$-converge to 
   the line 
   $\tau=\theta-\pi$ 
 (Propositions \ref{pgeom} and \ref{pslf}), hence $f_{\alpha_1}(\theta,\tau;\omega)\simeq\omega$ on $F_{\alpha_1,+}$. 
 This together with (\ref{asa}), (\ref{slf}) and the obvious inequality $|\cos'x|=|\sin x|\leq1$ implies (\ref{star2}). 
 The angle lower bound statement of Proposition \ref{pangle} follows from (\ref{star1}) and (\ref{star2}). Its last statement 
 on image of the flowbox $F_{\alpha_1,+}$ under the unit time flow map  of the field (\ref{slf})$_2$ 
 follows from the above angle bound and the fact that the vectors of the field (\ref{slf})$_2$ have length 
 no less than $\omega$, while the width of the flowbox $F_{\alpha_1,+}$ is exponentially small (the last statement 
 of Proposition \ref{pslf}). Proposition \ref{pangle} is proved.
   \end{proof}
   
  \begin{proof} {\bf of the Comparison Lemma.} 
  
  {\bf Claim.} {\it For every $\omega>0$ small enough for every $p\in W$ the  positive flow line of the field 
 (\ref{slf})$_2$ through $p$ in $W$  
  lies on the left from the corresponding flow line of the field (\ref{slf})$_1$.}
  
  The claim follows from   
  the orientation statement of Proposition \ref{pangle}.

 Fix an intermediate number $h'_1\in(h_0,h_1)$. 
 Consider  the flowbox $F_{\alpha_1,+}'$ constructed as in Proposition \ref{pslf}  with   
 $\Pi$ replaced by $\Pi':=\{h_1'\leq\tau\leq h_2\}$. One obviously has $\Pi\cap F_{\alpha_1,+}'=F_{\alpha_1,+}$. 
 The lengths of horizontal sections of the flowbox $F_{\alpha_1,+}'$ are uniformly 
 bounded by a quantity $\exp(-\frac{d}{\omega})$, with $d>0$ independent on $\omega$ (Proposition \ref{pslf}). 
 Take the lower  horizontal base of the flowbox $F_{\alpha_1,+}'$, which is a segment 
 in the line $\{\tau=h_1'\}$ with length bounded by the above exponent. Let 
 $q_1:=(\chi_1,h_1')$ denote its right boundary point, which lies in the (\ref{slf})$_1$-orbit of the  
 end $(3\pi,h_0)$ of the segment $I_+$. 
 
 Consider the analogous flowbox $F_{\alpha_2,+}'$ and point $q_2:=(\chi_2,h_1')$ 
 for the field (\ref{slf})$_2$. One has $\chi_2<\chi_1$, by the claim. 
 First suppose  
 that $q_2\notin F_{\alpha_1,+}'$. Then the lower base of the flowbox $F_{\alpha_2,+}'$ 
 is disjoint from the flowbox $F_{\alpha_1,+}'$ and lies on its left. 
 This together with the above claim implies that the flowboxes  are disjoint. 
 In the case, when $q_2\in F_{\alpha_1,+}'$, the image $q_2'$ of the point $q_2$ under 
 the  time 1 flow map of the field (\ref{slf})$_2$ would lie strictly to the left from the 
 flowbox $F_{\alpha_1,+}'$, by Proposition \ref{pangle}. Therefore, 
 the positive orbit of the point $q_2'$ also lies on its left, by the claim. 
 Note that 
 $$\tau(q_2')=\tau(q_2)+\omega=h_1'+\omega<h_1,$$ 
 whenever $\omega$ is small enough. Therefore, the above positive orbit intersects 
 the strip $\Pi=\{ h_1\leq\tau\leq h_2\}$ by an arc of curve going from its lower 
 base to its upper base and lying on the left from the flowbox $F_{\alpha_1,+}$. 
 The latter curve bounds $F_{\alpha_2,+}$ from the right, by construction. 
 Hence, $F_{\alpha_2,+}$ is disjoint from $F_{\alpha_1,+}$ and lies on its left.  
 The Comparison Lemma is proved.
 \end{proof}
\subsection{The Monotonicity Lemma}
Consider two families of vector fields (\ref{slf})$_{j}$, 
$j=1,2$ (treated as fields lifted to $\rr^2$),  as in the Comparison Lemma, corresponding to 
$\alpha_1>0$ and $\alpha_2>\alpha_1$. 
We study their {\it Poincar\'e maps} $P_j^{\tau_1,\tau_2}$: 
the time $\frac{\tau_2-\tau_1}\omega$ flow maps from the line $\{\tau=\tau_1\}$ to the 
line $\{\tau=\tau_2\}$ considered as functions of the coordinate $\theta$. 
For simplicity, we denote 
$$P_j(\theta):=P_j^{0,2\pi}(\theta).$$ 

\begin{lemma} \label{lemon} {\bf (Monotonicity Lemma)} 
For every $\omega>0$ small enough  
\begin{equation}P_2(\theta)<P_1(\theta) \text{ for every } \theta\in\rr.\label{monpoinc}
\end{equation}
\end{lemma} 
Lemma \ref{lemon} is proved below. In its proof we use the following proposition.

\begin{proposition} \label{inflow} 
 Let $C_0$, $C_1$, $C_2$, $h_k^-$, $F_{\alpha_j,-}$ be the same, 
 as in Remark \ref{remsym}.  The intersection of the positive 
orbit of the segment $[C_1,C_2]$ under the flow of the field (\ref{slf})$_j$ with 
the strip $\Pi=\{ h_1\leq\tau\leq h_2\}$ lies in the flowbox $F_{\alpha_j,+}$. The intersection of the negative orbit 
of the segment $[C_0,C_1]$ with the strip $\Pi_-:=\{ h_2^-\leq\tau\leq h_1^-\}$ lies in $F_{\alpha_j,-}$.  See Fig. 7.
\end{proposition}
\begin{proof} It suffices to prove the first statement  of the proposition, due to symmetry 
(Remark \ref{remsym}). The segment $I_+$ defining the flowbox $F_{\alpha_j,+}$ 
is horizontal and is obtained from the segment $[C_1,C_2]$ by vertical shift up. 
The shift length is fixed and equal to $h_0-\pi>0$.  Let $J_l$ and $J_r$ denote 
respectively the segment connecting $C_1$ ($C_2)$ to the left (respectively, 
right) endpoint of the segment $I_+$. One has $f_{\alpha_j}>0$ on $J_l$ and 
$f_{\alpha_j}<0$ on $J_r$, which follows from Remark \ref{rtend} and Proposition \ref{pgeom}. Thus, 
on the segment $J_l$ ($J_r$) the vectors 
of the field (\ref{slf})$_j$ are directed to the right (respectively, left). 
This implies that  the time $\frac{h_0-\pi}{\omega}$ flow map of the field 
sends the segment $[C_1,C_2]$ strictly inside the segment $I_+$. This together 
with the definition of the flowbox $F_{\alpha_j,\omega}$ implies the 
first statement of the proposition.
\end{proof}
\begin{figure}[ht]
  \begin{center}
   \epsfig{file=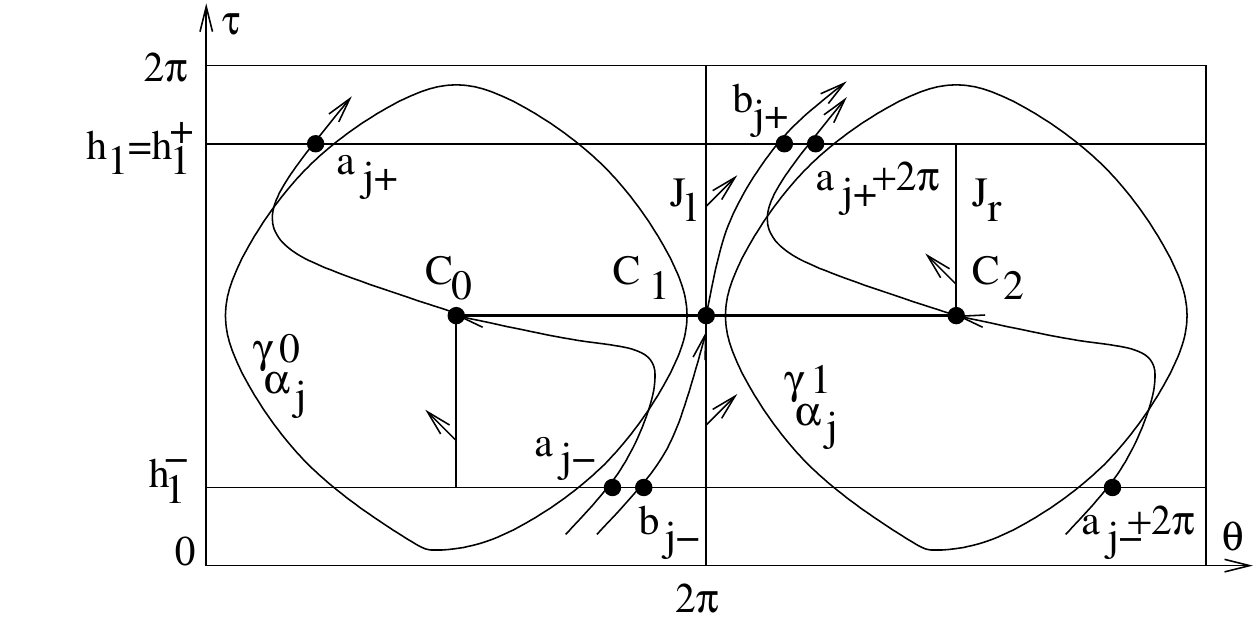}
    \caption{Orbits of segments $[C_0,C_1]$, $[C_1,C_2]$ and points $a_{j\pm}$, $b_{j\pm}$.}
  \end{center}
\end{figure}

\begin{proof} {\bf of the Monotonicity Lemma.} Set $h_j^+:=h_j$. One has 
\begin{equation}P_j=P_j^{h_1^+,2\pi}\circ \wt P_j\circ P^{0,h_1^-}_j, \ \ \wt P_j:=P_j^{h_1^-,h_1^+}.
\label{poincomp}\end{equation}
{\bf Claim 1.} {\it Whenever $\omega$ is small enough, 
one has $P^{0,h_1^-}_2(\theta)<P^{0,h_1^-}_1(\theta)$,   
$P_2^{h_1^+,2\pi}(\theta)<P_1^{h_1^+,2\pi}(\theta)$ for every $\theta\in\rr$.}

\begin{proof} Vector fields (\ref{slf})$_2$ and (\ref{slf})$_1$ have the same 
$\tau$-components. On the set $\{\tau\in [0, h_1^-]\cup[h_1^+,2\pi]\}$ 
the $\theta$-component of the former vector field is less than that of the latter, 
since $\cos\tau>0$ on this set.  This implies the  inequalities of the claim.
\end{proof}

Taking into account Claim 1 and (\ref{poincomp}), 
for the proof of the Monotonicity Lemma it suffices to prove the above  
inequality for the middle 
Poincar\'e maps in (\ref{poincomp}) for all $\omega$ small enough:
\begin{equation} \wt P_2(\theta)<\wt P_1(\theta) \text{ for every } 
\theta\in\rr.\label{middlep}\end{equation}
Consider the horizontal lines $L_{\pm}:=\{\tau=h_1^{\pm}\}$,   
which are the cross-sections for the Poincar\'e maps in question. We  identify each 
their point with its $\theta$-coordinate. For every $j=1,2$ let $b_{j\pm}$ denote 
the point of intersection of the line $L_{\pm}$ with the orbit of vector field 
(\ref{slf})$_j$  through the point $C_1=(2\pi,\pi)$. Let $a_{j\pm}$ denote the analogous 
intersection points with the orbit through the point $C_0=(\pi,\pi)$. See Fig. 8.

{\bf Claim 2.} {\it One has}
$$ a_{1-}<b_{1-}<a_{2-}<b_{2-}<a_{1-}+2\pi,$$
\begin{equation} 
a_{2+}<a_{1+}<b_{2+}<a_{2+}+2\pi<b_{1+}<a_{1+}+2\pi.\label{ineqp}\end{equation}

\begin{proof} The points $a_{j-}$ and $b_{j-}$ are the images of 
the points $C_0$ and $C_1$ respectively under  the Poincar\'e map 
$P_j^{\pi,h_1^-}$, and $\theta(C_0)<\theta(C_1)$, by definition.  Hence, 
$a_{j-}<b_{j-}$. The segment $[a_{j-},b_{j-}]$ lies in the flowbox $F_{\alpha_j,-}$, 
by Proposition \ref{inflow}. The flowbox $F_{\alpha_1,-}$ is disjoint from 
the flowbox $F_{\alpha_2,-}$ and lies on the left from it, by the Comparison Lemma. 
Therefore, the same is true for the corresponding segments $[a_{1-},b_{1-}]$ 
and $[a_{2-},b_{2-}]$. The four endpoints of the latter segments 
are $O(\omega)$-close to each other. Indeed the  flowboxes in question 
are $O(\omega)$-close to the right arcs  
of the corresponding intersections $\gamma_{\alpha_j}^0\cap\{h_2^-\leq\tau\leq h_1^-\}$ 
(Proposition \ref{pslf}). The 
latter arcs  are $\omega$-close, which follows from the Implicit Function 
Theorem for the equations defining the curves $\gamma_{\alpha_j}^0$. 
This together with the above discussion proves $O(\omega)$-closeness of the four 
points $a_{j-}$ and $b_{j-}$, $j=1,2$. This proves the first part of inequality (\ref{ineqp}). 
The proof of its second part is analogous. 
\end{proof}

\begin{proof} {\bf of  inequality (\ref{middlep}).} It suffices to prove it 
on the segment $K:=[a_{1-},a_{1-}+2\pi]\subset L_-$, by periodicity. The segment $K$ 
is split into 4 subsegments by points $a_{j-}$, $b_{j-}$. We check inequality (\ref{middlep}) 
on each splitting subsegment.

1) Let us start with the segment $[b_{2-},a_{1-}+2\pi]$. One has 
$$\wt P_1([b_{2-},a_{1-}+2\pi])\subset\wt P_1([b_{1-},a_{1-}+2\pi])=[b_{1+},a_{1+}+2\pi],$$
$$\wt P_2([b_{2-},a_{1-}+2\pi])\subset\wt P_2([b_{2-},a_{2-}+2\pi])=[b_{2+},a_{2+}+2\pi],$$
by (\ref{ineqp}). The latter segment-image  in the right-hand side is disjoint from the 
former one and lies on the left from it, by (\ref{ineqp}). This proves inequality (\ref{middlep}) 
on the segment $[b_{2-},a_{1-}+2\pi]$.  

2) Now we consider the segment $[a_{2-},b_{2-}]$. One has 
$$\wt P_1([a_{2-},b_{2-}])\subset\wt P_1([b_{1-},a_{1-}+2\pi])=[b_{1+},a_{1+}+2\pi],$$
$$\wt P_2([a_{2-},b_{2-}])=[a_{2+},b_{2+}], \  b_{2+}<b_{1+},$$
by (\ref{ineqp}). This proves inequality (\ref{middlep}) on  $[a_{2-},b_{2-}]$.  

3) The segment $[b_{1-},a_{2-}]$. One has 
$$\wt P_1([b_{1-},a_{2-}])\subset\wt P_1([b_{1-},a_{1-}+2\pi])=[b_{1+},a_{1+}+2\pi],$$
$$\wt P_2([b_{1-},a_{2-}]) \text{ lies on the left from the point } a_{2+}=\wt P_2(a_{2-})<b_{1+},$$
by (\ref{ineqp}). This proves inequality (\ref{middlep}) on  $[b_{1-},a_{2-}]$.  

4) The segment $[a_{1-},b_{1-}]$. One has 
$$\wt P_1([a_{1-},b_{1-}])=[a_{1+},b_{1+}],$$
$$\wt P_2([a_{1-},b_{1-}]) \text{ lies on the left from the point } a_{2+}=\wt P_2(a_{2-})<a_{1+},$$
since $b_{1-}<a_{2-}$, see (\ref{ineqp}). This proves inequality (\ref{middlep}) on  $[a_{1-},b_{1-}]$. Inequality 
(\ref{middlep}) is proved on all of the segment $K$, and hence, on the whole 
horizontal line $L_-$. 
\end{proof}

The statement of the Monotonicity Lemma follows from (\ref{poincomp}), 
Claim 1 and inequality (\ref{middlep}). 
\end{proof}

 \subsection{Absence of  constrictions with small ordinates} 
 Here we prove the following theorem and then Theorem \ref{noq}.
 \begin{theorem} \label{consm} 
 For every $\ell\in\nn$, $\beta>0$ and every $\omega>0$ small enough dependently 
 on $\ell$ and $\beta$ there are no constrictions $(B,A)$ with $B=\ell\omega$ and 
 $A\in[1-\ell\omega,1+(\ell-\beta)\omega]$.  
 \end{theorem}
 \begin{remark} Absence of constrictions with $B-1<A<B+1$, $B=\ell\omega$, 
 for small $\omega$ 
 was numerically observed in \cite[fig. 2, 3]{krs}. Theorem \ref{consm} confirms 
a part of  this experimental result  theoretically.
 \end{remark}
\begin{proof} {\bf of Theorem \ref{consm}.}  
  It suffices to prove the statement of the theorem for arbitrarily 
small $\beta$, e.g., $\beta<\frac12$. 
 Set $B=\ell\omega$. For every $\alpha>0$ set $A_{\alpha}=A_{\alpha}(\omega)=1+(\ell-\alpha)\omega$. 
 The family of systems 
(\ref{slf}) defined by this ordinate family $A_{\alpha}$ 
will be denoted  by (\ref{slf})$_{\alpha}$.

Suppose the contrary: there exists a sequence $\omega_k\to0$ such that 
there exists a sequence 
of constrictions $(B_k,A_{\alpha_k})$ with 
$$B_k=\ell\omega_k, \ \ 
A_{\alpha_k}=1+(\ell-\alpha_k)\omega_k, \ \ \beta\leq\lim\inf\alpha_k\leq\lim\sup\alpha_k\leq2\ell.$$ 
Passing to a subsequence, 
without loss of generality we can and will consider that  
$\alpha_k$ converge to some $\alpha^*\geq\beta$. Thus, the sequence 
of dynamical systems  corresponding to the above $(B_k,A_{\alpha_k})$ can be 
embedded into a continuous family of systems (\ref{slf}) with $\alpha$ 
replaced by $\alpha^*$. The latter new family of systems (\ref{slf}) 
will be denoted by (\ref{slf})$_{\alpha^*}$.

Fix an arbitrary $\alpha\in(0,\beta)$. 
Let $P$ and $P^*$ denote respectively the Poincar\'e maps $P^{0,2\pi}$ 
of the line $\{\tau=0\}$ to the line $\{\tau=2\pi\}$ defined by vector fields 
(\ref{slf})$_{\alpha}$ and (\ref{slf})$_{\alpha^*}$. 
For every $\omega$ small enough the point 
$(\ell\omega,1+(\ell-\alpha)\omega)$
 lies in 
the phase-lock area $L_\ell$, by Lemma \ref{quehigh} and since $\alpha<\beta<\frac12$. Therefore, the 
corresponding system (\ref{slf})$_{\alpha}$ has a periodic 
orbit with rotation number $\ell$. This means that there exists a point 
$a$ in the $\theta$-axis with $P(a)=a+2\pi\ell$. On the other hand, 
$$P^*<P, \ \ P^*(a)<P(a)=a+2\pi\ell, \text{ whenever } \omega \text{ is small enough,}$$
by the Monotonicity Lemma and since $\alpha^*\geq\beta>\alpha$. Therefore, 
the rotation number of system (\ref{slf})$_{\alpha^*}$ is no greater than $\ell$ and 
$a$ cannot be its periodic point with rotation number at least $\ell$ 
 for small $\omega$. In particular, the latter statements holds for the systems 
 corresponding to the above constrictions $(B_k,A_{\alpha_k})\in\La_\ell$. 
 On the other hand, the dynamical system (\ref{josvec}) corresponding to a constriction
lying in $\La_\ell$  should have rotation number at least $\ell$ 
 and all its orbits should be periodic with rotation number at least $\ell$, see 
\cite[theorem 1.2 and proposition 2.2]{4}. The contradiction thus obtained proves 
Theorem \ref{consm}.
\end{proof}
 
\begin{proof} {\bf of Theorem \ref{noq}.} Fix an $\ell\in\nn$. For every $\omega>0$ 
small enough all the constrictions lying in $\La_\ell$ with ordinates 
$A\geq A_{\frac12}=1+(\ell-\frac12)\omega$ are 
not ghost 
(Lemma \ref{quehigh}). There 
are no constrictions in $\La_\ell$ with  smaller 
positive ordinates (Theorem \ref{consm} and statement (\ref{nocons})). Theorem \ref{noq} is proved.
\end{proof}

 \subsection{Proof of Theorem \ref{noghost}}

 Let, to the contrary, there exist a ghost constriction $(B,A;\omega)$. 
 Then $\ell=\frac B{\omega}\in\zz\setminus\{0\}$, and without loss of generality we can and will consider 
 that $\ell\geq1$ (see the beginning of Section 5). Let $\mcc$ denote the 
connected component of the submanifold $Constr_\ell\subset(\rr_+^2)_{\mu,\eta}$ 
containing the corresponding point $(\frac A{2\omega},\omega^{-1})$. 
 The restriction to $\mcc$ of the function $\omega=\eta^{-1}$ is unbounded from   below, while 
 all the constrictions in $\mcc$ are ghost (Theorem \ref{famcons0}). Thus, there exist ghost constrictions 
 with given $\ell$ and arbitrarily small $\omega$. 
This yields a  contradiction to Theorem \ref{noq} and proves absence of ghost constrictions. The proof of 
 Theorem \ref{noghost}, and hence, Theorems \ref{thal} and \ref{thpos} 
 is complete.
 
 \section{Some applications and open problems}
 
 \subsection{Geometry of phase-lock areas}
 For every $\ell\in\zz_{\neq0}$ let 
 $\mcp_\ell=(\ell\omega,A(\mcp_\ell))\subset\rr^2_{B,A}$ 
 denote the higher 
 generalized simple intersection lying in $\La_\ell:=\{ B=\ell\omega\}$, see Subsection 5.1. 
Recall that 
$$S\ell:=\La_\ell\cap\{ A\geq A(\mcp_\ell)\}\subset L_\ell^+:=L_\ell\cap\{ A>0\}, \ \ \mcp_\ell\in\partial L_\ell^+.$$ 

 The {\bf Connectivity Conjecture}, see \cite[conjecture 1.14]{g18}, states that   {\it the intersection  
 $L_\ell^+\cap\La_\ell$ coincides with the ray $S\ell$, and thus, is connected.} 
 
 Theorem \ref{thpos} implies the following corollary 
 
 \begin{corollary} \label{cornp} Let, to the contrary to the above conjecture, the intersection 
 $L\La_\ell:=L_\ell^+\cap\La_\ell\cap\{0<A<A(\mcp_\ell)\}$ be non-empty. Then its lowest point 
(i.e., its point with minimal ordinate $A$) is a generalized simple intersection. 
 \end{corollary}
 \begin{proof} The  lowest point  $P\in L\La_\ell$ is well-defined, 
 has positive ordinate and lies in $\partial L_\ell$, since the growth point in $L_\ell$, 
 i.e., its intersection point with the abscissa axis, has abscissa 
 $\sqrt{\ell^2\omega^2+1}>\ell\omega$. Hence, it is either a constriction, or a 
 generalized simple intersection, by definition. If $P$ were a constriction, it would be negative, since 
 its lower  adjacent interval $\La_\ell\cap\{0<A<A(P)\}$ lies outside the phase-lock area $L_\ell$. 
 But there are no negative constrictions, by Theorem \ref{thpos}. Therefore, $P$ is a generalized 
 simple intersection. 
 \end{proof}
 
 \begin{remark} It is known that the generalized simple intersections $(\ell\omega, A)$ correspond to  the parameters 
 $(\la,\mu)$, $\mu=\frac A{2\omega}$, $\la=\frac1{4\omega^2}-\mu^2$, of those special 
 double confluent Heun equations (\ref{heun2}) that have polynomial solutions. The set of the latter parameters 
 $(\la,\mu)$ is a remarkable algebraic curve: the so-called {\it spectral curve} $\Gamma_\ell\subset\rr^2_{(\la,\mu)}$ introduced in \cite{bt0} and 
 studied in \cite{bt0, gn19}. It is the zero locus of the polynomial from \cite[formula (21)]{bt0}, which is the determinant of a three-diagonal matrix 
 formed by diagonal terms of type $\la+const$ and linear functions in $\mu$ at off-diagonal places.  
 See also \cite[formula (1.4)]{gn19}. (The complexification of the spectral curve is known to be irreducible, see 
  \cite[theorem 1.3]{gn19}.)  For every given $\omega>0$ the curve $\Gamma_\ell$ contains at most $\ell$ 
 points $(\la,\mu)$   corresponding to the given $\omega$ with $\mu>0$;  the point with the biggest $\mu$ corresponds to the higher  generalized simple intersection $\mcp_\ell$. This follows from B\'ezout Theorem and the fact that the spectral curve $\Gamma_\ell$ is the zero locus 
  of a polynomial of degree $\ell$ in $(\la,\mu^2)$, see \cite[p. 937]{bt0}.
 \end{remark}
 Corollary \ref{cornp} and the above remark reduce the Connectivity Conjecture to the following equivalent, 
  algebro-geometric conjecture.

 \begin{conjecture} \label{connconj2} For every $\omega>0$ the above real spectral curve $\Gamma_\ell$ contains a unique 
 point $(\la,\mu)$  with $\la=\frac1{4\omega^2}-\mu^2$ (up to change of sign at $\mu$) for which the corresponding rotation number 
 $\rho=\rho(\ell\omega, 2\mu\omega)$ equals $\ell$. (The  point $(B,A)=(\ell\omega,2\mu\omega)$ 
 coincides with  $\mcp_\ell$, see the above remark.)
 \end{conjecture}
 \begin{theorem} For every $\ell\in\zz_{\neq0}$ and every positive $\omega<\frac1{|\ell|}$  
 the Connectivity Conjecture holds.
 \end{theorem}
 \begin{proof} Let, say, $\ell>0$, and let $0<\omega<\frac1\ell$. 
 Then for every $r\in\nn$, $0<r<\ell$, the boundary $\partial L_r$ intersects $\La_\ell^+:=\La_\ell\cap\{ A>0\}$ 
 in at least two points. 
 Indeed, 
 the abscissa 
 $\sqrt{r^2\omega^2+1}$ of the growth point of the phase-lock area $L_r$ is greater than $\ell\omega<1$. 
 On the other hand, each boundary curve of the area $L_r$
  contains constrictions, which lie in the axis $\La_r$, and hence, 
 on the left from the axis $\La_\ell$.  Hence, 
 each boundary curve intersects $\La_\ell^+$ in at least one point (this statement is given by \cite[theorem 1.18]{gn19} 
 for all $\omega$ small enough). It cannot be a common 
 intersection point for both boundary curves, i.e., it cannot be a constriction, since $r=\rho<\ell$ and by Theorem \ref{thal}. 
 Therefore, the intersection $\partial L_r\cap\La_\ell^+$ contains at least two distinct points.
  Analogously, $\partial L_0$ intersects $\La_\ell^+$ in at least one point, since the point $(1,0)\in\partial L_0$ 
  lies on the right from the point $(\ell\omega,0)\in\La_\ell$. If $0\leq r<\ell$ and $r\equiv\ell(\operatorname{mod}2)$, then 
  each point of intersection $\partial L_r\cap\La_\ell^+$ is a generalized simple intersection. Taking 
  these intersections for all latter $r$  yields  $\ell-1$ distinct generalized simple intersections lying in $\La_\ell^+$. 
  But the total number of generalized simple intersections in $\La_\ell^+$ is no greater than $\ell$, see the above remark. 
   Therefore, 
  at most one of them may correspond to the rotation number $\ell$, and hence, is reduced to the known 
  generalized simple intersection $\mcp_\ell$ with $\rho=\ell$. In particular, there are no generalized simple intersections in 
  $\La_\ell\cap L_\ell$
   with $0<A<A(\mcp_\ell)$. This together with Corollary \ref{cornp} implies that $L_\ell^+\cap\La_\ell=S\ell$ 
  and proves the Connectivity Conjecture for $0<\omega<\frac1{|\ell|}$.
  \end{proof} 
  
\begin{problem} \label{omto0} \cite[subsection 5.8]{bg2} What is the asymptotic behavior of the phase-lock area portrait in family 
(\ref{josvec}),
 as 
$\omega\to0$?
\end{problem}
This problem is known and motivated by physics applications. 
V.M.Buchstaber, S.I.Tertychnyi and later by D.A.Filimonov, V.A.Kleptsyn, I.V.Schurov performed 
numerical experiences studying limit behavior of the phase-lock area portrait  
after appropriate rescaling of the variables $(B,A)$. Their experiences have shown that 
 the interiors of the phase-lock areas tend to open subsets (the 
so-called {\it limit rescaled phase-lock areas}) whose connected components form a partition of the plane. In 
some planar region, the latter partition looks like a chess table turned by $\frac{\pi}4$. It would be interesting 
to prove this mathematically and to find the boundaries of the limit phase-lock areas. 

Some results  on smallness of gaps between rescaled phase-lock areas for small $\omega$ 
were obtained in \cite{krs}.

To our opinion, methods elaborated in \cite{krs} and in the present paper  could be applied to study Problem \ref{omto0}.

\subsection{The dynamical isomonodromic foliation}

Let us consider family (\ref{josvec}) modeling overdamped Josephson junction 
as a three-dimensional family, with variable frequency $\omega$. Its {\it three-dimensional 
 phase-lock areas} in $\rr^3_{B,A,\omega}$ are defined in the same way, as in Definition \ref{defasl}. 
 Each three-dimensional phase-lock 
 area is fibered by two-dimensional phase-lock areas in $\rr^2_{B,A}$  corresponding to different  fixed values 
 of $\omega$.

Linear systems (\ref{tty}) corresponding to (\ref{josvec}) form a transversal hypersurface to the isomonodromic foliation of the 4-dimensional manifold $\bjr$ (Lemma \ref{lcross}).
It appears that there is another four-dimensional manifold with the latter property that has the following advantage: 
it consists of linear systems on $\oc$ coming from a family of dynamical systems on 2-torus. Namely, 
consider the following four-dimensional family of  dynamical systems on $\tt^2$ containing (\ref{josvec}): 
\begin{equation}  \frac{d\theta}{d\tau}=\nu+a\cos\theta+s\cos\tau+\psi\cos(\theta-\tau); \ \ \ 
\nu,a,\psi\in\rr, \ \ s>0, \ (a,\psi)\neq(0,0).\label{gen3}\end{equation}
The variable changes 
$\Phi=e^{i\theta}$, $z=e^{i\tau}$ 
transform (\ref{gen3}) to the Riccati equation
$$\frac{d\Phi}{dz}=\frac1{z^2}\left(\frac s2\Phi+\frac\psi2\Phi^2\right)+\frac1z\left(\nu\Phi+\frac a2(\Phi^2+1)\right) 
+\left(\frac s2\Phi+\frac\psi2\right).$$
A function $\Phi(z)$ is a solution of the latter Riccati equation, if and only if 
$\Phi(z)=\frac{Y_2(z)}{Y_1(z)}$, where $Y=(Y_1,Y_2)(z)$ is a solution of the linear system 
\begin{equation} Y'=\left(-s\frac{\mathbf K}{z^2}+\frac{\mathbf R}z+s\mathbf N\right)Y,\label{mchoy}\end{equation}
$$\mathbf K=\left(\begin{matrix}\frac12 & \chi \\ 0 & 0\end{matrix}\right), \ 
  \mathbf R=\left(\begin{matrix}-b & -\frac{a}2\\ \frac{a}2 & \chi a\end{matrix}\right),  \ 
  \mathbf N=\left(\begin{matrix}-\frac12 & 0 \\ \chi  & 0\end{matrix}\right);
  $$
  $$\chi=\frac{\psi}{2s}, \ b= \nu-\frac{\psi}{2s}a=\nu-\chi a.$$
  The residue matrix of the formal normal forms of system (\ref{mchoy}) at $0$ and at $\infty$ is the same and equal to 
  \begin{equation}\diag(-\ell,0), \ \ \ell:=b-\chi a=\nu-2\chi a=\nu-\frac{\psi a}{s}.\label{ellll}\end{equation}
  \begin{theorem} \label{tisom} The four-dimensional family of linear systems (\ref{mchoy}) is analytically foliated by 
  one-dimensional isomonodromic families defined by the following non-autonomous system of differential equations:
  \begin{equation}\begin{cases} \chi'_s=\frac{a-2\chi(\ell+2\chi a)}{2s}\\
  a'_s=-2s\chi+\frac as(\ell+2\chi a)\\
  \ell'_s=0\end{cases}.\label{isomnews}\end{equation} 
  For every $\ell\in\rr$ the function 
  \begin{equation}w(s):=\frac{a(s)}{2s\chi(s)}=\frac{a(s)}{\psi(s)}\label{wsnew}\end{equation}
  satisfies Painlev\'e 3 equation (\ref{p3}) along solutions of (\ref{isomnews}). 
  \end{theorem}
  \begin{proof} The composition of variable rescalings $z=s^{-1}\zeta$ and gauge transformations  
  \begin{equation}Y=\left(\begin{matrix} 1 & 0\\ -2\chi & 1\end{matrix}\right)\wt Y\label{gaugenew}\end{equation}
  sends family  (\ref{mchoy}) to the following family of linear systems: 
    \begin{equation}Y'_\zeta=\left(-\frac t{\zeta^2}K+\frac R{\zeta}+\diag(-\frac12,0)\right)Y,\label{mchoy2}\end{equation}
  $$t=s^2, \ K=\left(\begin{matrix} \frac12-2\chi^2 & \chi\\ \chi(1-4\chi^2) & 2\chi^2\end{matrix}\right), \ 
  R=\left(\begin{matrix} -\ell & -\frac a2\\ -2\chi(\ell+\chi a)+\frac a2 & 0\end{matrix}\right)$$
  Let  $\mathcal J$ denote the space of 
  special Jimbo type systems, see (\ref{djimbo}), (\ref{maa}), 
   with real matrices. Systems (\ref{mchoy2}) lie in 
  $\mathcal J$, since the formal normal forms of a system (\ref{mchoy}) 
  at $0$, $\infty$ have common residue matrix $\diag(-\ell,0)$. Every system $\mathcal L_0$ 
  of type  (\ref{mchoy2})  with $\chi\neq0$ has a neighborhood $W=W(\mcl)\subset\mathcal J$ where family (\ref{mchoy2}) 
  forms a hypersurface  
  $\mcx=\{ K_{22}=2K_{12}^2\}\cap W$ 
  so that each system $\mcl\in W$ can be projected to a system 
  $\mcl^*\in\mcx$  
  by  a diagonal gauge transformation $(Y_1,Y_2)\mapsto(Y_1,\la Y_2)$, $\la=\la(\mcl)$. 
  Projecting to $\mcx$ a Jimbo isomonodromic family given by (\ref{isomatr}) yields an isomonodromic family of systems (\ref{mchoy2}). 
  The differential equation satisfies by the projected isomonodromic families is found analogously to the proof of equation (\ref{isonormal}). 
  To do this, fix a $t_0>0$ and matrices $K(t_0)$, $R(t_0)$ as in (\ref{mchoy2}). 
  Let  $\wt K(t)$, $\wt R(t)$ be solutions of (\ref{isomatr}) with initial conditions $K(t_0)$ and $R(t_0)$ at $t_0$, and 
  let $(Y_1,Y_2)\mapsto(Y_1,\la(t)Y_2)$ be the family of the above normalizing  gauge transformations: 
  the matrices $K(t)=\diag(1,\la(t))\wt K(t)\diag(1,\la^{-1}(t))$, $R(t)=\diag(1,\la(t))\wt R(t)\diag(1,\la^{-1}(t))$ are the same, 
  as in (\ref{mchoy2}), that is $K_{22}(t)=2K_{12}^2(t)$; $\la(t_0)=1$. Set $\xi=\la'(t_0)$. 
  The equation on the matrix  function $\wt K(t)$ given by (\ref{isomatr}) yields 
   \begin{equation}K_{12}'(t_0)=-\xi K_{12}(t_0)+\frac1{t_0}[R(t_0),K(t_0)]_{12}, \ K_{22}'(t_0)=\frac1{t_0}[R(t_0),
  K(t_0)]_{22},\label{kr12t}\end{equation}
  Substituting $K_{22}=2\chi^2$, $K_{12}=\chi$ to the second equation in (\ref{kr12t}) and changing the time 
  parameter $t$ to $s=\sqrt t$ yields formula for the derivative $\chi'_t=(K_{12})'_t$ and 
  the first equation in (\ref{isomnews}). Substituting thus found 
  derivative $K_{12}'$ to the first formula in (\ref{kr12t}) yields a linear equation on $\xi$, 
  whose solution is $\xi=-\frac{\ell+2\chi a}{2t_0}$.  The differential equation on the matrix $\wt R(t)$ in (\ref{isomatr}) 
 yields the differential equation on $R_{12}$ analogous to the first equation in (\ref{kr12t}), which  
 also includes the above already found value $\xi$. Substituting $R_{12}=-\frac a2$ there yields 
 the second equation in (\ref{isomnews}). 
 Painlev\'e 3  equation (\ref{p3}) 
  on  $w(s)$ along isomonodromic families thus constructed follows from Theorem \ref{tjp},  
 since  diagonal gauge transformations do not change the ratio $\frac{R_{12}}{K_{12}}$. Equation (\ref{p3}) 
 can be also deduced directly from (\ref{isomnews}).
  \end{proof} 
  
  The  foliation from Theorem \ref{tisom} given by (\ref{isomnews})   induces a one-dimensional 
  foliation  in the 4-dimensional space of dynamical 
  systems (\ref{gen3}) given by the following non-autonomous system of 
  equations obtained from (\ref{isomnews}) by change of the variable $\chi$ to $\psi=2s\chi$: 
  \begin{equation}\begin{cases}\psi'_s=a+(1-\ell)\frac{\psi}s-\frac{a\psi^2}{s^2}\\
  a'_s=-\psi+\ell\frac as+\frac{\psi a^2}{s^2}.\end{cases}\label{isomn2}\end{equation}
 The latter foliation of family (\ref{gen3})  given by (\ref{isomn2}) will be denoted by $\mcg$ and called the
  {\it dynamical isomonodromic 
  foliation.} 
    \begin{lemma}  The conjugacy class of flow (\ref{gen3}) under diffeomorphisms $\tt^2\to\tt^2$ isotopic to identity, 
    its rotation number  and  $\ell$, see (\ref{ellll}), are constant on  leaves of the dynamical isomonodromic foliation $\mcg$. The hypersurface 
of  systems (\ref{josvec}) modeling Josephson junction is transversal to $\mathcal G$. The function 
$w(s)=\frac{a(s)}{\psi(s)}$, see (\ref{wsnew}), 
satisfies Painlev\'e 3 equation (\ref{p3}) along its leaves. A  point $(s,\psi,a,\ell)$ 
corresponds to a system (\ref{josvec}), if and only if $\psi=0$; this holds if and only if the function  $w$ 
has pole of order 1 at $s$ with residue 1.  
\end{lemma}
\begin{proof} The projectivized monodromy of linear system (\ref{mchoy}) is the complexification of the Poincar\'e map 
of the corresponding dynamical system (\ref{gen3}). Therefore,  constance of its conjugacy class along leaves  implies 
constance of  conjugacy class of the Poincar\'e map and hence, of the flow and of its rotation number; 
$\ell=const$, by Theorem \ref{tisom}. Family of systems (\ref{josvec}) coincides with the hypersurface 
$\{ \psi=0\}\cap\{ a>0\}$ in the parameter space. It is  transversal to the vector field (\ref{isomn2}), since 
$\psi'=a>0$ at all its points. The  characterization of 
systems (\ref{josvec}) in terms of poles  follows from construction and Lemma \ref{lempol} and, 
on the other hand,  immediately from (\ref{isomn2}): if  $\psi(s_0)=0$, then $\psi(s)\simeq a(s_0)(s-s_0)$, 
$w(s)=\frac{a(s)}{\psi(s)}\simeq \frac1{s-s_0}$, 
as $s\to s_0$, and vice versa. 
\end{proof}

Recall that the growth point of a two-dimensional phase-lock area $L_r$ has abscissa 
$\sign(r)\sqrt{r^2\omega^2+1}$. For a given $r\in\zz_{\neq0}$ the latter growth points form a curve bijectively 
parametrized by $\omega>0$ in the three-dimensional parameter space, which will be called  the {\it $r$-th growth curve.} 
We already know that the family of constrictions in $\rr^2_{B,A}\times(\rr_+)_{\omega}$ is a one-dimensional 
submanifold, by Theorem \ref{famcons0}. Thus, it is a disjoint union of connected curves, which will be called the {\it constriction curves.} 

In what follows  the  three-dimensional phase-lock area in $\rr^2_{B,A}\times(\rr_+)_{\omega}$ with a rotation 
number $r\in\zz$ will be denoted by $\wh L_r$.

\begin{conjecture}\label{conj3d} Each constriction curve is bijectively projected onto $(\rr_+)_{\omega}$. 
Each three-dimensional phase-lock area  of family (\ref{josvec}) is a countable garland of 
domains, where any two adjacent domains are separated either by the corresponding growth curve, or by a 
constriction curve.
\end{conjecture}

\begin{remark} Conjecture \ref{conj3d} {\bf does not} follow from known results on two-dimensional phase-lock area. 
A priori, a constriction curve may be not bijectively projected to the $\omega$-axis, and the projection 
may have some critical value $\omega_0$ that is a local maximum (minimum). In this case  the corresponding two-dimensional phase-lock area with $\omega$ less (greater) than $\omega_0$ has two constrictions 
that collide for $\omega=\omega_0$ and disappear, when $\omega$ crosses the critical value $\omega_0$. 
Numerical experiences made by S.I.Tertychnyi, D.A.Filimonov, V.A.Kleptsyn, I.V.Schurov show that such a 
scenario does not arise. This can be viewed as a numerical confirmation of Conjecture \ref{conj3d}. 
\end{remark}
\begin{remark} In each two-dimensional phase-lock area $L_r$ the constrictions lying in the half-plane $\{ A>0\}$ 
are ordered by natural numbers $k$ 
corresponding to their heights: the lowest constriction is ordered by 1, the second one by two, etc. If Conjecture 
\ref{conj3d} is true, then along each constriction curve $\mcc$ lying in the upper quarter-space $\{ A>0\}$ the above 
height number is constant. In this case  each constriction curve $\mcc=\mcc_{\ell,k}\subset\{ A>0\}$ is numerated 
by two integer numbers 
 $$\ell=\rho=\frac B{\omega}, \ \ k:=\text{ the above height number.}$$
\end{remark}
Studying of the following two problems, which are of independent interest, 
would have important applications to Conjecture \ref{conj3d} and related problems. 

\begin{problem} \label{pb2} Study the Poincar\'e map of the dynamical isomonodromic 
foliation $\mathcal G$, see (\ref{isomn2}), acting on the transversal hypersurface given by family of 
systems (\ref{josvec}). The Poincar\'e map sends the intersection of its definition domain 
with each three-dimensional phase-lock area in family (\ref{josvec}) to the same phase-lock area, 
by constance of the rotation number along leaves. {\it Study the action of the Poincar\'e map of the foliation $\mcg$ 
given by (\ref{isomn2}) 
on  the three-dimensional phase-lock area portrait of family (\ref{josvec}).} 
\end{problem}

\begin{problem} Is it true that the above Poincar\'e map  
is well-defined on each constriction curve $\mcc_{\ell,k}$  and sends its diffeomorphically onto $\mcc_{\ell,k+1}$? 
\end{problem}
If Conjecture \ref{conj3d} is true, then for every $\ell\in\zz$ and $k\in\nn$ there is a unique connected component  
$\mathcal O_{\ell,k}$ of the interior of the three-dimensional phase-lock area $\wh L_\ell$ that is 
adjacent to the constriction curves $\mcc_{\ell,k}$ and $\mcc_{\ell,k+1}$. 
\begin{problem}
Is it true that the Poincar\'e map of the foliation $\mathcal G$ 
  is well-defined on each component $\mathcal O_{\ell,k}$ and sends it diffeomorphically 
onto $\mathcal O_{\ell,k+1}$? Is it  a well-defined  diffeomorphism on  a neighborhood of the 
closure $\overline{\mathcal O_{\ell,k}}$?
What is the intersection of its definition domain with the component of the phase-lock area $\wh L_\ell$ 
adjacent to $\mathcal O_{\ell,1}$ and to the corresponding growth curve? How does it act there?
\end{problem}

\begin{remark} The  Poincar\'e map of the foliation $\mathcal G$ (where it is defined) can be viewed as the suspension over the map sending 
a given simple pole $s_0>0$ with residue 1  of solution $w(s)$ of Painlev\'e 3 equation (\ref{p3}) to its next pole $s_1>s_0$ 
of the same type (if any). Many solutions of (\ref{p3}) have an infinite lattice of simple poles 
with residue 1 converging to $+\infty$. Our Painlev\'e 3 equations (\ref{p3}) admit a one-dimensional family of 
Bessel type solutions, see \cite{conte},  whose poles are zeros 
of  solutions of Bessel equation and are known to form an infinite lattice. 
Victor Novokshenov's recent numerical experience has shown 
that their small deformations also have an infinite lattice of poles. 
Few  solutions,  e.g., the  {\it tronqu\'ee solutions} \cite{lidati},  are bounded on some semi-interval $[C,+\infty)$, and hence, 
do not have poles there.  
\end{remark}
\begin{problem} \label{pb3} Describe those parameter values of family (\ref{josvec}) 
for which the corresponding solution $w(s)$ of (\ref{p3}) is tronqu\'ee.  Is it true that this holds 
for some special points of boundaries of the three-dimensional phase-lock areas?{\it Does this hold for the 
higher generalized simple intersections $\mcp_\ell$} discussed in  Subsection 6.1?
\end{problem}

\begin{problem} \label{pb4} Study geometry of phase-lock areas\footnote{Recently it was observed by 
V.M.Buchstaber and the second author (A.A.Glutsyuk) that the rotation number quantization effect holds 
in family (\ref{gen3}): phase-lock areas exist only for integer values of the rotation number. 
The proof is the same, as in \cite{buch2}.} in four-dimensional family (\ref{gen3}) of dynamical 
systems on $\tt^2$. Study  special points of boundaries of the phase-lock areas: analogues of growth points, constrictions and generalized simple 
intersections. 
\end{problem}

 Let $\Sigma$ denote the subfamily in (\ref{gen3}) consisting of dynamical systems with 
trivial Poincar\'e map. The value $\ell=\nu-\frac{\psi a}{s}$ 
corresponding to a system in $\Sigma$ 
 should be integer, as in Proposition \ref{protriv}, 
and  its rotation number $\rho$ is also integer. For every $\ell,\rho\in\zz$ let  $\Sigma_{\ell,\rho}\subset\Sigma$ 
denote the subset consisting of systems 
with given $\ell$ and $\rho$.  Those systems (\ref{josvec}) with given $\ell$ that correspond to constrictions 
are contained in $\Sigma_{\ell,\ell}$, by Theorem \ref{thal}. 
\begin{problem} \label{pb5} Is it true that systems  (\ref{josvec}) with given $\ell$ corresponding to constrictions 
lie in one connected component of the set $\Sigma_{\ell,\ell}$?
\end{problem}
\begin{remark} 
One can show that a 
 positive solution of Conjecture \ref{conj3d} would imply positive answer to Problem \ref{pb5}.
\end{remark}
To our opinion, a progress in studying the above problems would have applications to 
problems on geometry of phase-lock areas, for example, to problems discussed in the previous subsection.

Studying Conjectures \ref{connconj2}, \ref{conj3d} and  Problems \ref{pb2}, \ref{pb4}, \ref{pb5} is a work in progress.

 \section{Acknowledgements}
We are grateful to  V.M.Buchstaber and Yu.S.Ilyashenko for attracting 
our attention to problems on 
model of Josephson effect and helpful discussions. We are grateful to 
them and to M.Bertola, D.A.Filimonov, E.Ghys, V.I.Gromak, M.Mazzocco, V.Yu.Novokshenov, V.N.Roubtsov, I.V.Schurov, S.I.Tertychnyi, 
I.V.Vyugin for helpful discussions. We are grateful to S.I.Tertychnyi for the figure of different types of constrictions, for 
careful reading the paper and helpful remarks. 
We are grateful to the referees 
and the editors 
for  careful  reading the paper and helpful remarks and suggestions.

\end{document}